\documentclass[a4paper]{amsart}

\usepackage{etex}
\usepackage{amssymb,amsfonts,amsmath,amsthm}
\usepackage{comment}
\usepackage[all,arc]{xy}
\usepackage{enumerate}
\usepackage{mathrsfs,mathtools}
\usepackage{todonotes,booktabs}
\usepackage{xcolor} 
\usepackage{graphicx}

\usepackage{tikz}
\usepackage{tikz-cd}
\usetikzlibrary{trees}
\usetikzlibrary[shapes]
\usetikzlibrary[arrows]
\usetikzlibrary{patterns}
\usetikzlibrary{fadings}
\usetikzlibrary{backgrounds}
\usetikzlibrary{decorations.pathreplacing}
\usetikzlibrary{decorations.pathmorphing}
\usetikzlibrary{positioning}

\usepackage[pagebackref]{hyperref}
\definecolor{dark-red}{rgb}{0.5,0.15,0.15}
\definecolor{dark-blue}{rgb}{0.15,0.15,0.6}
\definecolor{dark-green}{rgb}{0.15,0.6,0.15}
\hypersetup{
    colorlinks, linkcolor=dark-red,
    citecolor=dark-blue, urlcolor=dark-green
}

\newcommand{\iHom}{\underline{\operatorname{Hom}}}

\newcommand{\iExt}{\underline{\operatorname{Hom}}}

\renewcommand*{\backref}[1]{}
\renewcommand*{\backrefalt}[4]{%
  \ifcase #1 %
No citations.
  \or
(cit. on p. #2).%
  \else
(cit on pp. #2).%
  \fi%
}

\usepackage[nameinlink,capitalise,noabbrev]{cleveref}

\newtheorem{thm}{Theorem}[section]
\newtheorem{cor}[thm]{Corollary}
\newtheorem{prop}[thm]{Proposition}
\newtheorem{lem}[thm]{Lemma}

\newtheorem{defn}[thm]{Definition}

\newtheorem{exmp}[thm]{Example}

\newtheorem{rem}[thm]{Remark}
\newtheorem{conv}[thm]{Convention}
\makeatletter
\let\c@equation\c@thm
\makeatother
\numberwithin{equation}{section}

\DeclareMathOperator{\Sp}{Sp}
\DeclareMathOperator{\Hom}{Hom}

\DeclareMathOperator{\End}{End}
\DeclareMathOperator{\colim}{colim}
\DeclareMathOperator{\cA}{\mathcal{A}}
\DeclareMathOperator{\cC}{\mathcal{C}}
\DeclareMathOperator{\cD}{\mathcal{D}}
\DeclareMathOperator{\cE}{\mathcal{E}}
\DeclareMathOperator{\cP}{\mathcal{P}}
\DeclareMathOperator{\cS}{\mathcal{S}}
\DeclareMathOperator{\cI}{\mathcal{I}}
\DeclareMathOperator{\cF}{\mathcal{F}}
\DeclareMathOperator{\cG}{\mathcal{G}}
\DeclareMathOperator{\Id}{\mathrm{Id}}
\DeclareMathOperator{\fX}{\mathfrak{X}}

\DeclareMathOperator{\cO}{\mathcal{O}}
\DeclareMathOperator{\Ext}{Ext}
\DeclareMathOperator{\Map}{Map}
\DeclareMathOperator{\Tor}{Tor}
\DeclareMathOperator{\Spec}{Spec}
\DeclareMathOperator{\Mod}{Mod}
\DeclareMathOperator{\Stable}{Stable}
\DeclareMathOperator{\Comod}{Comod}
\DeclareMathOperator{\Ch}{Ch}
\DeclareMathOperator{\D}{\mathcal{D}}
\DeclareMathOperator{\Kos}{Kos}

\DeclareMathOperator{\fib}{fib}

\DeclareMathOperator{\Loc}{Loc}

\DeclareMathOperator{\Ho}{Ho}
\DeclareMathOperator{\Thick}{Thick}
\DeclareMathOperator{\Ind}{Ind}
\DeclareMathOperator{\Fun}{Fun}
\DeclareMathOperator{\stable}{Stable}
\DeclareMathOperator{\Aff}{Aff}
\DeclareMathOperator{\Ob}{Ob}
\DeclareMathOperator{\QCoh}{QCoh}
\DeclareMathOperator{\IndCoh}{IndCoh}
\DeclareMathOperator{\im}{Im}
\newcommand{\N}{\mathbb{N}}
\newcommand{\Q}{\mathbb{Q}}
\DeclareMathOperator{\Dinj}{\mathcal{D}^{\mathrm{inj}}}
\DeclareMathOperator{\Dinjpsi}{\mathcal{D}^{\mathrm{inj}}_{\Psi}}

\newcommand{\sh}{H}

\newcommand{\Locid}[1]{\mathrm{Loc}_{#1}^{\otimes}}

\newcommand{\fZ}{\mathfrak{Z}}

\newcommand{\tComod}{\Comod_\Psi^{I-\textnormal{tors}}}

\newcommand{\m}{\frak{m}}

\newcommand{\G}{\Gamma}

\newcommand{\cK}{\mathcal{K}}

\newcommand{\xr}{\xrightarrow}
\newcommand{\CH}{\check{C}H}

\newcommand{\Z}{\mathbb{Z}}
\newcommand{\tMod}[1]{\Mod_{#1}^{I-\textnormal{tors}}}
\newcommand{\uExt}{\underline{\Ext}}

\newcommand{\psilim}{\lim_{\Psi,k}}

\newcommand{\Stableh}{\mathrm{Stable}^\mathrm{h}}
\renewcommand{\frak}{\mathfrak}

\newcommand{\ilim}{\lim{\vphantom \lim}}
\let\lim\relax
\DeclareMathOperator{\lim}{lim}

\begin{document}

\title{Local duality in algebra and topology}

\author{Tobias Barthel}
\address{Department of Mathematical Sciences, University of Copenhagen, Universitetsparken 5, 2100 K{\o}benhavn {\O}, Denmark}


\author{Drew Heard}
\address{Department of Mathematics, University of Haifa, Mount Carmel, Haifa, 31905, Israel}

\author{Gabriel Valenzuela}
\address{Department of Mathematics, The Ohio State University, 100 Math Tower,
231 West 18th Avenue,	
Columbus, OH 43210-1174, USA}

\begin{abstract}
The first goal of this paper is to provide an abstract framework in which to formulate and study local duality in various algebraic and topological contexts. For any stable $\infty$-category $\cC$ together with a collection of compact objects $\mathcal{K} \subset \cC$ we construct local cohomology and local homology functors satisfying an abstract version of local duality. When specialized to the derived category of a commutative ring $A$ and a suitable ideal in $A$, we recover the classical local duality due to Grothendieck as well as generalizations by Greenlees and May. More generally, applying our result to the derived category of quasi-coherent sheaves on a quasi-compact and separated scheme $X$ implies the local duality theorem of Alonso Tarr{\'i}o, Jerem{\'i}as L{\'o}pez, and Lipman.

As a second objective, we establish local duality for quasi-coherent sheaves over many algebraic stacks, in particular those arising naturally in stable homotopy theory. After constructing an appropriate model of the derived category in terms of comodules over a Hopf algebroid, we show that, in familiar cases, the resulting local cohomology and local homology theories coincide with functors previously studied by Hovey and Strickland. Furthermore, our framework applies to global and local stable homotopy theory, in a way which is compatible with the algebraic avatars of these theories. In order to aid computability, we provide spectral sequences relating the algebraic and topological local duality contexts. 
\end{abstract}

\maketitle

{\hypersetup{linkcolor=black}\tableofcontents}

\section{Introduction}

\addtocontents{toc}{\protect\setcounter{tocdepth}{1}}
\subsection*{Motivation and main goals}

Local cohomology was introduced by Grothendieck as a substitute in algebraic geometry for the relative cohomology of a pair $U \subseteq X$ of spaces. In \cite{grothendieck_localcohomology,sga2}, he develops the foundations of the theory, which culminates in his local duality theorem, a local avatar of Serre duality. As such, it plays an important role in commutative algebra, algebraic geometry, representation theory, as well as algebraic approximations to homotopy theory.

Subsequently, this theory has been extended in many directions. In particular, for $I$ an ideal in a suitable commutative ring $A$, Greenlees and May~\cite{gm_localhomology, greenleesmay_completions} construct local homology functors $H_*^{I}$ on the derived category of $A$ as a homology theory dual to local cohomology $H_{I}^*$, and interpret Grothendieck's local duality as a special case of a spectral sequence
\begin{equation}\label{eqintro:gmss}
E_2^{p,q} = \Ext_A^p(H_I^{-q}(A),M) \implies H_{-(p+q)}^I(M),
\end{equation}
known as Greenlees--May duality. Moreover, they show that local homology is related to derived functors of $I$-adic completion in the same way that local cohomology is related to derived functors of $I$-power torsion. As a further generalization, Alonso Tarr{\'i}o, Jerem{\'i}as L{\'o}pez, and Lipman~\cite{local_cohom_schemes} globalize this theory to quasi-coherent sheaves over an appropriate scheme or even formal scheme. 

However, shadows of local cohomology and local homology functors together with the corresponding local duality type results appear in various areas that are not covered by the theory described above. As we will explain shortly, local duality admits a uniform and conceptual explanation when working in a more categorical context, at the same time simplifying constructions and computations as well as uncovering new instances of the theory by specializing the resulting abstract duality results. In light of this, the goal of the present paper is threefold. 
\begin{enumerate}
	\item Produce a simple abstract framework which allows the construction of local cohomology and homology in general categorical contexts, and deduce Greenlees--May duality formally in this framework. 
	\item Develop the theory for categories of comodules over certain Hopf algebroids, which translates into local duality for quasi-coherent sheaves over a general class of algebraic stacks. As an application, we generalize classical structural results about derived completion for local rings to certain algebraic stacks.
	
	\item Exhibit local duality as a fundamental phenomenon appearing in a variety of contexts, as for example stable, equivariant, and motivic homotopy theory. In particular, we study local duality in stable homotopy theory, recovering and generalizing many well-known results in chromatic homotopy theory from a more conceptual point of view. Moreover, we show that these different topological and algebraic incarnations are compatible with each other in a strong sense.  
\end{enumerate}

\subsection*{Outline of the paper}

We now describe the main results of this paper in more detail. As a starting point, we implicitly work within a homotopical enrichment of triangulated categories unless stated otherwise, and refer to such a model as a stable category. For concreteness, the reader can think of this as a pre-triangulated dg-category, a stable model category, or a stable $\infty$-category. 

Suppose $\cC$ is a closed symmetric monoidal, presentable stable category, and let $\mathcal{K} \subset \cC$ be a collection of compact objects; we refer to the pair $(\cC,\mathcal{K})$ as a local duality context. Following~\cite{hps_axiomatic, dwyer_complete_2002}, in Section~\ref{sec:abstract}, we construct from this data two functors $\Gamma = \Gamma_{\mathcal{K}}$ and $\Lambda = \Lambda_{\mathcal{K}}$ which are abstract versions of local cohomology and local homology with respect to $\mathcal{K}$, respectively. These functors have many desirable properties and in particular satisfy an abstract version of local duality.

\begin{thm}[Abstract local duality]
Let $(\cC,\mathcal{K})$ be a local duality context as above. For all $X,Y \in \cC$, there is a natural equivalence 
\begin{equation}\label{eqintro:localduality}
\xymatrix{\iExt(\Gamma X,Y) \simeq \iExt(X,\Lambda Y),}
\end{equation}
where $\iExt$ denotes the derived internal Hom object of $\cC$. Moreover, there exists a filtered system $(\Kos_s)_s$ of compact objects in $\cC$ such that
\[
\Gamma(-) \simeq \colim_s \Kos_s \otimes - \quad \text{and} \quad \Lambda(-) \simeq \lim_s D\Kos_s \otimes -.
\]
\end{thm}

In fact, these functors can be organized in a diagram of adjoints,
\begin{equation}\label{eqintro:diagram}
\xymatrix{& \cC^{\text{loc}} \ar@<0.5ex>[d] \ar@{-->}@/^1.5pc/[ddr] \\
& \cC \ar@<0.5ex>[u]^{L} \ar@<0.5ex>[ld]^{\Gamma} \ar@<0.5ex>[rd]^{\Lambda} \\
\cC^{\text{tors}} \ar@<0.5ex>[ru] \ar[rr]_{\sim} \ar@{-->}@/^1.5pc/[ruu] & & \cC^{\text{comp}}, \ar@<0.5ex>[lu]}
\end{equation}
in which the dotted maps indicate passing to the left orthogonal subcategory; we refer the reader to Section~\ref{sec:abstract} for the unexplained terminology in \eqref{eqintro:diagram}. Other formal consequences of the general theory include the existence of a fracture square relating $L$ and $\Lambda$, as well as an abstract version of affine duality in this context. 

In order to see the theory in action, Section~\ref{sec:modules} starts with the toy case that $\cC=\cD_A$ is the derived category of modules over a  commutative ring $A$ and $\mathcal{K} = \{A/I\}$ for an appropriate ideal $I \subset A$. We show that abstract local duality applied to the local duality context $(\cD_A,A/I)$ immediately recovers Greenlees--May duality and thus Grothendieck's local duality theorem. In particular, $\Gamma \simeq H_I$ and $\Lambda \simeq H^I$ in this case. 

\begin{thm}
Let $A$ be a reasonable commutative ring and $I \subset A$ an appropriate ideal; for the precise conditions on $A$ and $I$, see~\Cref{thm:main_D_R}. Abstract local duality for $(\cD_A,A/I)$ gives an isomorphism
\[
\Ext^*(H_I^*(A),M) \cong H_*^I(M)
\]
for all $M \in \cD_A$, whose associated Grothendieck spectral sequence coincides with the Greenlees--May spectral sequence \eqref{eqintro:gmss}.
\end{thm}

The case of module spectra over a commutative ring spectrum is also considered, and does not cause any extra difficulties. Moreover, we construct several spectral sequences calculating local cohomology and local homology groups which have been obtained by different methods in~\cite{greenleesmay_completions}, and deduce many well-known properties of these groups. Since local homology can now easily be identified with derived completion, as another application we discuss how to quickly reprove some results about derived completion due to Greenlees and May.

We then turn to the main example of interest in this paper, namely comodules over a flat Hopf algebroid $(A,\Psi)$. Hopf algebroids are a natural generalization of Hopf algebras, so include examples like the group algebra $k[G]$ and the Steenrod algebra, and form a link between stable homotopy theory and algebraic geometry: On the one hand, they arise as the cooperations $R_*R$ of flat ring spectra $R$, and on the other hand they provide rigid models for a large class of algebraic stacks. 

Although the category of comodules $\Comod_{\Psi}$ is a Grothendieck abelian category and thus has a corresponding derived category $\cD_{\Psi}$, this latter category has several deficiencies. Our first task is therefore to define an appropriate replacement $\Stable_{\Psi}$ of $\cD_{\Psi}$. Inspired by Krause's work~\cite{krausestablederivedcat}, in Section~\ref{sec:comodules} the stable category $\Stable_{\Psi}$ is constructed as an ind-category generated by the collection of dualizable objects and is shown to satisfy all the conditions we need for applying abstract local duality. This approach is familiar in the setting of stacks, where the category of ind-coherent sheaves is often better behaved than the (derived) category of quasi-coherent sheaves.  Moreover, we show that $\Stable_{\Psi}$ can be identified with the underlying stable category of the model category studied by Hovey in~\cite{hovey_htptheory}.

In Section~\ref{sec:comodules-duality}, we apply abstract local duality to the local duality context $(\Stable_{\Psi}, A/I)$, where $I \subseteq A$ is a suitable ideal. We then identify the local cohomology functors $\Gamma_I$ constructed abstractly, showing that they are compatible with the ordinary local cohomology of underlying modules. 

\begin{thm}
Suppose $(A,\Psi)$ is a Hopf algebroid and $I \subseteq A$ is a finitely generated regular invariant ideal, then there are local cohomology $\Gamma_I$ and homology functors $\Lambda^I$ on $\Stable_{\Psi}$ such that there is an equivalence
\[
\iHom_{\Psi}(\Gamma_IM,N) \simeq \iHom_{\Psi}(M,\Lambda^IM)
\]
for all $M,N \in \Stable_{\Psi}$. Moreover, the underlying $A$-module of $\Gamma_IM$ is equivalent to the local cohomology of the underlying $A$-module of $M$.
\end{thm}

This result translates immediately into a statement about quasi-coherent sheaves on certain algebraic stacks in view of the equivalence
\[
\Stable_{\Psi} \simeq \IndCoh_{\fX}, 
\]
where $\fX$ is the stack presented by $\Psi$ and $\IndCoh_{\fX}$ is the stable category of ind-coherent sheaves over $\fX$. We therefore obtain a construction of local cohomology $\Gamma_{\fZ}$ and local homology $\Lambda^{\fZ}$ for a certain class of algebraic stacks with respect to a closed substack $\fZ$, together with an associated local duality theorem. 

\begin{thm}
Suppose $\fX$ is a presentable algebraic stack and $\fZ \subseteq \fX$ is a closed substack cut out by an appropriate ideal $I$ of $A$, then there exists an equivalence
\[
\iHom_{\fX}(\Gamma_{\fZ}\cF,\cG) \simeq \iHom_{\fX}(\cF,\Lambda^{\fZ}\cG)
\]
for all $\cF,\cG \in \IndCoh_{\fX}$. Moreover, the pullback of $\Gamma_{\fZ}\cF$ to the affine cover $f\colon \Spec(A) \to \fX$ is equivalent to $\Gamma_{I}f^*\cF$. 
\end{thm}

In stark contrast to the last part of this theorem however, the local homology functors $\Lambda_{\fZ}$ are much more complicated than their affine analogues. In particular, they can be non-trivial both in positive and negative degrees, and the negative degree parts measure --- morally speaking --- the stackyness of $\fX$.

The next two sections contain a study of local dualities appearing in chromatic stable homotopy theory. While many specific instances of such results are well known, our approach emphasizes the uniform and conceptual nature of these phenomena. As a starting point, Section~\ref{sec:duality-spectra} is concerned with local duality in the category $\Sp$ of $p$-local spectra. In particular, we consider and compare three local duality contexts:
\begin{itemize}
 \item $(\Sp,F(n))$, for $F(n)$ a finite spectrum of type $n$.
 \item $(\Sp_{E(n)},L_nF(n))$, where $\Sp_{E(n)}$ is the category of $E(n)$-local spectra.
 \item $(\Mod_{E(n)}, E(n) \otimes F(n))$, where $\Mod_{E(n)}$ is the category of $E(n)$-module spectra. 
\end{itemize}
The various functors relating these local duality contexts are shown to be compatible with local cohomology. This allows us to generalize several of the structural results on localization from~\cite{hovey_morava_1999}. 

If $R$ is a ring spectrum, the associated homology theory $R_*$ naturally takes values in the category of comodules over the cooperations $R_*R$, i.e., we get a lift
\[
\xymatrix{& \Comod_{R_*R} \ar[d]^{\text{forget}} \\
\Sp \ar[r]_-{R_*} \ar@{-->}[ru]^{R_*} & \Mod_{R_*}.}
\]
In the case where $R = BP$ or $E(n)$ we study the resulting functors between the topological and algebraic local duality contexts. We show in Section~\ref{sec:duality-chromatic} how this gives a conceptual framework for the work of Hovey and Strickland \cite{hs_leht, hs_localcohom, hovey_chromatic}, and moreover provide some natural generalizations of their work. This structure is most naturally expressed in terms of plethories, which are an algebraic incarnation of the descent data attached to the inclusion of an open substack into $\mathcal{M}_{\mathrm{fg}}$, the stack of 1-dimensional formal groups. The main result of this section is the following transchromatic comparison result, which generalizes a theorem of Hovey~\cite{hovey_chromatic} for $k=n$. 
\begin{thm}
Let $k \le n$, $I_{k+1} =( p,v_1,\ldots,v_{k})$, and $\Stable_{E(n)_*E(n)}^{I_{k+1}-\mathrm{loc}}$ be the stable category of $I_{k+1}$-local $E(n)_*E(n)$-comodules, then there exists an equivalence 
\[
\Stable_{E(n)_*E(n)}^{I_{k+1}-\mathrm{loc}} \simeq \Stable_{E(k)_*E(k)}
\]
of symmetric monoidal stable categories.
\end{thm}
As an immediate application we obtain a new change of rings theorem, passing between $E$-theories of different heights. 

Section~\ref{sec:other} of the paper collects several other local duality contexts to which our theory applies, as for instance quasi-coherent sheaves on schemes, equivariant homotopy theory, and motivic homotopy theory. While we do not discuss these examples in all detail, we hope that they illustrate the flexibility of our methods and results. 

For the convenience of the reader, the final section displays various algebraic and topological local duality contexts and their relation in a compact form.

\subsection*{Comparison with the literature}

Local duality has undergone a significant development since Grothendieck's 1961 seminar. Rather than attempting to review this rich history, we would like to point out one of the key insights: The local duality spectral sequence, as well as its generalization \eqref{eqintro:gmss} due to Greenlees and May, are direct consequences of the equivalence of derived hom-objects:
\[
\mathbf{R}\Hom(\Gamma_IM,N) \simeq \mathbf{R}\Hom(M,\Lambda^I N).
\]
In other words, local cohomology $\Gamma_I$ is derived left adjoint to local homology $\Lambda^I$, a statement usually referred to as Greenlees--May duality; for simplicity, we shall say $\Gamma_I$ and $\Lambda^I$ satisfy local duality. This derived point of view was exploited in a series of papers by Alonso Tarr{\'i}o, Jerem{\'i}as L{\'o}pez, and Lipman, extending local duality to quasi-coherent sheaves on schemes~\cite{local_cohom_schemes}, and even formal schemes~\cite{ajl_formalschemes}. Moreover, in the affine situation this has helped  to relax and clarify the assumptions on the ring and the ideal $I$, see for example~\cite{gm_localhomology, schenzel_proreg, psy}. 

Our approach combines the abstract categorical framework of Hovey, Palmieri, and Strickland~\cite{hps_axiomatic} with the construction of torsion and completion functors for modules by Dwyer, Iyengar, and Greenlees~\cite{dwyer_complete_2002, dgi_duality,greenlees_axiomatic}.  Similar ideas have appeared in derived algebraic geometry~\cite{dag12} and very recently in equivariant stable homotopy theory~\cite{mnn_eqnilpotence}. 

The difference between our approach and the classical one becomes most transparent in the case of the derived category of modules over a commutative ring $A$. Under some mild conditions on $A$, Schenzel~\cite{schenzel_proreg} proved that classical local cohomology $\Gamma_I^{\text{cl}}$ and local homology $\Lambda^{I,\text{cl}}$, defined using the Koszul complex of the ideal $I \subseteq A$, satisfy local duality \emph{if and only if} $I$ is generated by a weakly proregular sequence. In contrast to this, we construct local cohomology functors $\Gamma_I$ and local homology functors $\Lambda^I$ on $\cD_A$ abstractly for any $I$, such that
\begin{enumerate}
	\item For all $I$, $\Gamma_I$ and $\Lambda^I$ satisfy local duality.
	\item If $I$ is generated by a weakly proregular sequence, then $\Gamma_I \simeq \Gamma_I^{\text{cl}}$ and hence $\Lambda^{I} \simeq \Lambda^{I,\text{cl}}$. 
\end{enumerate} 

As mentioned before, our results apply in particular to ind-coherent sheaves on a large class of algebraic stacks, yielding local cohomology and local homology functors, and a local duality theorem in this context. This question has been studied previously in unpublished work of Goerss~\cite{goerss_quasi-coherent_2008}; however, he uses a local-to-global argument and works in the derived category of quasi-coherent sheaves. Moreover, local cohomology for comodules over a Hopf algebroid or quasi-coherent sheaves on the associated stack has been carefully constructed and analysed in the recent PhD thesis by Tobias Sitte~\cite{sitte2014local}. 

The applications in Sections~\ref{sec:duality-spectra} and~\ref{sec:duality-chromatic} mostly recover or mildly extend the results contained in the tremendous body of work by Hovey and Strickland~\cite{hovey_morava_1999, hs_leht, hs_localcohom, hovey_chromatic}. We claim no novelty in this respect, but hope that our framework provides a conceptual and unifying perspective on this material, interpreting it as an avatar of local duality in chromatic homotopy theory.

\subsection*{Conventions and terminology}

In this paper, we work within a homotopically enriched categorical context, thereby avoiding some of the difficulties inherent in the theory of triangulated categories. There is a variety of different models available, like (pre-triangulated) dg-categories, (stable) model categories, or (stable) derivators, and we choose the quasi-categorical setting developed by Joyal~\cite{joyalqcat} and Lurie~\cite{htt,ha}. For simplicity, we will refer to a quasi-category as an $\infty$-category throughout this paper.  

Unless otherwise noted, all categorical constructions are implicitly considered derived. For example, the tensor product $\otimes$ usually refers to the derived tensor product, and limits and colimits mean homotopy limits and homotopy colimits, respectively. All limits and colimits in this paper are assumed to be small, i.e., indexed by a small category. Moreover, all filtered diagrams will be $\omega$-filtered, although most of the abstract theory can be developed for an arbitrary regular cardinal $\kappa$ as in~\cite{htt}. 

A presentable $\infty$-category $\cC$ is said to be generated by a collection of objects $\cG \subseteq \cC$ if the smallest localizing subcategory of $\cC$ containing $\cG$ is $\cC$. By the proof of \cite[Lem.~2.2.1]{schwedeshipley_modules}, if the elements of $\cG$ are compact, then $\cG$ generates $\cC$ if and only if $\cG$ detects equivalences in $\cC$. 

If $\cC$ is a closed symmetric monoidal stable $\infty$-category, the internal function object will be denoted by $\iHom_{\cC}$ to distinguish it from the merely spectrally enriched categorical mapping object $\Hom_{\cC}$. This is related to the usual mapping space via a natural weak equivalence $\Omega^\infty \Hom_{\cC}(X,Y) \simeq \Map_{\cC}(X,Y)$. If no confusion is likely to arise, the subscript $\cC$ will be omitted from the notation. 
\color{black}
Unless noted otherwise, a functor $F\colon \cC \to \cD$ between stable $\infty$-categories will be assumed to be exact, i.e., preserving finite limits and finite colimits. Similarly, an equivalence between symmetric monoidal categories is assumed to be symmetric monoidal. 

When dealing with chain complexes, we will always  employ homological grading, i.e., complexes are written as 
\[
\xymatrix{\ldots \ar[r]^-d & X_{1} \ar[r]^-d & X_0 \ar[r]^-d & X_{-1} \ar[r]^-d & \ldots}
\]
with the differential $d$ lowering degree by 1. As usual, taking cohomology of a chain complex $X$ reverses the sign of the homology, that is $H^*(X) = H_{-*}(X)$.

\subsection*{Acknowledgements}

This paper would not have been possible without the work of Hovey and Strickland, and their influence will be apparent throughout. Furthermore, we would like to thank Martin Frankland, Paul Goerss, John Greenlees, Mike Hopkins, Mark Hovey, Akhil Mathew, Peter May, Hal Sadofsky, Tomer Schlank, Chris Schommer-Pries, Tobias Sitte, Nat Stapleton, Craig Westerland, and Luke Wolcott for helpful discussions, Jay Shah for helpful comments on an earlier version of this paper, the referee for many useful suggestions and corrections, Andrew Blumberg for his tremendous support during the submission process, and the Max Planck Institute for Mathematics for its hospitality. The first-named author was partially supported by the DNRF92 and the European Unions Horizon 2020 research and innovation programme under the Marie Sklodowska-Curie grant agreement No.~751794, and the second-named author was partially supported by the SPP 1786.

\section{Abstract local duality}\label{sec:abstract}

The goal of this section is to formulate and prove our abstract local duality theorem, \Cref{thm:hps}, generalizing the approach taken by Hovey, Palmieri, and Strickland~\cite{hps_axiomatic} as well as Dwyer and Greenlees in \cite{dwyer_complete_2002}. To this end, we start by recalling some relevant background material from~\cite{hps_axiomatic}, \cite{htt}, and \cite{ha}. We then develop enough of the theory of ind-categories for our applications in later sections. The reader familiar with this material is invited to skip Sections \ref{sec:axiomatic} and \ref{sec:ind-categories}, and to return to them when needed. In \Cref{sec:ald}, we set up our general local duality framework and draw some easy consequences. As we will see, a few other duality results that have appeared previously follow formally. 

\subsection{Stable categories}\label{sec:axiomatic}

The axiomatization of the stable homotopy category, referred to as an axiomatic stable homotopy category, was introduced and studied extensively by Hovey, Palmieri, and Strickland~\cite{hps_axiomatic}. The authors work in the setting of triangulated categories; for our purposes we have to work with a homotopically enhanced model of axiomatic stable homotopy theory, based on the notion of quasi-category as developed by Joyal~\cite{joyalqcat} and Lurie~\cite{htt}; these will be referred to as $\infty$-categories throughout.

\begin{defn}
	A stable category $\cC = (\cC,\otimes,A)$ is a stable, presentable, symmetric monoidal $\infty$-category $(\mathcal{C},\otimes,A)$ with tensor unit $A$, such that $\otimes$ commutes with colimits separately in each variable. 
\end{defn}

Since the symmetric monoidal product of $\cC$ preserves colimits separately in each variable and $\cC$ is presentable, there exists an internal hom object $\iHom(Y,-)$ in $\cC$ right adjoint to $- \otimes Y$, i.e., for all $X,Y,Z \in \cC$ there is a natural equivalence
\[
\Hom(X\otimes Y, Z) \simeq \Hom(X, \iHom(Y,Z)).
\]
In other words, the symmetric monoidal structure on $\cC$ is closed. In particular, we get an internal duality functor $D(-) = \iHom(-,A)\colon \cC^{\mathrm{op}} \to \cC$. We record the following lemma, which will be used repeatedly in the rest of the paper. 

\begin{lem}\label{lem:homihom}
Suppose that $\cC$ is a stable category which is generated by a set of generators $\cG$. Given $X, X', Y,Y' \in \cC$ and a map $\phi\colon\iHom(X,Y) \to \iHom(X',Y')$, the following are equivalent:
\begin{enumerate}
	\item $\phi$ is an equivalence $\iHom(X,Y) \simeq \iHom(X',Y')$.
	\item $\phi$ induces equivalences $\Hom(G \otimes X,Y) \simeq \Hom(G \otimes X',Y')$ for all $G \in \cG$. 
\end{enumerate}
\end{lem}
\begin{proof}
Statement (1) implies (2) by applying $\Hom(G,-)$ and using the adjunction. Conversely, (1) can be deduced from (2) by using the adjunction and the fact that $\cG$ detects equivalences in $\cC$ by \cite[Lem.~1.4.5(b)]{hps_axiomatic}. 
\end{proof}

\begin{defn}
Let $X$ be an object of $\cC$. If the functor $\Hom(X,-)$ corepresented by $X$ preserves filtered colimits, then $X$ is called compact, and the full subcategory of $\cC$ on the compact objects is denoted by $\cC^{\omega}$.  
If the natural map $DX \otimes Y \to \iHom(X,Y)$ is an equivalence for all $Y \in \cC$, then $X$ is said to be dualizable.
\end{defn}

As shown in~\cite[Cor.~1.4.4.2]{ha}, a stable category $\cC$ admits a $\kappa$-compact generator $G$ for some regular cardinal $\kappa$. If $\cC$ admits a set of $\omega$-compact generators, we say that $\cC$ is compactly generated. 

\begin{rem}
Let $\cC$ be a stable category which is compactly generated by a collection $\cG$ of dualizable objects. Then by~\cite[Thm.~1.1.2.15]{ha} and \cite[Thm.~2.3.2]{hps_axiomatic}, the homotopy category $\Ho(\cC)$ is an algebraic stable homotopy category in the sense of~\cite{hps_axiomatic}. Many of the results proven there carry over to the setting of stable categories, as demonstrated in \cite{ha}.
\end{rem}

\begin{lem}\label{lem:algcompdual}
Suppose $\cC$ is a stable category, compactly generated by dualizable objects. If $X \in \cC$ is compact, then $X$ is dualizable. Moreover, if $A$ is compact, then the converse holds.
\end{lem}
\begin{proof}
This is an easy consequence of the definitions. For a proof in the setting of axiomatic stable homotopy theory, see \cite[Thm.~2.1.3(c),(d)]{hps_axiomatic}.
\end{proof}
\begin{rem}
	We note that this need not be true without the assumption that $\cC$ is compactly generated by dualizable objects. A counterexample arises in the stable motivic homotopy category $\Sp(S)$, where $S$ is the spectrum of a discrete valuation ring, see \cite[Rem.~8.2]{mot_le}. However, this condition is satisfied for all the examples we consider in this paper. 
\end{rem}
\begin{defn}
A full subcategory $\cD \subseteq \cC$ is called thick if it is closed under retracts, desuspensions, and finite colimits in $\cC$. A localizing subcategory is a thick subcategory that is also closed under filtered colimits in $\cC$. Finally, a localizing subcategory is said to be a localizing ideal if it is also closed under tensor products with objects in $\cC$. 
\end{defn}

Given a collection $\cS \subseteq \cC$ of objects in $\cC$, the smallest thick subcategory of $\cC$ containing $\cS$ will be denoted by $\Thick_{\cC}(\cS)$ and is called the thick subcategory of $\cC$ generated by $\cS$. The localizing subcategory $\Loc_{\cC}(\cS)$ is similarly defined as the smallest localizing subcategory of $\cC$ containing $\cS$. Likewise, we define the localizing ideal generated by $\cS$, $\Locid{\cC}(\cS)$, to be the smallest localizing ideal containing $\cS$. If the ambient category $\cC$ is clear from context, then the subscript will be omitted.
\color{black}

\subsection{Ind-categories}\label{sec:ind-categories}

In this section, we recall the definition and the salient features of the ind-category $\Ind(\cC)$ associated to a small $\infty$-category $\cC$, and refer to~\cite{htt} and~\cite{ha} for the details. 

Roughly speaking, the ind-category on $\cC$ is the smallest $\infty$-category closed under filtered colimits and containing $\cC$. Its objects can be described as formal filtered colimits of objects in $\cC$. If $F=\colim_i F(i), G=\colim_j G(j) \in \Ind(\cC)$ are two such objects, then the space of morphisms between them is defined as
\[ 
\Hom_{\Ind(\cC)}(F,G) = \lim_i \colim_j \Hom_{\cC}(F(i),G(j)).
\]

By virtue of~\cite[Cor.~5.3.5.4]{htt}, there is the following explicit model for $\Ind(\cC)$.
\begin{defn}
Let $\cC$ be a small $\infty$-category and denote by $\cP(\cC) = \Fun(\cC^{\mathrm{op}},\cS)$ the $\infty$-category of space-valued presheaves on $\cC$. The ind-category $\Ind(\cC)$ of $\cC$ is defined as the full subcategory of $\cP(\cC)$ on those objects which are filtered colimits of representable ones. 
\end{defn}

\begin{rem}
If $\cC$ admits finite colimits, then $\Ind(\cC)$ can be equivalently described as those presheaves which preserve finite limits, the so-called flat functors. 
\end{rem}

The construction of $\Ind(\cC)$ comes with a (restricted) Yoneda embedding
\[ 
j\colon \cC \to \Ind(\cC)
\]
which preserves all colimits that exist in $\cC$ and has image in the compact objects of $\Ind(\cC)$, see~\cite[Prop.~5.3.5.5, Prop.~5.3.5.14]{htt}. Note that, if $\cC$ is idempotent complete, then $\cC = \Ind(\cC)^{\omega}$ via the Yoneda embedding. The ind-category has the following universal property~\cite[Prop.~5.3.5.10]{htt}, which essentially characterizes it, together with the property that it admits filtered colimits.

\begin{prop}\label{prop:induniversal}
Let $\cC$ and $\cD$ be $\infty$-categories and assume that $\cC$ is small and that $\cD$ admits filtered colimits. The Yoneda embedding $j$ then induces an equivalence
\[
\xymatrix{\Fun^{\mathrm{cont}}(\Ind(\cC),\cD) \ar[r]^-{\sim} & \Fun(\cC,\cD)}
\]
of $\infty$-categories, where $\Fun^{\mathrm{cont}}$ is the $\infty$-category of continuous functors, i.e., functors which preserve filtered colimits. 
\end{prop}

The construction of the ind-category behaves well with respect to symmetric monoidal structures, as shown in~\cite[Cor.~4.8.1.13]{ha}.

\begin{prop}\label{prop:indsymmmon}
If $\cC$ is a small symmetric monoidal $\infty$-category, then there exists a symmetric monoidal structure $\otimes$ on $\Ind(\cC)$ preserving filtered colimits separately in each variable and such that the Yoneda embedding $j$ is symmetric monoidal. Moreover, if $X \in \cC$ is a dualizable object, then $jX \in \Ind(\cC)$ is also dualizable.

Suppose $\cD$ is a symmetric monoidal $\infty$-category closed under filtered colimits. If the symmetric monoidal structure on $\cD$ preserves filtered colimits separately in each variable, then $j$ induces an equivalence
\[
\xymatrix{\Fun^{\otimes}(\Ind(\cC),\cD) \ar[r]^-{\sim} & \Fun(\cC,\cD)}
\]
where $\Fun^{\otimes}$ denotes the $\infty$-category of symmetric monoidal continuous functors from $\cC$ to $\cD$.  
\end{prop}

Suppose $\cD$ is an $\infty$-category admitting filtered colimits and $f\colon \cC \to \cD$ is a functor, then the universal property of $\Ind(\cC)$ gives a functor $F\colon \Ind(\cC) \to \cD$ and a commutative diagram 
\[ 
\xymatrix{\cC  \ar[r]^-j \ar[rd]_f & \Ind(\cC) \ar[d]^F \\
& \cD}
\]
of $\infty$-categories. 
By~\cite[Prop.~5.3.5.11]{htt}, the induced functor $F$ is an equivalence if and only if the following three conditions are satisfied:
\begin{enumerate}
 \item $f$ is fully faithful.
 \item $f$ factors through the full subcategory of compact objects of $\cD$. 
 \item The essential image of $f$ generates $\cD$ under filtered colimits. 
\end{enumerate}

Following~\cite[Prop.~5.3.2.9]{htt}, we say that a functor $f\colon \cC \to \cC'$ between $\infty$-categories admitting finite colimits is right exact if it preserves all finite colimits. The construction of the ind-category can be extended to an endofunctor on the $\infty$-category of $\infty$-categories, and we write 
\[ 
\Ind(f)\colon \Ind(\cC) \to \Ind(\cC')
\]
for the image of $f$ under this functor. In~\cite[Prop.~5.3.5.13]{htt}, Lurie proves a criterion for the existence of a right adjoint to $\Ind(f)$.

\begin{prop}\label{prop:indrightadjoint}
Suppose $f\colon  \cC \to \cC'$ is a functor between small $\infty$-categories admitting finite colimits, then the following are equivalent:
\begin{enumerate}
 \item $f$ is right exact.
 \item $\Ind(f)$ admits a right adjoint.
\end{enumerate}
Moreover, if these conditions hold, then the right adjoint to $\Ind(f)$ is given by  restricting the functor $-\circ f\colon \cP(\cC') \to \cP(\cC)$ to the corresponding ind-categories. 
\end{prop}

The ind-category $\Ind(\cC)$ inherits many pleasant properties from $\cC$; the next result summarizes the ones that we will need later in the paper. 

\begin{thm}[Lurie]\label{prop:indproperties}
Let $\cC$ be a small stable $\infty$-category with a closed symmetric monoidal structure $\otimes$, and assume that $\otimes$ preserves finite colimits separately in each variable. The ind-category $\Ind(\cC)$ on $\cC$ is then a compactly generated stable category. In other words, $\Ind(\cC)$ is a compactly generated stable $\infty$-category which has a closed symmetric monoidal structure $\otimes$ preserving colimits separately in each variable, making the Yoneda embedding $j$ into a symmetric monoidal functor. \end{thm}
\begin{proof}
This is a combination of \cite[Thm.~5.5.1.1]{htt}, \cite[Prop.~1.1.3.6]{ha}, and \cite[Cor. 4.8.1.13]{ha}. Since $\Ind(\cC)$ is presentable, and $\otimes$ preserves colimits separately in each variable, the symmetric monoidal structure is closed by the adjoint functor theorem~\cite[Cor.~5.5.2.9]{htt}. 
\end{proof} 

Ind-categories will provide us with a sufficient supply of compactly generated categories. Given a stable category $(\cD,\otimes,A)$, let $\cG_{\cD}$ denote a set of representatives for the dualizable objects in $\cD$. Then, we can use \Cref{prop:indproperties} to produce a new stable category $(\cC,\otimes,A)$ as
\[
\cC = \Ind(\Thick_{\cD}(\cG_{\cD})).
\]
Note that $\Thick_{\cD}(\cG_{\cD})$ is precisely the full subcategory of dualizable objects in $\cD$. By construction, $\cC$ is compactly generated by $\cG_{\cD}$, and in many examples turns out to be a good approximation to $\cD$, see \cite[Sec.~1.3.6]{ha}. We will use this idea in \Cref{sec:comodules}.

\begin{rem}
If the elements of $\cG_{\cD}$ are a set of compact generators of $\cD$, then $\cC \simeq \cD$ naturally.
\end{rem}

We finish this subsection with a result that will be useful later.

\begin{lem}\label{lem:locindthick}
Suppose $(\cD,\otimes,A)$ is a stable category and let $\cC = \Ind(\Thick_{\cD}(\cG))$ be the stable presentable $\infty$-category compactly generated by some set $\cG \subseteq  \cD$. If $\cK \subseteq \Thick_{\cD}(\cG)$ is a collection of objects, then there exists a canonical equivalence of $\infty$-categories $\Ind(\Thick_{\cD}(\cK)) \simeq \Loc_{\cC}(\cK)$, and a natural adjunction
\[
\xymatrix{F\colon \Loc_{\cC}(\cK) \ar@<0.5ex>[r] & \cC \colon G \ar@<0.5ex>[l]}
\]
with $F$ the obvious inclusion.  
\end{lem}
\begin{proof}
Since $\cK \subseteq \Thick_{\cD}(\cG)$, there is a fully faithful functor $f\colon \Thick_{\cD}(\cK) \to \Thick_{\cD}(\cG)$ inducing a commutative diagram of $\infty$-categories:
\[
\xymatrix{\Thick_{\cD}(\cK) \ar[r]^-j \ar[d]_f & \Ind(\Thick_{\cD}(\cK)) \ar@{-->}[d]^{F} \\
\Thick_{\cD}(\cG) \ar[r]_-j & \cC}
\]
with $F \simeq \Ind(f)$. The composite $j \circ f$ is clearly right exact, so~\Cref{prop:indrightadjoint} applies to give a right adjoint $G\colon \cC \to \Ind(\Thick_{\cD}(\cK))$ to $F$. 

It remains to identify $\Ind(\Thick_{\cD}(\cK))$ with the localizing subcategory of $\cC$ generated by $\cK$. The universal property of $\Ind(\Thick_{\cD}(\cK))$ induces a commutative diagram
\[
\xymatrix{\Thick_{\cD}(\cK) \ar[r]^-j \ar[rd]_-{h} & \Ind(\Thick_{\cD}(\cK)) \ar@{-->}[d]^-{H} \\
& \Loc_{\cC}(\cK),}
\]
where $h$ is the inclusion. Since $h$ is fully faithful and $\cK$ is a collection of compact generators for $\Loc_{\cC}(\cK)$, $H$ is an equivalence.
\end{proof}

\subsection{Abstract local duality}\label{sec:ald}

In this section, let $\cC = (\cC,\otimes,A)$ be a stable category compactly generated by dualizable objects. The assumption of compact generation could be weakened, but all examples we will consider in this paper satisfy it, so there is not much loss in generality. To start with, we recall the notion of left-orthogonal to a subcategory. 

\begin{defn}\label{def:leftorthogonal}
Suppose $\cD$ is a full subcategory of $\cC$, then the left-orthogonal $\cD^{\perp}$ of $\cD$ in $\cC$ is the full subcategory of $\cC$ on those objects $Y \in \cC$ such that $\Hom(X,Y) \simeq 0$
for all $X \in \cD$. 
\end{defn}

The following construction gives an abstract version of Miller's finite localization functors \cite{miller_finiteloc}.

\begin{lem}\label{lem:leftorthogonal}
If $\cD$ is a localizing subcategory of $\cC$, then its left orthogonal $\cD^{\perp}$ is also a stable category. If $\cD$ is additionally generated by a set of compact objects in $\cC$, then $\cD^{\perp}$ is a compactly generated localizing subcategory as well.  
\end{lem}
\begin{proof}
Since $\cC$ is stable and $\cD$ is localizing, \cite[Prop.~1.4.4.11]{ha} implies the existence of an accessible $t$-structure $(\cC_{< 0},\cC_{\ge 0})$ on $\cC$ with $\cC_{\ge 0} = \cD$ and $\cC_{< 0} = \cD^{\perp}$. In particular, $\cD^{\perp}$ is presentable by \cite[Prop.~1.4.4.13]{ha}. 

In order to show that $\cD^{\perp}$ is a localizing subcategory of $\cC$ it suffices to prove that the inclusion $\cD^{\perp} \to \cC$ preserves filtered colimits. The fiber sequence of endofunctors of $\cC$
\[
\xymatrix{\tau_{\ge 0} \ar[r] & \Id \ar[r] & \tau_{< 0}}
\]
reduces the claim to showing that $\tau_{\ge 0}$ preserves filtered colimits. By \cite[Prop.~5.5.7.2]{htt}, this is equivalent to the claim that the inclusion  $\cD \to \cC$ preserves compact objects, which holds by assumption. 
\end{proof}

\begin{rem}\label{rem:perpgen}
In fact, this proof shows that the images under $\tau_{<0}$ of the compact generators of $\cC$ form a set of compact generators for $\cD^{\perp}$.
\end{rem}

We are now ready to describe the categorical framework in which abstract local duality holds. 

\begin{defn}
If $\cC$ is a stable category, compactly generated by dualizable objects, and $\cK$ is a set of compact objects in $\cC$, then the pair $(\cC,\cK)$ is called a local duality context. If $\cK = \{ K \}$ consists of a single object, we simply write $(\cC,K)$ for $(\cC,\cK)$.
\end{defn}

Let $(\cC,\cK)$ be a local duality context. We define $\cC^{\cK-\text{tors}}$ as the localizing ideal in $\cC$ generated by $\cK$. Denote the left-orthogonal of $\cC^{\cK-\text{tors}}$ by $\cC^{\cK-\text{loc}}$, which in turn has a left-orthogonal called $\cC^{\cK-\text{comp}}$. If $\cK$ is clear from context, it will be omitted from the notation, and we will refer to $\cC^{\text{tors}}$, $\cC^{\text{loc}}$, and $\cC^{\text{comp}}$ as the category of torsion, local, and complete objects, respectively. Moreover, let $\iota_{\text{tors}}\colon \cC^{\text{tors}} \to \cC$ be the inclusion functor, and similarly for $\iota_{\text{loc}}$ and $\iota_{\text{comp}}$. By the previous \Cref{lem:leftorthogonal}, these categories inherit the structure of stable categories. \Cref{lem:locindthick} and the adjoint functor theorem~\cite[Cor.~5.5.2.9]{htt} then yield the first part of our abstract local duality theorem, which is an $\infty$-categorical version of \cite[Thm.~3.3.5]{hps_axiomatic}.

\begin{rem}\label{loc=locid}
	Note that if $\cC$ is generated by the tensor unit and $\cK$ is any class of objects, then $\Locid{\cC}(\cK)=\Loc_{\cC}(\cK)$. In particular, $\cC^{\cK-\text{tors}}=\Loc_{\cC}(\cK)$ under these conditions.
\end{rem}
\color{black}

\begin{thm}[Abstract local duality]\label{thm:hps}
Let $(\cC,\otimes,A)$ be a stable category compactly generated by dualizable objects and let $(\cC,\cK)$ be a local duality context.
\begin{enumerate}
 \item The functor $\iota_{\mathrm{tors}}$ has a right adjoint $\Gamma = \Gamma_{\cK}$, and the functors $\iota_{\mathrm{loc}}$ and $\iota_{\mathrm{comp}}$ have left adjoints $L = L_{\cK}$ and $\Lambda = \Lambda_{\cK}$, respectively, 
 which induce natural cofiber sequences
 \[
 \Gamma X \longrightarrow X \longrightarrow LX 
 \]
 and 
 \[
 \Delta(X) \longrightarrow X \longrightarrow \Lambda X
 \]
 for all $X \in \cC$. In particular, $\Gamma$ is a colocalization functor and both $L$ and $\Lambda$ are localization functors. 
 
 \item Both $\Gamma \colon \cC^{} \to \cC^{\mathrm{tors}}$ and $L\colon \cC^{} \to \cC^{\mathrm{loc}}$ are smashing, i.e., $\Gamma(X) \simeq X \otimes \Gamma(A)$ for all $X$ and similarly for $L$. Moreover, $L\colon \cC^{} \to \cC^{\mathrm{loc}}$ preserves compact objects. 
 
 \item The functors $\Lambda \iota_{\mathrm{tors}}\colon \cC^{\mathrm{tors}} \to \cC^{\mathrm{comp}}$ and $\Gamma \iota_{\mathrm{comp}}\colon \cC^{\mathrm{comp}} \to \cC^{\mathrm{tors}}$ are mutually inverse equivalences of stable categories. Moreover, there are natural equivalences of functors
 \[
 \xymatrix{\Lambda \Gamma \ar[r]^-{\sim} & \Lambda & \Gamma \ar[r]^-{\sim} & \Gamma \Lambda.}
 \]
 \item When viewed as endofunctors on $\cC$ via the inclusions, the functors $(\Gamma,\Lambda)$ form an adjoint pair, so that we have a natural equivalence 
\[\xymatrix{\iExt(\Gamma X,Y) \simeq \iExt(X,\Lambda Y)}\]
for all $X,Y \in \cC$. In particular, $\iHom(\Gamma A,Y) \simeq \Lambda Y$ for all $Y \in \cC$. Similarly, $L$ is left adjoint to $\Delta$. 
\end{enumerate}
\end{thm}
\begin{proof}

Let $\mathcal{A}$ denote the closure of $\cK$ under tensoring with the generators of $\mathcal{C}$. By assumption, $\mathcal{C}$ is compactly generated by dualizable objects. Therefore \cite[Thm.~3.3.5]{hps_axiomatic} applies and yields the first three parts of the theorem. From this, we then formally deduce our abstract version of local duality in (4). Indeed, using \Cref{lem:homihom}, there is the following binatural string of equivalences:
\color{black}
\begin{align*}
\iExt(\Gamma X,Y) & \simeq \iExt(\Lambda \Gamma X,\Lambda Y) \\ 
& \simeq \iExt(\Lambda X,\Lambda Y) \\ 
& \simeq \iExt(X,\Lambda Y), 
\end{align*}
by (3) and (1). Using the cofiber sequences of (1), the claim about the pair $(L,\Delta)$ follows formally from this.
\end{proof}

\begin{rem}
	It is worth noting that the functors produced by the theorem above only depend on the category $\cC^{\mathrm{tors}}$ rather than on any particular set of generators $\cK$. In practice, the category of torsion objects is often generated by a set objects that is natural to the context one is working in, which justifies our choice of notation in \cref{thm:hps}.
\end{rem}
\color{black}

\begin{rem}
In view of \Cref{lem:homihom}, Part (4) of this theorem can be stated equivalently  as $\Hom(\Gamma X,Y) \simeq \Hom(X,\Lambda Y)$ for all $X,Y \in \cC$, using that $\Gamma$ is smashing.
\end{rem}

\begin{defn}
The functor $\Gamma$ is called the local cohomology or torsion functor (with respect to $\cK$), and $\Lambda$ will be referred to as the local homology or completion functor (with respect to $\cK$).
\end{defn}

The situation of the theorem can be summarized in the following diagram of adjoints,
\begin{equation}\label{eq:abstractdiagram}
\begin{gathered}
\xymatrix{& \cC^{\text{loc}} \ar@<0.5ex>[d] \ar@{-->}@/^1.5pc/[ddr] \\
& \cC \ar@<0.5ex>[u]^{L} \ar@<0.5ex>[ld]^{\Gamma} \ar@<0.5ex>[rd]^{\Lambda} \\
\cC^{\text{tors}} \ar@<0.5ex>[ru] \ar[rr]_{\sim} \ar@{-->}@/^1.5pc/[ruu] & & \cC^{\text{comp}}, \ar@<0.5ex>[lu]}
\end{gathered}
\end{equation}
where the dotted arrows indicate left orthogonality. The two cofiber sequences of \Cref{thm:hps}(1) are related formally through a fracture square, generalizing the various well-known Hasse squares in stable homotopy theory; for a proof in this generality, see~\cite[Cor.~2.4]{greenlees_axiomatic}. 

\begin{cor}\label{cor:fracturesquare}
In the situation of the above theorem, there is a pullback square of functors
\[\xymatrix{\mathrm{Id} \ar[r] \ar[d] & \Lambda \ar[d] \\
L \ar[r] & L\Lambda,}\]
whose horizontal and vertical fibers are $\Delta$ and $\Gamma$, respectively.
\end{cor}

\begin{exmp}
Let $\cD_{\Z}$ be the derived category of abelian groups and consider the full subcategory $\cD_{\Z}^{\mathrm{tors}} = \Loc(\{\Z/p\}_p)$ of torsion abelian groups. For any $M \in \cD_{\Z}$, the pullback square of \Cref{cor:fracturesquare} can then be identified with the well-known fracture square
\[
\xymatrix{M \ar[r] \ar[d] & \prod_{p}M_p \ar[d] \\
\Q \otimes M \ar[r] & \Q \otimes \prod_{p}M_p,}
\]
where $M_p$ denotes the $p$-completion of $M$ and $\Q \otimes \prod_{p}M_p$ are by definition the finite adeles of $M$. 
\end{exmp}

\begin{cor}
With notation as in \Cref{thm:hps}, there is a pullback square of $\infty$-categories
\[
\xymatrix{\cC \ar[r]^-{\Lambda} \ar[d] & \cC^{\mathrm{comp}} \ar[d]^{L} \\
\Fun(\Delta^1,\cC^{\mathrm{loc}}) \ar[r]_-{p_1} & \cC^{\mathrm{loc}},}
\]
with $p_1$ denoting the evaluation at $1 \in \Delta^1$. The left vertical map sends an object $X \in \cC$ to the canonical morphism $LX \to L\Lambda X$.
\end{cor}
\begin{proof}
The proof follows formally from \Cref{cor:fracturesquare}; for the details of this argument see \cite[Thm.~5.5]{chromaticfracturecubes} or \cite[Appx.~A.8]{ha}. 
\end{proof}
\color{black}

\begin{rem}
The Tate construction $t_{\cK} = L\Lambda$ associated to the local duality context $(\cC,\cK)$ is defined as the cofiber of the natural map 
\[
\Gamma \simeq \Gamma \Lambda \longrightarrow \Lambda.
\]
This definition agrees with the construction of axiomatic Tate cohomology given in~\cite{greenlees_axiomatic}, as there is a natural equivalence of cofiber sequences
\[\xymatrix{\iExt(\Gamma A,X) \otimes \Gamma A \ar[r] \ar[d]_{\sim} & \iExt(\Gamma A,X) \ar[r] \ar[d]_{\sim} & \iExt(\Gamma A,X) \otimes LA \ar[d]^{\sim} \\
\Gamma \Lambda X \ar[r] & \Lambda X \ar[r] & t_{\cK} X,}\]
where the top row is obtained from the first cofiber sequence in \Cref{thm:hps}(1) by tensoring with $\iExt(\Gamma A,X)$. 

Moreover, Warwick duality~\cite[Cor.~2.5]{greenlees_axiomatic} gives an alternative description of the Tate construction as
\[
t_{\cK}X \simeq L\Lambda X \simeq \Sigma \Delta\Gamma X
\]
for all $X \in \cC$. 
\color{black}
\end{rem}

It is often useful to have a more explicit description of the functors $\Gamma$ and $\Lambda$. To this end, note that~\Cref{lem:locindthick} provides an equivalence $\cC^{\text{tors}} = \Loc_{\cC}^{\otimes}(\cK) \simeq \Ind(\Thick_{\cC}(\overline  \cK))$, where $\overline{\cK}=\{G\otimes K \mid G\in\cG, K\in\cK\}$ and $\cG$ is a set of generators for $\cC$. Since $\Gamma$ is smashing, this implies that there exists a filtered diagram $J$ and a system $(\Kos_j)_{j \in J} \in (\Thick_{\cC}(\overline \cK))^J$ so that 
\begin{equation}\label{eq:torsionkoszul}
\Gamma(X) \simeq \colim_{J} \Kos_j \otimes X,
\end{equation}
for all $X \in \cC$. The diagram $(\Kos_j)_{j \in J}$ will be referred to as a Koszul system; of course, this system is far from being uniquely determined by $\overline \cK$. In many situations, there exist particularly well-behaved Koszul systems which allow for a more explicit description of the completion functor $\Lambda$. 

\begin{defn}
If $(\Kos_j)_{j \in J}$ is a diagram satisfying \eqref{eq:torsionkoszul} such that $\Kos_j$ is self-dual with a fixed shift, i.e., if there exists $n$ such that for all $j \in J$
\[
D\Kos_j \simeq \Sigma^n\Kos_j,
\]
then $(\Kos_j)_{j \in J}$ is called strongly Koszul. 
\end{defn}

Examples of diagrams which are strongly Koszul abound, for example, see \Cref{prop:torsionfunctorspectra}. In this case, the formula for the completion functor reads:

\begin{cor}\label{cor:completion}
If $(\Kos_j)_{j \in J}$ is strongly Koszul for some integer $n$, then 
\[
\Lambda(-) \simeq \lim_J \Sigma^n \Kos_j \otimes -,
\]
where we indentify the system $(\Sigma^n \Kos_j)_{j\in J}$ with the dual of the Koszul system.
\end{cor}
\begin{proof}
Because $\Kos_j \in \Thick_{\cC}(\cK)$ and $\cK \subseteq \cC^{\omega}$, $\Kos_j$ is compact in $\cC$ for all $j$, so the result follows from abstract local duality,
\begin{align*}
\Lambda Y & \simeq \iExt(\Gamma A,Y) \\
& \simeq \iExt(\colim_J \Kos_j,Y) \\
& \simeq \lim_J\iExt(\Kos_j,Y) \\
& \simeq \lim_J D\Kos_j \otimes Y \\
& \simeq \lim_J\Sigma^n\Kos_j \otimes Y,
\end{align*}
where the first isomorphism uses \Cref{thm:hps}(4). 
\end{proof}

Finally, we give a convenient characterization of the local cohomology functor $\Gamma$, which is often useful for identifying it explicitly.

\begin{lem}\label{lem:torsionred}
A smashing functor $F:\cC \to \cC$ with essential image in $\cC^{\mathrm{tors}}$ together with a natural transformation $\phi\colon F \to \Id$ is naturally equivalent to $\Gamma$ if and only if $\phi$ induces equivalences
\[
F(A) \otimes K \simeq K
\]
for all $K \in \cK$.
\end{lem}
\begin{proof}
Since $F$ and $\Gamma$ are smashing, we have $F\Gamma \simeq \Gamma F$, which in turn is equivalent to $F$ by assumption on the essential image of $F$. Suppose that $\phi$ induces equivalences $F(A) \otimes K \simeq K$ for all $K \in \cK$, then also $F\Gamma \simeq \Gamma$, so we obtain a natural equivalence $F \simeq \Gamma$.

Conversely, if $F$ is equivalent to $\Gamma$, then the claim follows from \Cref{thm:hps}.
\end{proof}

We finish this section with a different characterization of the local homology functor in the case $\cC^{\mathrm{tors}}$ is generated by a single compact object $K$. 

\begin{prop}\label{prop:lochombl}
Let $(\cC,K)$ be a local duality context, then the local homology functor $\Lambda = \Lambda^K$ is equivalent to the Bousfield localization\footnote{This functor should not be confused with $L=L_{\mathcal{K}}\colon \cC^{\mathrm{loc}} \to \cC$.} $L_K$ at $K$, i.e., there is a natural equivalence of functors
\[
\xymatrix{L_{K} \ar[r]^-{\sim} & \Lambda^{K}.}
\]
\end{prop}
\begin{proof}
For $X \in \cC$ we have to show that $\Lambda^{K}X$ is $K$-local and that the natural map $\eta\colon X \to \Lambda^{K}X$ is a $K$-equivalence. To see the first claim, suppose $T \in \cC^{\mathrm{loc}}$, i.e., $DK \otimes T \simeq \iHom(K,T) \simeq 0$; then 
\[
\Hom(T,\Lambda X) \simeq \Hom(T,\iHom(\Gamma A,X)) 
\simeq \Hom(\Gamma T,X) \simeq 0, 
\]
so $\Lambda X$ is $K$-local. Since $DK \otimes LA \simeq L(DK) =0$, we get $K \otimes \iHom(LA,X) \simeq \iHom(L(DK),X) \simeq 0$. Therefore, it follows from the fiber sequence
\[
\xymatrix{\iHom(LA,X) \ar[r] & X \ar[r]^-{\eta} & \iHom(\Gamma A,X)\simeq \Lambda X}
\]
that $K \otimes X \simeq K \otimes \iHom(\Gamma A,X) \simeq K \otimes \Lambda X$, hence $K\otimes \eta$ is an equivalence. 
\end{proof}

\subsection{Other duality theorems}

In this section we give a few other abstract duality theorems that follow formally from our framework. We use the same notation as in the previous section. 

In~\cite{dwyer_complete_2002}, Dwyer and Greenlees refine \Cref{thm:hps}(3) by using Morita theory to construct an intermediate category between $\cC^{\text{tors}}$ and $\cC^{\text{comp}}$, at least in the case $\cC$ is compactly generated by its tensor unit $A$. In order to state their result, let $(\cC,\cK)$ be a local duality context with $\cK = \{K\}$ and set $\cE = \End_{\cC}(K)$.

Since $\cC$ is stable, it is enriched, tensored, and cotensored over the stable category of spectra $\Sp$. To distinguish this structure from the closed symmetric monoidal structure on $\cC$, we denote the tensor by $M \wedge C$ and the cotensor by $C^M$, where $M\in \Sp$ and $C \in \cC$. Moreover, there is a natural $\cE$-action on $K$, so $E(-)=\Hom_{\cC}(K,-)$ lifts to a functor with values in $\Mod_{\cE}$, the stable category of left $\cE$-modules in spectra. By the adjoint functor theorem, $E$ admits a left adjoint $\Gamma'(-)  = (-) \wedge_{\cE} K\colon \Mod_{\cE} \to \cC^{\text{tors}}$, constructed from $- \wedge K$ by taking the $\cE$-action into account; for the details, see \cite{schwedeshipley_modules} and \cite{dwyer_complete_2002}. Similarly, there is a right adjoint $\Lambda'(-) = \Hom_{\cE}(DK, -)$ to $E(-)\colon \cC^{\text{tors}} \to \Mod_{\cE}$.

\begin{thm}[Dwyer--Greenlees]
Let $(\cC,K)$ be a local duality context and $\cE = \End_{\cC}(K)$. There is a commutative diagram of adjunctions
\[
\xymatrix{& \cC \ar@<-0.5ex>[ld]_{\Gamma} \ar@<0.5ex>[rd]^{\Lambda} \ar[d]^E \\
\cC^{\mathrm{tors}} \ar@<-0.5ex>[ru] \ar@<-0.5ex>[r]_E & \Mod_{\cE} \ar@<-0.5ex>[r]_{\Gamma'} \ar@<-0.5ex>[l]_{\Lambda'} & \cC^{\mathrm{comp}}, \ar@<0.5ex>[lu] \ar@<-0.5ex>[l]_E}
\]
where all horizontal functors are equivalences, and are given as follows:
\[\xymatrix{E(X) = \Hom_{\cC}(K,X) & \Gamma'(M)  = M \wedge_{\cE} K & \Lambda'(M) = \Hom_{\cE}(DK, M),}\]
equipped with the obvious module structures.
\end{thm}
\begin{proof}
The proof is a straightforward adaptation of the argument in~\cite{dwyer_complete_2002}. An $\infty$-categorical version of the derived Morita theory of Schwede and Shipley~\cite{schwedeshipley_modules} can be found in \cite[Thm.~7.1.2.1, Rem.~7.1.2.3]{ha}.
\end{proof}

\begin{rem}
Using the techniques of~\cite{schwedeshipley_modules}, it is possible to generalize the above result to cover the case where $\cK$ is a set of compact objects of $\cC$ as well. 
\end{rem}

In \cite{hartshorne_affine_1969}, Hartshorne shows that the classical local duality theorems of Grothendieck \cite{grothendieck_localcohomology} can be lifted to the derived category. It turns out that the arguments are completely formal; in the following, we present a generalization of it to any local duality context $(\cC,\cK)$. We start with a preliminary lemma. 
\begin{lem}\label{lem:torscomplihom}
Let $\Gamma$ and $\Lambda$ be the torsion and completion functor of some fixed local duality context $(\cC,{\cK})$. If $X,Y \in \cC$, then there are natural equivalences:
\begin{enumerate}
 \item $\Lambda\iHom(X,Y) \simeq \iHom(X,\Lambda Y)$.
 \item $\Gamma\iHom(X,Y) \simeq \iHom(X,\Gamma Y)$, if $X$ is compact. 
\end{enumerate}
\end{lem}
\begin{proof}
By local duality and \Cref{lem:homihom}, we have equivalences
\begin{align*}
\Hom(W,\Lambda\iHom(X,Y)) & \simeq \Hom(\Gamma W, \iHom(X,Y)) \\
& \simeq \Hom(\Gamma (W \otimes X),Y) \\
& \simeq \Hom(W \otimes X, \Lambda Y) \\
& \simeq \Hom(W,\iHom(X,\Lambda Y)),
\end{align*}
for any $W \in \cC$, hence $\Lambda\iHom(X,Y) \simeq \iHom(Y,\Lambda Y)$. For the second claim, the assumption on $X$ and the fact that $\Gamma$ is smashing give rise to equivalences
\[
\Gamma\iHom(X,Y) \simeq \Gamma (DX \otimes Y) \simeq DX \otimes \Gamma Y \simeq \iHom(X,\Gamma Y). 
\]
\end{proof}

Fix an object $\Omega \in \cC$, thought of as a dualizing complex. We then define the duality functor $D_{\Omega}$ with respect to ${\Omega}$ as 
\[
D_{\Omega}\colon \cC^{\mathrm{op}} \longrightarrow \cC,\ \ X \mapsto \iHom(X,\Omega).
\]
Furthermore, define the $\cK$-dual of any $X \in \cC$ to be
\[ 
D_{\cK}X = \iHom(X,\Gamma \Omega);
\]
this construction implicitly depends on $\Omega$, but following the literature, this dependence is omitted from the notation. The next corollary recovers and generalizes the affine duality results proven in \cite[Thm.~4.1]{hartshorne_affine_1969} and \cite[Sec.~(5.2)]{local_cohom_schemes}.

\begin{thm}[Abstract affine duality]
Suppose $(\cC,\cK)$ is a local duality context. If $X\in \cC^{\omega}$ is compact, then there is a natural equivalence
\[
\xymatrix{\Lambda D_{\Omega}^2X \ar[r]^{\sim} & D_{\cK}^2X.}
\]
Furthermore, if the dualizing object $\Omega$ is compact, then the left hand side simplifies to $\Lambda X \simeq \Lambda D_{\Omega}^2X$.
\end{thm}
\begin{proof}
First note that $\Gamma D_{\Omega}X = \Gamma \iHom(X,{\Omega}) \simeq \iHom(X,\Gamma {\Omega}) \simeq  D_{\cK}X$ by \Cref{lem:torscomplihom}(2). We thus get:
\begin{align*}
\Lambda D_{\Omega}^2X & \simeq \iHom(D_{\Omega}X,\Lambda {\Omega}) & & \text{by \Cref{lem:torscomplihom}(1)} \\
& \simeq \iHom(\Gamma D_{\Omega}X, {\Omega}) & & \text{by~\Cref{thm:hps}} \\
& \simeq \iHom(\Gamma D_{\Omega}X,\Gamma {\Omega}) \\
& \simeq \iHom(D_{\cK}X,\Gamma {\Omega}) \\
& \simeq D_{\cK}^2X.
\end{align*}
If ${\Omega}$ is compact, then $D_{\Omega}^2 \simeq \mathrm{id}$, hence the claim.
\end{proof}

Another reformulation of local duality, originally due to Hartshorne~\cite[Thm.~2.1]{hartshorne_affine_1969}, can be generalized as follows:

\begin{cor}\label{cor:der_groth_dual_II}
Let $(\cC,\cK)$ be a local duality context, then there is a natural equivalence	
	\[
	\xymatrix{\Gamma X \ar[r]^-{\sim} & \iHom(DX,\Gamma A),}
	\]
for all compact objects $X\in \cC$.  
\end{cor}
\begin{proof}
	Suppose that $X$ is compact. By assumption on $\cC$, $X$ compact implies $X$ dualizable by~\Cref{lem:algcompdual}, and so we have a chain of natural equivalences
\begin{align*}
	\Gamma X\simeq \iHom(A,\Gamma X)\simeq & \iHom(A,X\otimes \Gamma A)\\
					\simeq & \iHom(A\otimes DX, \Gamma A)\\
					\simeq & \iHom(DX,\Gamma A).
\end{align*}
\end{proof}

\section{Modules}\label{sec:modules}

In this section, we study local duality between torsion and completion functors in the setting of modules over rings and module spectra over commutative ring spectra. In these cases, we recover the familiar notions of local cohomology and local homology, which go back to Grothendieck~\cite{grothendieck_localcohomology}, and their subsequent generalizations to the topological setting by Greenlees and May~\cite{gm_localhomology,greenleesmay_completions} and Lurie \cite[Section 4]{dag12}. Instead of working in the derived category like Dwyer and Greenlees do \cite{dwyer_complete_2002}, we will use the stable category of chain complexes as defined in \cite{ha}. In this framework, the local duality spectral sequence of Grothendieck is obtained as a composite functor spectral sequence. Similarly, we give alternative deductions of some results of Greenlees and May about derived functors of completion for complete regular local Noetherian rings.

Many of the results in this section are not new, but we hope that the unified formulation will illuminate existing treatments in the literature. Furthermore, this section serves as a toy example for the discussion of local duality in the setting of $E_*E$-comodules and the $E$-local category.

\begin{conv}
	We follow \cite{broadmann_sharp} when working in the graded case. Let $A$ be a commutative $\Z$-graded ring, and $I$ an ideal generated by a homogeneous sequence $\{x_1, x_2,\dots,x_n\}$. Let $\Mod^*_A$ denote the category of $\Z$-graded $A$-modules. As shown in \cite[Prop.~2.5]{abrams_categorical_1996}, $\Mod^*_R$ is a Grothendieck category. Moreover, using the graded versions of the usual functors, we can see this is a closed symmetric monoidal category. Thus, we can work in the stable category of chain complexes of graded modules; we then use the unique way of grading the functors so that when composing with the forgetful functor $\Mod^*_A\to \Mod_A$ we get their familiar non-graded versions back \cite[Ch.~13]{broadmann_sharp}. While not explicit in the notation, all functors are thus bigraded in this sense. 
\end{conv}

\subsection{Local duality for module spectra}\label{sec:modulespectra}

We start by discussing local cohomology and local homology for commutative ring spectra; this closely follows~\cite{greenleesmay_completions}, so we shall be brief. Suppose $A$ is an $\mathbb{E}_{\infty}$-ring spectrum\footnote{In fact, an $\mathbb{E}_{2}$-structure on $A$ would be sufficient for our purposes.} and let $\Mod_A$ be the symmetric monoidal stable category of $A$-module spectra. In the following, we will often refer to a module spectrum simply as a module and write the smash product in $\Mod_A$ as $\otimes_A$ or simply $\otimes$ if no confusion is likely to arise. In this section, we can work with $\Hom$ throughout since the internal hom-object is just an enrichment of the categorical one.

 Consider $I \subseteq \pi_*A$ a finitely generated homogeneous ideal and choose a minimal set of generators $\{x_1,\ldots,x_n\}$ for $I$; it can be checked easily that the constructions and results of this section are independent of this choice.

\begin{defn}\label{defn:koszul}
If $x \in \pi_iA$ is represented by an $A$-linear map $x\colon \Sigma^iA \to A$, then we define an $A$-module $\Kos(x) = \fib(A \to A[x^{-1}])$, where $A[x^{-1}]$ is the telescope of $x$, i.e., the colimit of the diagram
\[\xymatrix{A \ar[r]^-x & \Sigma^{-i}A  \ar[r]^-x & \Sigma^{-2i}A \ar[r]^-x & \ldots.}\] 
The Koszul object $\Kos(I)$ of the ideal $I = (x_1,\ldots,x_n)$ is then constructed as 
\[\Kos(I) = \bigotimes_{i=1}^n \Kos(x_i).\]
Completing the natural map to a cofiber sequence
\[\xymatrix{\Kos(I) \ar[r] & A \ar[r] & \check C(I)}\]
defines the {\v C}ech object $\check C(I)$ of $I$. 
\end{defn}

\begin{lem}\label{lem:koszulfiltration}
Let $\Sigma^{-1}A/x_i^s$ be the fiber of the map $ A \xr{x_i^s} \Sigma^{-k_i s}A$, where $k_i$ is the degree of $x_i$. Then there is a natural equivalence 
\[\xymatrix{\colim \left ( \bigotimes_{i=1}^n \Sigma^{-1}A/x_i^s \right ) \ar[r]^-{\sim} & \Kos(I)}\]
of $A$-modules.
\end{lem}
\begin{proof}
This is~\cite[Lem.~3.6]{greenleesmay_completions}.
\end{proof}
\begin{rem}
	A previous version of this paper had an incorrect version of this lemma. As a consequence, the formulas given in \cite[Cor.~3.13 and Lem.~3.17]{bhv2} have an incorrect suspension. The correct formulas are given below in \Cref{thm:ringlocdual}. This does not affect any of the main results in \cite{bhv2}. 
\end{rem}
It is hence natural to write $\Kos_s(I)= \bigotimes_{i=1}^n \Sigma^{-1}A/x_i^s$ for any $s \in \N$.  We call $(\Kos_s(I))_{s \ge 0}$ a Koszul system, terminology that will be justified by~\Cref{thm:ringlocdual}.  

\begin{prop}\label{prop:modulestronglykoszul}
	The system $(\Kos_s(I))_{s}$ is self-dual with $D\Kos_s(I)\simeq\Sigma^{n-sd} \Kos_s(I)$, $s\in\N$, where $d=k_1+k_2+\cdots + k_n$.
\end{prop}
\begin{proof}
	It suffices to show that $DA/x_i^s\simeq \Sigma^{-k_is-1} A/x_i^s$, which follows immediately from comparing the cofiber sequence
	\[
	\Sigma^{k_is} A \xrightarrow{x_i^s}  A\to A/x^s_i
	\]
to its dual.
\end{proof}

In the language of \cref{sec:abstract}, we now consider the local duality context $(\Mod_A, \Kos_1(I))$. Let $A/I=\bigotimes_{i=1}^n A/x_i$. Note that since $\Kos_1(I) \simeq \Sigma^{-n} A/I$ we could have taken $(\Mod_A, A/I)$ as our local duality context. We will see that in the case of modules over a discrete ring $A$, $A/I$ as a chain complex concentrated in degree zero is not in general equivalent to $\Kos_1(I)$ unless regularity conditions are imposed on $I$. Thus, we prefer to use the Kozsul object to avoid any confusion.

Following~\cite{greenleesmay_completions}, we can thus construct a category of $I$-torsion module spectra; this terminology will be justified in \cref{rem:torsion_just}. 
\begin{defn}
Define the category $\Mod_A^{I-\mathrm{tors}}$ of $I$-torsion $A$-modules as
\[\Mod_A^{I-\mathrm{tors}} = \Locid{\Mod_A}(\Kos_1(I))=\Loc_{\Mod_A}(\Kos_1(I)),\]
the localizing subcategory of $\Mod_A$ generated by the compact object $\Kos_1(I)$. The inclusion of the category $\Mod_A^{I-\text{tors}}$ of $I$-torsion $A$-modules into $\Mod_A$ will be denoted $\iota_{\text{tors}}$. 
\end{defn}
\color{black}

Since $\Mod_A^{\mathrm{tors}}$ is compactly generated, $\iota_{\text{tors}}$ admits a right adjoint $\Gamma_I$ 
\[
\xymatrix{\iota_{\mathrm{tors}}\colon \Mod_A^{I-\mathrm{tors}} \ar@<0.5ex>[r] & \Mod_A\colon \Gamma_I.\ar@<0.5ex>[l]}
\] 
Therefore, we can apply the results of Section~\ref{sec:ald} to obtain localization and completion adjunctions 
\[\xymatrix{\Mod_A^{I-\mathrm{loc}} \ar@<0.5ex>[r]^-{\iota_{\mathrm{loc}}} & \Mod_A \ar@<0.5ex>[l]^-{L_I} \ar@<0.5ex>[r]^-{\Lambda^I} & \Mod_A^{I-\mathrm{comp}} \ar@<0.5ex>[l]^-{\iota_{\mathrm{comp}}}
}\]
with respect to the ideal $I$, as well as a version of local duality. The first part of the next result partially recovers~\cite[Thm.~4.2, Thm.~5.1]{greenleesmay_completions}, and the second one is implicit in~\cite[Thm.~2.1]{dwyer_complete_2002}. 

\begin{thm}\label{thm:ringlocdual}
Let $A$ be a ring spectrum as above and consider a finitely generated homogeneous ideal $I \subseteq \pi_*A$. 

\begin{enumerate}
 \item The functors $\Gamma_I$ and $L_I$ are the smashing localization and colocalization functors associated to the fiber sequence
 \[\xymatrix{\Kos(I) \ar[r] & A \ar[r] & \check C(I)}\]
 and $\Lambda^I$ is the localization functor with category of acyclics given by $\Mod_A^{I-\mathrm{loc}}$. Moreover, $\Gamma_I$ and $\Lambda^I$ satisfy $\Lambda^I \circ \Gamma_I \simeq \Lambda^I$ and $\Gamma_I \circ \Lambda^I \simeq \Gamma_I$ and they induce mutually inverse equivalences
 \[\xymatrix{\Mod_A^{I-\mathrm{comp}} \ar@<1ex>[r]^{ \Gamma_I}_{\sim} & \Mod_A^{I-\mathrm{tors}}. \ar@<1ex>[l]^{\Lambda^I}}\]
 \item For $M,N \in \Mod_A$, there is an equivalence
 \[\xymatrix{\Hom(\Gamma_IM,N) \ar[r]^-{\sim} & \Hom(M,\Lambda^IN),}\]
 natural in each variable. 
  \item For $M \in \Mod_A$, there are natural equivalences
  \[
\Gamma_I M \simeq \colim_s \Kos_s(I) \otimes M,
  \] 
  and
	\[\Lambda^I M \simeq \lim_s \Sigma^{n-sd}\Kos_s(I) \otimes M,\]
with $d$ as in \cref{prop:modulestronglykoszul}.
\end{enumerate}
\end{thm}
\begin{proof}
The identification of $\Gamma_I$ and $L_I$ with the Koszul and \v Cech objects follows from \cite[Thm.~5.1(i)]{greenleesmay_completions}, while (3) is a result of the identification of $\Gamma_I$ and $\Kos(I)$, and \cref{prop:modulestronglykoszul} together with the proof of \Cref{cor:completion}. The rest is immediate from~\cref{thm:hps}. 
\end{proof}

The situation is summarized in the diagram on the left  of \eqref{eq:modulediagram}.

\begin{rem}\label{rem:torsion_just}
	$\Mod_A^{I-\text{tors}}=\Loc(\Kos_1(I))$ is the category of acyclics of $L_I$, or equivalently, the essential image of $\Gamma_I$. It follows by \cref{prop:homotopy_ss} that this is precisely the category of module spectra whose homotopy groups are $I$-torsion as $\pi_*A$-modules, thus justifying its name. 
\end{rem}

 This completes our discussion of local duality for the category of $A$-module spectra. In the next subsection, we specialize further to the case of a discrete commutative ring, and give a comparison between the topological and the algebraic context.

\subsection{Local duality for \texorpdfstring{$A$}{A}-modules}

Let $A$ be a commutative ring, and $I$ a finitely generated ideal with a fixed set of generators $\{x_1,x_2,\dots, x_n\}$. We denote by $\Mod_A$ the abelian category of $A$-modules, and $\Hom_A(-,-)$ the set of $A$-module maps.\footnote{Unfortunately this conflicts with the previous section, but it should be clear from context if we are working in the derived or underived setting.} We will often work in the stable category $\D_A$ of chain complexes as introduced in \cite[Def.~1.3.5.8]{ha}. This is just the differential graded nerve of the subcategory of fibrant objects in the injective model structure on $\mathrm{Ch}_A$. The space of maps in the derived category will be denoted by $\Hom(-,-)$ to distinguish it from the set of $A$-module morphisms which we always write as above. 

We will proceed by building upon local duality for module spectra, using Shipley's work on the relation between the category $\Mod_{HA}$ of module spectra over the Eilenberg--MacLane spectrum $HA$ and $\Ch(A)$. The key difference in this case is the existence of the abelian category $\Mod_A$ of $A$-modules, whose corresponding derived category is $\cD_A$. Consequently, we can and will compare the abstractly constructed local cohomology and local homology functors on $\cD_A$ to the derived functors of torsion and completion, thereby relating our version of local duality to the one studied in the literature. Let
\[
\xymatrix{\Theta\colon \Mod_{HA} \ar[r]^-{\sim} & \D_A}
\]
be the symmetric monoidal equivalence induced between the stable category of $HA$-module spectra and the derived category of $A$; see~\cite{shipley_hz} and~\cite[Rem.~7.1.1.16]{ha} for details. By a mild abuse of notation, we will also denote the object $\Theta(\Kos_1(I))$ in $\D_A$ by $\Kos_1(I)$, where the the module spectrum $\Kos_1(I)$ is defined in the previous subsection. 
Then, $\Kos_1(I)$ is a compact object in $\D_A$ and \Cref{thm:hps} applies to give a pair of adjoints
\[
\xymatrix{\iota_{\mathrm{tors}}\colon \D_A^{I-\mathrm{tors}} \ar@<0.5ex>[r] & \D_A \colon \Gamma_I,\ar@<0.5ex>[l].}
\]
where $\D_A^{I-\mathrm{tors}}$ is the localizing subcategory of $\D_A$ compactly generated by $\Kos_1(I)$, and $\Gamma_I$ is the resulting right adjoint to the inclusion $\iota_{\mathrm{tors}}$ of $\D_A^{I-\mathrm{tors}}$ in $\D_A$. Likewise, we obtain localization and completion adjunctions 
\[\xymatrix{\D_A^{I-\mathrm{loc}} \ar@<0.5ex>[r]^-{\iota_{\mathrm{loc}}} & \D_A \ar@<0.5ex>[l]^-{L_I} \ar@<0.5ex>[r]^-{\Lambda^I} & \D_A^{I-\mathrm{comp}} \ar@<0.5ex>[l]^-{\iota_{\mathrm{comp}}}.
}\]

We will see that all the functors above can be interpreted in terms of local (co)homology, thus recovering well known results. Let us begin by showing that $\Kos_1(I)$ is the usual cochain complex in $\D_A$.

\begin{prop}\label{prop:koszul_D_R}
	There is an equivalence $\Kos_1(I)\simeq \bigotimes_{i=1}^n (A\xrightarrow{x_i} A)$ in $\D_A$.	
\end{prop}
\begin{proof}
	By definition, the module spectrum $\Kos_1(I)$ is given by $\bigotimes_{i=1}^n\fib( HA\xrightarrow{x_i} HA)$. Since the functor $\Theta$ constructed in~\cite{shipley_hz} is a symmetric monoidal equivalence of stable categories, the result follows.
\end{proof}

Motivated by this we let $\Kos_s(I)=\bigotimes_{i=1}^n (A\xrightarrow{x_i^s} A)$ and $\Kos(I)=\colim_s \Kos_s(I) = \bigotimes_{i=1}^n (A\xrightarrow{} A[x_i^{-1}])$.

\begin{rem}
	There is an important difference between working with $A$-modules and module spectra. As mentioned before, in module spectra both $\Kos_1(I)$ and $HA/I$ are equivalent and thus give the same duality context. However, if we consider $A/I$ as a complex concentrated in degree zero in $\D_A$, this is not in general compact so in particular it is not equivalent to $\Kos_1(I)$. If $I$ is generated by a regular sequence, then it is known that $\Kos_1(I) \simeq \Sigma^n A/I$. Even if $I$ is not regular the two objects still generate the same duality context; this follows as $\Loc(A/I)\simeq \Loc(\Kos_1(I))$, see \cite[Prop. 6.1]{dwyer_complete_2002}.
\end{rem}

In order to identify the abstract functors above it suffices to show that cofiber sequence
\[
  \xymatrix{\Kos(I) \ar[r] & A \ar[r] & \check C(I)}
\]
is indeed $\Gamma_I A \longrightarrow A \longrightarrow L_I A$. This is equivalent to showing that $\Kos(I)$ is in $\Loc(\Kos_1(I))$ (i.e., $L_I$-acyclic) while $\check C(I)$ is $L_I$-local, facts that can be found in~\cite[Sec.~5]{greenlees_axiomatic}.

We are now ready to state our local duality theorem for the derived category of $A$.

\begin{thm}\label{thm:main_D_R}
Let $A$ be a ring as above and consider a finitely generated ideal $I \subseteq A$. 
\begin{enumerate}
 \item The functors $\Gamma_I$ and $L_I$ are the smashing localization and colocalization functors associated with the cofiber sequence
 \[\xymatrix{\Kos(I) \ar[r] & A \ar[r] & \check C(I)}\]
 and $\Lambda^I$ is the localization functor with category of acyclics given by $\D_A^{I-\mathrm{loc}}$.
 
 \item In addition, $\Gamma_I$ and $\Lambda^I$ satisfy $\Lambda^I \circ \Gamma_I \simeq \Lambda^I$ and $\Gamma_I \circ \Lambda^I \simeq \Gamma_I$ and they induce mutually inverse equivalences
 \[\xymatrix{\D_A^{I-\mathrm{comp}} \ar@<1ex>[r]^-{\Gamma_I}_{\sim} & \D_A^{I-\mathrm{tors}}. \ar@<1ex>[l]^-{ \Lambda^I}}\]
 \item For $X,Y \in \D_A$, there is an equivalence
 \[\xymatrix{\Hom(\Gamma_I X,Y) \ar[r]^{\sim} & \Hom(X,\Lambda^I Y),}\]
 natural in each variable.
 \item For $X \in \D_A$, there is a natural isomorphism  
	  \[
\Gamma_I X \simeq \colim_s \Kos_s(I) \otimes X,
  \] 
  and
	\[\Lambda^I X \simeq \lim_s \Sigma^n\Kos_s(I) \otimes X.\]

\end{enumerate}
\end{thm}
\begin{proof}
The identification of the cofiber sequence in (1) is contained in \cite[Sec.~5]{greenlees_axiomatic}. The rest of (1)--(3) are immediate from~\Cref{thm:hps}, while the equivalence in (4) follows from~\Cref{cor:completion} and~\Cref{prop:modulestronglykoszul}.
\end{proof}

We now work in the abelian category $\Mod_A$.
\begin{defn}
	For each $M \in \Mod_A$, we define the $I$-power torsion of $M$ as
	\[
	T_I^A M= \{x\in M| I^k x=0 \text{ for some }k\in \N \}.
	\]
\end{defn}
When the ring is clear from the context, it will be omitted from the notation. It is easy to see that the $I$-power torsion functor for modules can be identified with
\begin{equation}\label{eq:abeliantorsfunctorcolimitdescription}
T_I^A(M) \cong \colim_k \Hom_A(A/I^k,M);
\end{equation}
we refer the reader to \cite[Ch.~1]{broadmann_sharp} for a nice exposition on $I$-power torsion including this result.

	In \cref{prop:homology_ss}, we will see that the homology of $\Gamma_I M$ can be computed in terms of local cohomology; in particular, $H^*(\Gamma_I M)=H^*_I(M)$ for any module $M$ viewed as a complex concentrated in degree zero. It is then natural to ask when $\Gamma_I$ is in fact the total right derived functor of the $I$-power torsion functor $T_I$ defined at the abelian level. A complete answer to this question was given in \cite{schenzel_proreg}. First, we recall the following definition.

\begin{defn}
	An ideal $I$ is generated by a weakly proregular sequence $\{x_1, x_2,\dots,x_n\}$ if for all $k\geq 0$ there exists an $m \geq k$ such that
	\[
	\Kos_m(I)\to \Kos_k(I) 
	\]
	is the zero map. 
\end{defn}
\begin{exmp}
	If $A$ is Noetherian, then every ideal in $A$ is weakly proregular, see \cite[Thm.~3.34]{psy}. 
\end{exmp}
Here it is important to remember that $\Kos_s(I)$ depends on the choice of generators of $I$, unlike $\Kos(I)$ which does not.

Let $a\in A$, then an $A$-module $M$ is said to be $a$-bounded torsion if the increasing sequence $\ker(A\xrightarrow{a^k}A)$ stabilizes. Assume that $A$ is $x_i$-bounded torsion for all $i=1,\dots,n$, and let $\mathbf{L}$ and $\mathbf{R}$ indicate a total left and right derived functor, respectively. Schenzel shows that $I$ is generated by a weakly proregular sequence if and only if there is a natural equivalence $\mathbf{R}T_I X \simeq \Kos(I) \otimes X$ in the derived category \cite[Thm.~1.1]{schenzel_proreg}. Moreover, $\mathbf{L}C^I X \simeq \Hom(\Kos(I), X)$ if and only if $I$ is weakly proregularly generated, where $C^I$ denotes the $I$-adic completion functor. With this and the formulae in \cref{thm:main_D_R} under consideration, we get 
\begin{prop}\label{prop:gamma=derived_torsion}
	Let $A$ be a commutative ring and $I=(x_1,x_2,\dots,x_n)$ an ideal in $A$. Suppose that $A$ is bounded $x_i$-torsion for $i=1,\dots,n$. Then the following are equivalent:
	\begin{enumerate}
		\item $\{x_1, x_2,\dots,x_n\}$ is weakly proregular.
		\item There is a natural equivalence $\Gamma_I\simeq \mathbf{R}T_I$.
		\item There is a natural equivalence $\Lambda^I\simeq \mathbf{L}C^I$.
	\end{enumerate}
\end{prop}

Thus, the functors $\Gamma_I$ and $\Lambda^I$ which satisfy local duality for any finitely generated ideal $I$ agree with the total derived functors of classical torsion and completion provided $I$ is sufficiently regular.

\begin{rem}
This result exhibits the central difference between our approach to local duality and the one usually taken: The classical approach starts with explicit torsion and completion functors and then shows that they are related by local duality in certain cases. In contrast, we abstractly construct local cohomology and local homology functors satisfying local duality and then identify these functors with the usual ones under mild conditions. Therefore, we are able to produce well behaved generalizations of local cohomology and local homology in those cases not covered by the literature.
\end{rem}

Using the proposition above, we get the following identification.

\begin{prop}\label{prop:Gamma=colimHom}
	Let $A$ be a commutative ring and $I=(x_1,x_2,\dots,x_n)$ a weakly proregularly generated ideal. Suppose that $A$ is bounded $x_i$-torsion for $i=1,\dots,n$. Then, there is an equivalence of functors in $\D_A$
	\[
	\Gamma_I\simeq \colim_k \Hom(A/I^k,-).
	\]
\end{prop}
\begin{proof}
	Using that $T_I(-)\cong \colim_k\Hom_A(A/I^k,-)$, and the fact that colimits are exact in $\Mod_A$, we conclude that the total right derived functor is given by
	\[
	\mathbf{R}T_I \simeq \colim_k\Hom(A/I^k,-)
	\]
	in the derived category. Thus, the claim follows by \cref{prop:gamma=derived_torsion}.
\end{proof}
Next, we give interpretations of all the functors involved in terms of the classical notion of local (co)homology and \v Cech cohomology. Recall that the $s$-th local cohomology of an $A$-module $M$ is given by
\[
	H^s_I(M)=H^s(\Kos(I)\otimes M).
\]
Dually, the $s$-th local homology of $M$ was defined in~\cite{gm_localhomology} as 
\[
	H^I_s(M)=H_s(\Hom(\Kos(I), M)).
\]

We can also define local \v Cech cohomology by 
\[
  \CH^s_I(M)=H^s(\check C(I)\otimes M). 
\]

As noted by Greenlees and May in~\cite{greenleesmay_completions}, there is a filtration on $\Kos(x_i)$ induced by \Cref{lem:koszulfiltration}; the following spectral sequences then arise from the resulting product filtration of the Koszul and \v Cech objects after applying $\pi_*$.

\begin{prop}\label{prop:homotopy_ss}
	Let $A$ be a ring spectrum and $I$ a finitely generated homogeneous ideal of $\pi_*A$. Let $M\in \Mod_{A}$. There are strongly convergent spectral sequences of $\pi_*A$-modules:
	\begin{enumerate}
		\item $E^2_{s,t}=(H^{-s}_I(\pi_*M))_{t} \implies \pi_{s+t}(\Gamma_I M)$, with differentials $d^r\colon E^r_{s,t} \to E^r_{s-r,t+r-1}$,
		\item $E^2_{s,t}=(H_{s}^I(\pi_*M))_{t} \implies \pi_{s+t}(\Lambda^I M)$, with differentials $d^r\colon E^r_{s,t} \to E^r_{s-r,t+r-1}$, and
		\item $E^2_{s,t}=(\CH^{-s}_I(\pi_*M))_{t} \implies \pi_{s+t}(L_I M)$, with differentials $d^r\colon E^r_{s,t} \to E^r_{s-r,t+r-1}$.
	\end{enumerate}
\end{prop}

By specializing to the ring spectrum $HA$ for a given commutative ring $A$, we can deduce similar spectral sequences converging to the homology of the corresponding functors in the derived category. Alternatively, these are easy to obtain directly using the formulae from \Cref{thm:main_D_R}, and the first two have previously appeared in \cite{dwyer_complete_2002}.

\begin{prop}\label{prop:homology_ss}
		Let $A$ be a commutative ring and $I$ a finitely generated ideal. Let $X \in \D_A$. There are strongly convergent spectral sequences of $A$-modules:
	\begin{enumerate}
		\item $E_2^{s,t}=H^{s}_I(H^tX) \implies H^{s+t}(\Gamma_I X)$, 
		\item $E^2_{s,t}=H_{s}^I(H_tX) \implies H_{s+t}(\Lambda^I X)$, and
		\item $E_2^{s,t}=\CH^{s}_I(H^t X) \implies H^{s+t}(L_I X)$.
	\end{enumerate}
\end{prop}

The situation in~\Cref{prop:homotopy_ss} can be visualized in the following diagram:

\begin{equation}\label{eq:modulediagram}
\xymatrix{& \Mod_{HA}^{I-\mathrm{loc}} \ar@<0.5ex>[d] \ar@{-->}@/^1.5pc/[ddr] & & & \D_{A}^{I-\mathrm{loc}} \ar@<0.5ex>[d] \ar@{-->}@/^1.5pc/[ddr] \\
& \Mod_{HA} \ar@<0.5ex>[u]^{L_{I}} \ar@<0.5ex>[ld]^{\Gamma_{I}} \ar@<0.5ex>[rd]^{\Lambda^I} & \ar@{~>}[r]^{\pi_*} & & \D_{A} \ar@<0.5ex>[u]^{L_{I}} \ar@<0.5ex>[ld]^{\Gamma_{I}} \ar@<0.5ex>[rd]^{\Gamma^I} \\
\Mod_{HA}^{I-\mathrm{tors}} \ar@<0.5ex>[ru] \ar[rr]_{\sim} \ar@{-->}@/^1.5pc/[ruu] & &\Mod_{HA}^{I-\mathrm{comp}} \ar@<0.5ex>[lu] &  \D_{A}^{I-\mathrm{tors}} \ar@<0.5ex>[ru] \ar[rr]_{\sim} \ar@{-->}@/^1.5pc/[ruu] & & \D_{A}^{I-\mathrm{comp}}. \ar@<0.5ex>[lu]}
\end{equation}

\subsection{Classical local duality} 
\label{sub:applications}

In this section, we show how our framework allows us to effortlessly recover important duality theorems found in the literature. First, we deduce the Greenlees--May spectral sequence introduced in \cite[Prop.~3.1]{gm_localhomology}, often simply referred to as Greenlees--May duality.

\begin{thm}[Greenlees--May]\label{thm:gmdualityclassical}
	For any $M \in \cD_A$ there is a third quadrant spectral sequence
	\[
	E_2^{s,t}=\Ext_A^s(H^{-t}_I(A),M) \implies H^I_{-t-s}(M), 
	\]
	with differentials $d_r\colon E_r^{s,t} \to E_r^{s+r,t-r+1}$.
\end{thm}
\begin{proof}
	By local duality, we have an equivalence $\Lambda^I M \simeq \Hom(\Gamma_I A, M)$. Thus, the spectral sequence above is just the hypercohomology spectral sequence corresponding to the composition of the functor $\Gamma_I$ with $\Hom(-, M)$.
\end{proof}

The original local duality theorem of Grothendieck appeared in \cite{grothendieck_localcohomology} as lecture notes from a seminar given by Grothendieck at Harvard in 1961. Our next result recovers Grothendieck local duality in its more general form.

\begin{thm}[\textbf{Grothendieck local duality}, {\cite[Thm.~6.8]{grothendieck_localcohomology}}]\label{thm:glocalduality}
	Let $A$ be a ring with $H^q_I(A)=0$ for $q>n$, and write $D_Q=\Hom_A(-,Q)$ for the duality functor defined by an injective module $Q$. Then, for any $M \in \Mod_A$ there is a convergent first quadrant spectral sequence
	\[
	E_2^{p,q}=\Ext_A^p(M,D_Q(H^{n-q}_I A)) \implies D_Q(H_I^{n-p-q}(M)).
	\]

\end{thm}
\begin{proof}
Since $\Gamma_I$ is smashing, we can write
 \begin{equation}\label{eq:gldauxeq}
 \Hom(M,\Hom(\Gamma_I A,Q)) \simeq \Hom(\Gamma_I M,Q),
 \end{equation}
where $M$ is an $A$-module thought of as a complex concentrated in degree zero in $\D_A$. The desired spectral sequence is the hypercohomology spectral sequence for the composition of $\Gamma_I$ and $\Hom(M,\Hom(-,Q))$, with abutment given by the cohomology of the left hand side in the equivalence of \eqref{eq:gldauxeq}, where the identification of the $E_2$-term uses the injectivity of $Q$. By \cite[5.7.9]{weibel_homological} this spectral sequence is always weakly convergent, and the existence of a vanishing line implies that it is also strongly convergent. 
\end{proof}
\begin{rem}
	This result can also be deduced from Greenlees--May duality \Cref{thm:gmdualityclassical}, see \cite[Prop.~3.8]{gm_localhomology}; in contrast, our proof only uses that $\Gamma_I$ is smashing. 
\end{rem}

	Grothendieck's original statement was written for $A$ a complete local Noetherian ring of dimension $n$, and $Q$ a dualizing module for $A$. Recall that the Krull dimension of a ring $A$, $\dim(A)$, is the length of its longest strictly ascending chain of prime ideals, while its $I$-depth is the length of a maximal regular sequence contained in $I$. A local Cohen--Macaulay ring is a Noetherian local ring with Krull dimension equal to its $\m$-depth, where $\m$ is the maximal ideal of $A$. A local Cohen--Macaulay ring $A$ is called Gorenstein if its injective dimension as an $A$-module is finite. If $A$ is Gorenstein, $H^n_\m(A)$ is known to be isomorphic to the injective hull $E(k)$ of the residue field $k=A/\m$ of $A$~\cite[Lem.~11.2.3]{broadmann_sharp}.
	Under these hypotheses, $E(k)$ is a dualizing module for $A$ by \cite[Prop.~4.10]{grothendieck_localcohomology}, and we can take $D_{E(k)}=\Hom(-,E(k))$ for the duality functor.

\begin{cor}[{\cite[Thm.~6.3]{grothendieck_localcohomology}}]\label{cor:gdlocalduality}
	Suppose that $A$ is a Gorenstein local ring of dimension $n$. Then, for any finitely generated $A$-module $M$ and all $k$ there is a natural isomorphism
	\[
  		C^\m \Ext_A^k(M, A) \cong \Hom_A(H^{n-k}_\m(M), E(k)),
\]
where $C^\m$ is the $\m$-adic completion functor. Moreover, if $A$ is complete, then there is a natural isomorphism
	\[
		H^k_\m(M)\cong D_{E(k)}(\Ext_A^{n-k}(M,A)),
	\]
\end{cor}

\begin{proof}
	Since $A$ is Gorenstein, $H^n_\m(A)\cong E(k)$ is the only non-trivial local cohomology of $A$. Hence, the spectral sequence of \cref{thm:glocalduality} collapses to give the isomorphism
	\[
	\Ext^k_A(M, D_{E(k)}(H^n_\m A)) \cong \Hom_A(H^{n-k}_\m(M), E(k)).
	\]
	To reinterpret the left hand side, first note that we have
	\[
	\Hom( M, \Hom(\Gamma_\m A, \Gamma_\m A))\simeq \Hom( M, \Hom(\Gamma_\m A, A)) \simeq \Hom(\Gamma_\m A, \Hom(M,A)), 
	\]
where the first equivalence comes from the fact that $\Gamma_\m$ is right adjoint to the inclusion of $\D_{A}^{I-\mathrm{tors}}$ in $\D_A$. Moreover, $\Hom(\Gamma_\m A, \Hom(M,A))\simeq \Lambda^\m \Hom(M,A)$ by local duality.

In what follows let $L^\m_sM$ be the $s$-th left derived functor of $\m$-adic completion for $s \ge 0$, to be studied further in~\Cref{sec:lcompletion}. For now we just need the following: if $N$ is a finitely generated $A$-module, then $L^\m_0N \cong C^\m N$, and $L^\m_sN = 0$ for $s>0$, as proven in~\cite[Prop.~A.4, Thm.~A.6]{hovey_morava_1999}. Moreover, for arbitrary $N$, we have $H^{-s}(\Lambda^\m N) \cong L^\m_s N$, which follows from~\Cref{thm:gm_localhomology} below.  Then, using the Grothendieck spectral sequence we have
\[
   H^k(\Lambda^\m \Hom(M,A))\cong L^\m_0 \Ext_A^k(M,A) \cong C^\m \Ext^k_A(M,A)
\]
as $\Ext_A^k(M,A)$ is finitely generated. This completes the proof of the first part. 

Under the assumption that $A$ is complete, the second equivalence is a consequence of the first and Matlis duality \cite[Thm.~3.2.13]{cmrings}.
\end{proof}

\subsection{$L$-completion}\label{sec:lcompletion}

Let $A$ be a Noetherian ring and $I\subseteq A$ an ideal as before. It is well known that the $I$-adic completion of an $A$-module, $C^I(M)= \lim_s M/I^sM$, is neither right nor left exact as a functor on $\Mod_A$, the abelian category of $A$-modules, see \cite[App.~A]{hovey_morava_1999} for example. 

We could deal with this situation by either considering the zeroth left derived or right derived functor of $C^I$ as the ``correct'' homological approximation of completion. These were initially studied (under more general assumptions on the ring $A$) by Greenlees and May in \cite{gm_localhomology}, where many results on the left derived functors were established, including some that we recover here. 

In contrast, they showed that the right derived functors of completion vanish, so it is natural to focus only on the left derived functors of the $I$-adic completion. Moreover, these derived functors are relevant to our discussion; the connection comes through the following interpretation of these functors in terms of local homology. We note that the following result holds for \emph{any} ring $A$. 

\begin{thm}[{\cite[Thm.~2.5]{gm_localhomology}}]\label{thm:gm_localhomology}
		Let $A$ be a commutative ring. Let $I=(x_1,\dots,x_n)$ be an ideal generated by weakly proregular finite sequence and assume that $A$ is $x_i$-bounded torsion for all $i$. Then, for all $A$-modules $M$, and all $s\geq 0$, there is a natural isomorphism
		
		\[
  				H^I_s(M)\cong L_s^IM,
\]		
where $L^I_s$ is the $s$-th left derived functor of $I$-adic completion. 
\end{thm}
\begin{proof}
	On the one hand, by \cref{prop:gamma=derived_torsion} we have $\Lambda^I\simeq \mathbf{L}C^I$. On the other hand, local duality implies that $\Lambda^I\simeq \Hom(\Gamma A, X)$. The conclusion follows after applying homology to the equivalence $\Hom(\Gamma A, X)\simeq \mathbf{L}C^I$. 	
\end{proof}

\begin{rem}
	The original statement in \cite[Thm.~2.5]{gm_localhomology} required the ideal $I$ to be proregular, see \cite[Def.~1.8]{gm_localhomology}.  By \cite[Lem.~2.7]{schenzel_proreg} any proregular ideal is weakly proregular, although the converse is not true, and thus our result is slightly stronger. Moreover, since the regularity condition we use above is optimal in \cref{prop:gamma=derived_torsion}, it would be reasonable to expect that the same is true for \Cref{thm:gm_localhomology}.
\end{rem}

Continuing in the same fashion, we use the Grothendieck spectral sequence to give a quick alternative proof of~\cite[Prop.~1.1]{gm_localhomology}. Here, $\lim^p$ denotes the $p$-th derived functor of $\lim$. 

\begin{cor}[Greenlees--May]\label{lem:gmses}
If $M \in  \Mod_A$, then there is a short exact sequence 
\[\xymatrix{0 \ar[r] & \lim^1_k \Tor^{s+1}_A(A/I^k,M) \ar[r] & L_sM \ar[r] & \lim_k \Tor^s_A(A/I^k,M) \ar[r] & 0}\]
for any $s \ge 0$. In particular, there exists a natural epimorphism $L_0M \to C^I(M)$.
\end{cor}
\begin{proof}
The Grothendieck spectral sequence associated to $C^I(-) = \lim A/I^k \otimes (-)$ exists because an inverse system of modules satisfying the Mittag--Leffler condition remains Mittag--Leffler under tensoring with arbitrary modules. Therefore, we can work with a flat resolution of $M$. Thus, the spectral sequence takes the form
\[
\xymatrix{E_2^{p,q}=\lim^p_k\Tor_q^A(A/I^k,M) \implies L_{p-q}^IM.}
\]

The fact that $\lim^p = 0$ for all $p>1$ in the category $\Mod_A$ implies that the spectral sequence degenerates to a short exact sequence.
\end{proof}

\section{Comodules over a flat Hopf algebroid}\label{sec:comodules}
We now move on to the main example of interest in this paper, which is the stable category of comodules over a Hopf algebroid. In this section we start by reviewing the basic properties of the abelian category of comodules over a flat Hopf algebroid $(A,\Psi)$, which we denote by $\Comod_\Psi$, collecting known results that will be required.  It is known that $\Comod_\Psi$ is a Grothendieck abelian category, and hence has a corresponding derived category $\mathcal{D}_\Psi$. The latter has some deficiencies however, as noted in~\cite{hovey_htptheory} or~\cite[Sec.~3]{hovey_chromatic}. To overcome these we instead give an alternative construction based on work of Hovey~\cite{hovey_htptheory} and Krause~\cite{krausestablederivedcat}, and prove enough about this stable category so that we can study local duality thoroughly in the sequel.  We finish by showing that, in many cases, our model agrees with Hovey's model in a suitable sense.

\subsection{Comodules over a Hopf algebroid}\label{sec:comodintro}
We will begin by reviewing the abelian category of comodules over a Hopf algebroid. None of the results in this section are new, and we refer the reader to~\cite[Sec.~1]{hovey_htptheory} and~\cite[App.~A1]{ravenel_86} for further details. 

Recall that a Hopf algebroid $(A,\Psi)$ over a commutative ring $K$ is a cogroupoid object in the category of commutative $K$-algebras. In particular, it consists of a pair $(A,\Psi)$ of commutative $K$-algebras, along with $K$-algebra morphisms
\[
\begin{split}
	\eta_L&\colon A \to \Psi \\
	\eta_R&\colon A \to \Psi \\
	\Delta&\colon \Psi \to \Psi \otimes_A \Psi \\
	\epsilon&\colon \Psi \to A \\
	c&\colon \Psi \to \Psi,
\end{split}
\]
satisfying a series of identities that ensure that, for any commutative ring $B$, the sets $\Hom(A,B)$ and $\Hom(\Psi,B)$ are the objects and morphisms of a groupoid, respectively. Here, $\eta_L$ gives $\Psi$ the structure of a left $A$-module, $\eta_R$ gives $\Psi$ a right $A$-module structure, and $\Psi \otimes_A \Psi$ refers to the bimodule tensor product.

\begin{defn}
	A left $\Psi$-comodule $M$ is a left $A$-module $M$ together with a counitary and coassociative left $A$-linear morphism $\psi_M\colon M \to \Psi \otimes_A M$. 
\end{defn}

We write $\Comod_\Psi$ for the category of $\Psi$-comodules; we will always assume that $\Psi$ is a flat $A$-module, which ensures that $\Comod_\Psi$ is a cocomplete abelian subcategory of $\Mod_A$~\cite[Lem.~1.1.1]{hovey_htptheory}. The forgetful functor $\Comod_\Psi \to \Mod_A$ has a right adjoint that sends an $A$-module $M$ to the $\Psi$-comodule $\Psi \otimes_A M$, with structure map $\Delta \otimes 1$~\cite[Def.~A1.2.1]{ravenel_86}; such a comodule is called an extended comodule. It follows that the colimit in $\Comod_\Psi$ is just the colimit in $A$-modules equipped with its canonical comodule structure. The forgetful functor does not preserve limits, or even infinite products, and so in general it is difficult to construct right adjoints in $\Comod_\Psi$. Nonetheless, the category $\Comod_\Psi$ is also complete, see~\cite[Prop.~1.2.2]{hovey_htptheory}. Since we will need to understand inverse limits of comodules in the sequel, we now explain how they are constructed, mimicking Hovey's construction of products. 

Let $(N_k)$ be an inverse system of comodules and let $\psilim N_k$ be the inverse limit we wish to construct. Suppose first that we have an inverse system $(M_k)$ of $A$-modules. Arguing via adjointness we see that we must have 
\begin{equation}\label{eq:comodlimextended}
\xymatrix{
\psilim (\Psi \otimes_A M_k) \cong \Psi \otimes_A \lim_k M_k,}
\end{equation}
where $\lim_kM_k$ denotes the inverse limit in $\Mod_A$. Given a set of comodule maps $f_k\colon \Psi \otimes_A M_k \to \Psi \otimes_A L_k$ we define $\psilim(f_k)$ by \[
\begin{split}
\Psi \otimes_A \lim_k M_k \xr{\Delta \otimes 1} \Psi \otimes_A \Psi \otimes_A \lim_k M_k \xr{1 \otimes \alpha} \Psi \otimes_A \lim_k (\Psi \otimes_A M_k) \\
\xr{1 \otimes \lim_k(f_k)} \Psi \otimes_A \lim_k(\Psi \otimes_A L_k) \xr{1 \otimes \lim_k(\epsilon \otimes 1)} \Psi \otimes_A \lim_k L_k, 
\end{split}
\] 
where $\alpha$ is the canonical map. One can check that this defines the inverse limit on the full subcategory of extended $\Psi$-comodules. 

 Finally, given an inverse system $(N_k)$ of comodules, there are left exact sequences
\[
0 \to N_k \xr{\psi} \Psi \otimes_A N_k \xr{p_k} \Psi \otimes_A T_k,
\]
where $T_k$ is the cokernel of $\psi$ and $p_k$ is the composite
\[
\Psi \otimes_A N_k \to T_k \xr{\psi} \Psi \otimes_A T_k,
\]
see~\cite[Sec.~1.2]{hovey_htptheory}.

It follows that $\psilim(N_k)$ is the kernel, i.e., is defined by the short exact sequence
\begin{equation}\label{eq:kernelinvlimit}
\xymatrix{0 \ar[r] & \psilim(N_k) \ar[r] & \Psi \otimes \lim_k(N_k) \ar[r] & \Psi \otimes \lim_k (T_k). }
\end{equation}
One can check that this construction does indeed define the inverse limit in $\Comod_\Psi$.

The category of $\Psi$-comodules is symmetric monoidal. Let $M$ and $N$ be $\Psi$-comodules and define their comodule tensor product $M \otimes N$ to be the module $M \otimes_A N$. The comodule structure map is slightly more subtle; consider the composite 
\[
M \otimes_K N \xr{\psi_M \otimes \psi_N} (\Psi \otimes_A M) \otimes_K (\Psi \otimes_A N) \to \Psi \otimes_A M \otimes_A N,
\]
where the last map sends $g_1 \otimes m \otimes g_2 \otimes n$ to $g_1g_2 \otimes m \otimes n$. This composite then descends to a well-defined morphism $M \otimes_A N \to \Psi \otimes_A M \otimes_A N$ (however the separate maps are only defined when we tensor over $K$). It is sometimes useful to know that for a comodule $N$ there is a natural isomorphism of comodules $\Psi \otimes N \cong \Psi \otimes_A N$, where the latter refers to the extended $\Psi$-comodule, and thus we typically omit the subscript on tensor products. 

The tensor product has a right adjoint $\iHom_{\Psi}(-,-)$ making $\Comod_\Psi$ a closed symmetric monoidal category~\cite[Thm.~1.3.1]{hovey_htptheory}, although it is harder to explicitly give a formula for $\iHom_{\Psi}(M,N)$. We make the following observations: for an extended $\Psi$-comodule $\Psi \otimes N$, adjointness gives an isomorphism 
\begin{equation}\label{eq:ihomextended}
\iHom_\Psi(M,\Psi\otimes N) \cong \Psi \otimes \Hom_A(M,N)
\end{equation} and, as an $A$-module, there is always a morphism $\iHom_{\Psi}(M,N) \to \Hom_A(M,N)$; if $M$ is finitely presented over $A$, then this morphism is an isomorphism~\cite[Prop.~1.3.2]{hovey_htptheory}. 

With this closed structure, we can define the dual of a $\Psi$-comodule $M$ to be $DM = \iHom_{\Psi}(M,A)$. Then, $M$ is dualizable in $\Comod_\Psi$ if and only if $M$ is dualizable as an $A$-module, i.e., $M$ is finitely generated and projective as an $A$-module~\cite[Prop.~1.3.4]{hovey_htptheory}. Similarly, $M$ is compact in $\Comod_{\Psi}$ if and only if the underlying $A$-module of $M$ is compact in $\Mod_A$, which is equivalent to being finitely presented over $A$. A Hopf algebroid $(A,\Psi)$ will be called an Adams Hopf algebroid if it satisfies Adams' condition, i.e., if $\Psi$ is a filtered colimit of dualizable $\Psi$-comodules. 

We also note that $\Comod_\Psi$ is a Grothendieck abelian category, see~\cite[Thm.~5.4]{alonso_functorial} or~\cite[Prop.~3.4.7]{sitte2014local}; this implies that $\Comod_\Psi$ is locally presentable~\cite[Prop.~3.10]{bekor}. To summarize what we need:

\begin{prop}\label{prop:comoduleprop}
	The category $\Comod_{\Psi}$ of comodules over a flat Hopf algebroid $(A,\Psi)$ is a complete and cocomplete locally presentable Grothendieck abelian category, with a closed symmetric monoidal structure. In addition:
	\begin{enumerate}
		\item A comodule $M$ is dualizable (compact) if and only $M$ is dualizable (compact) as an $A$-module. 
		\item There is a morphism of $A$-modules $\iHom_{\Psi}(M,N) \to \Hom_A(M,N)$ which is an isomorphism when $M$ is compact.
		\item Filtered colimits in $\Comod_{\Psi}$ are exact and are computed in $\Mod_A$. 
	\end{enumerate}
\end{prop}

A morphism $\Phi\colon (A,\Psi) \to (B,\Sigma)$ of Hopf algebroids consists of a pair $(\Phi_0,\Phi_1)$ of ring homomorphisms $\Phi_0\colon A \to B$ and $\Phi_1\colon \Psi \to \Sigma$ making the obvious diagrams commute. Such a morphism induces an adjoint pair of functors~\cite{miller_ravenel_local}
\[
\xymatrix{\Phi_\ast\colon \Comod_\Psi \ar@<0.5ex>[r] & \ar@<0.5ex>[l] \Comod_\Sigma\colon\Phi^\ast,}
\]
with $\Phi_*$ symmetric monoidal~\cite[Lem.~1.13]{hovey_htptheory}. Consider the morphism $\epsilon\colon(A,\Psi) \to (A,A)$ which is the identity on $A$ and is given by the counit $\epsilon$ on $\Psi$. It is easy to check that there is an equivalence of categories $\Comod_A \simeq \Mod_A$, and that $\epsilon_*$ corresponds to the forgetful functor from $\Comod_\Psi$ to $\Mod_A$. In this case $\epsilon^\ast$ is the extended comodule functor mentioned previously.

Recall that a functor is called conservative if it reflects isomorphisms. 
\begin{lem}\label{lem:epsilonconservative}
The forgetful functor $\epsilon_*\colon\Comod_\Psi \to \Mod_A$ is conservative. 	
\end{lem}
\begin{proof}
	We need only show that $\epsilon_*$ is faithful, since any faithful functor whose domain is an abelian category reflects isomorphisms. That $\epsilon_*$ is faithful is clear from the definition of a morphism. 
\end{proof}
There are enough injectives in $\Comod_\Psi$; in fact, a $\Psi$-comodule is injective if and only if it is a comodule retract of $\Psi \otimes_A I$ for some injective $A$-module $I$~\cite[Lem.~2.1]{hs_localcohom}.
\begin{defn}\label{defn:iext}
For $\Psi$-comodules $M$ and $N$, $\uExt_\Psi^i(M,N)$ is the $i$-th right derived functor of $\iHom_\Psi(M,N)$, formed via an injective resolution of $N$. Similarly, we write $\ilim_{\cD_{\Psi},k}^iM_k$ or simply $\psilim^i M_k$ for the right derived functors of inverse limit of a system of comodules $(M_k)_k$.
\end{defn}

We will see in~\Cref{lem:derivedhomcohom} that if $M$ is finitely presented, then $\uExt^i_\Psi(M,N)$ also commutes with $\epsilon_*$. 
\subsection{Ind-categories of comodules}\label{sec:indconstruction}

In this subsection, we establish the categorical foundations for the stable category of comodules we are going to use in the rest of the paper. As pointed out by Hovey~\cite{hovey_htptheory}, the usual derived category $\cD_{\Psi}$ of comodules over some flat Hopf algebroid $(A,\Psi)$ lacks some desired properties. In particular, the comodule $A$ is in general not a compact object in $\cD_{\Psi}$, even for familiar examples like $(A,\Psi) = (BP_*,BP_*BP)$ or $(E_*,E_*E)$ with $E$ some version of Morava $E$-theory. Instead, in~\cite{hovey_htptheory} Hovey constructs a model structure on the category of chain complexes of comodules, which we denote by $\Stableh_\Psi$, whose homotopy category is, under suitable hypotheses, a stable homotopy category generated by a set of dualizable $\Psi$-comodules. This construction has two important additional properties:
\begin{itemize}
 \item If $(A,\Psi) = (A,A)$ is discrete, then $\Stableh_{\Psi} \simeq \cD_{A}$, the usual (unbounded) derived category of $A$-modules. 
 \item In general, the derived category $\cD_{\Psi}$ of $\Comod_{\Psi}$ is a left Bousfield localization of $\Stableh_{\Psi}$ at the homology isomorphisms. 
\end{itemize}
This category provides a natural framework for doing homological algebra of comodules, as amply demonstrated in~\cite{hovey_chromatic}. Moreover, $\Stableh_{\Psi}$ is well behaved enough so that our abstract local duality result can be applied. However, in order to concretely identify and study the local cohomology and local homology functors so constructed, we would have to work on the level of the model structure and not just in the homotopy category, which becomes technically demanding due to the construction of the model structure.

To explain the idea of the construction, consider the case of modules over the group algebra $k[G]$ for $k$ some field of characteristic $p$, and $G$ a finite $p$-group. The abelian category $\Mod_{k[G]}$ of $k[G]$-modules can be identified naturally with the category of comodules over the Hopf algebra $(k,k[G]^*)$, $k[G]^*$ being the $k$-linear dual of $k[G]$. However, for certain purposes the derived category $\D(G)$ of $\Mod_{k[G]}$ has some deficiencies; for example, the object $k \in \D(G)$ is not compact. As a remedy, one can work with the stable category $\stable_{k[G]}$, which is a triangulated category in which $k$ becomes a compact generator. There are several equivalent constructions available in the literature, but we would like to single out the following: Consider the thick subcategory $\Thick_{G}(k)$ of $\D(G)$ generated by $k$. Work of Krause~\cite{krausestablederivedcat}, Benson and Krause~\cite{krausebenson_kg}, and Hovey, Palmieri, and Strickland~\cite[Sec.~9.5]{hps_axiomatic}
provides equivalences
\[
\xymatrix{\Ind(\Thick_{G}(k)) \ar[r]^-{\sim} & \Ho(\cD_{kG}^{\mathrm{inj}}) \ar[r]^{\sim} & \stable_{k[G]}}
\]
of triangulated categories, where $\Ho(\cD_{kG}^{\mathrm{inj}})$ is the homotopy category of complexes of injective $k[G]$-modules. Note that the stable module category $\mathrm{StMod}(G)$ of $k[G]$ is the finite localization of $ \stable_{k[G]}$ at the thick subcategory generated by $k[G]$.

Let $(A,\Psi)$ be a flat Hopf algebroid, and write $\cG=\cG_{\Psi}$ for a set of representatives of dualizable comodules in $\Comod_{\Psi}$. By~\Cref{prop:comoduleprop}, the category $\Comod_{\Psi}$ of $\Psi$-comodules is a Grothendieck abelian category, so there exists a corresponding derived category of $\Psi$-comodules. The next result shows that this derived category admits a well-behaved homotopical model.
\begin{prop}[Hovey, Barnes--Roitzheim]\label{prop:dermodelstructure}
There is a cofibrantly generated monoidal model structure on the category $\Ch_{\Psi}$ of chain complexes of $\Psi$-comodules whose homotopy category is the usual derived category of $\Psi$-comodules. 
\end{prop}
\begin{proof}
This is a special case of \cite[Thm.~6.9]{barnesroitzheim_frankesmodel}, but for the convenience of the reader we sketch a construction. To this end, recall that Hovey \cite[Thm.~2.1.3]{hovey_htptheory} constructs a projective model structure on $\Ch_{\Psi}$. In this model structure, a map $f$ is a weak equivalence or fibration if and only if $\Hom(P,f)$ is a quasi-isomorphism or surjective map for all dualizable $\Psi$-comodules $P$, respectively. Left Bousfield localization at the quasi-isomorphisms then yields the desired model structure; to see that the resulting model structure is monoidal, see \cite{white_bousfield}. 
\end{proof}

We will write $\cD_{\Psi}$ for the underlying symmetric monoidal stable $\infty$-category of the model category constructed in \Cref{prop:dermodelstructure} above, and refer to it as the derived stable $\infty$-category of $\Psi$-comodules.
\color{black}
Using the methods of~\Cref{sec:ind-categories} we can now construct our version of the stable category of comodules. 

\begin{defn}\label{defn:stable_comod}
We define the stable $\infty$-category of $\Psi$-comodules as the ind-category of the thick subcategory of $\cD_{\Psi}$ generated by $\cG$, i.e.,
\[
\Stable_{\Psi} = \Ind(\Thick_{\Psi}(\cG)),
\]
where $\cG$ is a set of representatives of the class of dualizable objects in $\Comod_\Psi$ viewed as complexes concentrated in degree $0$.
\end{defn}

Since $\Thick_{\Psi}(\cG)$ is a small stable $\infty$-category with a closed symmetric monoidal structure $\otimes$ preserving finite colimits separately in each variable, \Cref{prop:indproperties} immediately yields the next result.

\begin{prop}\label{prop:adjoint}
If $(A,\Psi)$ is a flat Hopf algebroid, then the stable category of comodules $\Stable_{\Psi}$ is a presentable stable $\infty$-category compactly generated by $\cG$, and equipped with a closed symmetric monoidal structure $\otimes$ preserving colimits separately in each variable. 
\end{prop}
Following this we will refer to $\Stable_{\Psi}$ simply as the stable category of $\Psi$-comodules. \\

The stable category $\cD_{\Psi}$ is cocomplete, so the universal property of the ind-category yields a continuous functor
\[ 
\xymatrix{\Stable_{\Psi} \ar[r]^-{\omega} & \cD_{\Psi}}
\]
extending the natural inclusion $\Thick_{\Psi}(\cG) \to \cD_{\Psi}$. It follows from~\Cref{prop:indsymmmon} that $\omega$ is in fact a symmetric monoidal functor of stable categories, and it has a right adjoint $\iota$ by~\Cref{prop:indrightadjoint}. In the case of a discrete Hopf algebroid we get an analogue of the result mentioned in the introduction. 

\begin{lem}\label{lem:inddiscrete}
If $(A,A)$ is a discrete Hopf algebroid, then there is a symmetric monoidal equivalence
\[ 
\xymatrix{\Stable_A \ar[r]^-{\omega}_-{\sim} & \cD_A.}
\]
\end{lem}
\begin{proof}
If $(A,A)$ is a discrete Hopf algebroid, then $\Thick(\cG) = \Thick(A)$ as every dualizable $A$-module is finitely presented and projective. Because $A$ is a compact generator of $\cD_A$, $\omega$ is an equivalence.  
\end{proof}

This motivates the following definition. 

\begin{defn}\label{def:homotopygroups}
The homology groups of an object in $\Stable_{\Psi}$ are defined as the homology of the underlying chain complex, i.e.,
\[
\sh_*M = H_*(\omega M)
\]
for any $M \in \Stable_{\Psi}$. 
\end{defn}

We now study the relation between the stable category of comodules $\Stable_{\Psi}$ for a Hopf algebroid $(A,\Psi)$ and the derived category $\cD_A \simeq \Stable_A$. To this end, recall that the Hopf algebroid morphism $\epsilon\colon (A,\Psi) \to (A,A)$ induces an adjunction
\[ 
\xymatrix{\epsilon_*\colon \Comod_{\Psi} \ar@<0.5ex>[r] & \Comod_A\colon \epsilon^* \ar@<0.5ex>[l]}
\]
of abelian categories, where $\epsilon_*$ is the forgetful functor and $\epsilon^*(M) = \Psi \otimes_A M$, the extended comodule on $M \in \Comod_A \simeq \Mod_A$. Note that both functors are exact, because $\Psi$ is flat over $A$. It follows that there is an induced adjunction $(\epsilon_*,\epsilon^*)$ of stable categories between $\cD_{\Psi}$ and $\cD_A$. This adjunction lifts to an adjunction of stable comodule categories.

\begin{lem}\label{lem:derivedforget}
There is an adjunction of stable categories
\[
\xymatrix{\epsilon_*\colon\Stable_{\Psi} \ar@<0.5ex>[r]^-{} & \cD_A\colon \epsilon^* \ar@<0.5ex>[l]}
\]
which extends the adjunction $\epsilon_* \dashv \epsilon^*$ on derived categories. Moreover, the functor $\epsilon_*$ is symmetric monoidal while $\epsilon^*$ is continuous. 
\end{lem}
\begin{proof}
Since $\epsilon_*$ is exact, we can consider the restriction of $\epsilon_*$ to a functor $\Thick_{\Psi}(\cG) \to \Thick_A(A)$. The composite 
\[
\xymatrix{\Thick_{\Psi}(\cG) \ar[r] &  \Thick_A(A) \ar[r] & \cD_A}
\] 
extends to a symmetric monoidal functor $\Stable_{\Psi} \to \cD_A$, which we also call $\epsilon_*$ by a mild abuse of notation. Using~\Cref{lem:inddiscrete}, this functor can be identified with $\Ind(\epsilon_*)$, so~\Cref{prop:indrightadjoint} yields a right adjoint $\epsilon^*\colon \cD_A \to \Stable_{\Psi}$ to $\epsilon_*$. Finally, $\epsilon_*$ preserves compact objects in light of \cref{prop:comoduleprop} so it follows that $\epsilon^*$ is continuous.
\end{proof}

The following diagram summarizes the categories and functors introduced so far:
\[
\xymatrix{\Stable_{\Psi} \ar@<0.5ex>[r]^-{\omega} \ar@<0.5ex>[rd]^>>>>>>{\epsilon_*} & \D_{\Psi} \ar@<0.5ex>[l]^-{\iota} \ar@<0.5ex>[d]^-{\epsilon_*} \\
& \D_A. \ar@<0.5ex>[u]^-{\epsilon^*} \ar@<0.5ex>[ul]^{\epsilon^*}}
\]

\begin{rem}\label{rem:stablechoices}
We note that there is at least one reasonable variation of \Cref{defn:stable_comod}: Instead of defining $\Stable_{\Psi}$ as $\Ind(\Thick_{\Psi}(\cG))$, we could consider the ind-category on the thick subcategory of $\D_{\Psi}$ generated by the compact objects of $\Comod_{\Psi}$. While this choice comes with some advantages and disadvantages, it induces equivalent local cohomology and local homology functors in many examples of interest. In particular, this is the case for Landweber Hopf algebroids introduced in the next subsection, and hence for the theories studied in \Cref{sec:duality-chromatic}. 
\end{rem}

\subsection{Landweber Hopf algebroids}

Under some more restrictive assumptions on the Hopf algebroid, the relation between $\Stable_{\Psi}$ and $\D_{\Psi}$ becomes even closer. To this end, we will introduce the notion of a Landweber Hopf algebroid. 

\begin{defn}
Suppose $(A,\Psi)$ is a flat Hopf algebroid and write $\Comod_{\Psi}^{\omega}[0]$ for the image of the (nerve of the) abelian category of compact comodules in $\D_{\Psi}$. If $M$ is in $\Thick_{\Psi}(A)$ for every $M \in \Comod_{\Psi}^{\omega}[0]$, then we call $(A,\Psi)$ a Landweber Hopf algebroid.
\end{defn}

This terminology is inspired by the Landweber filtration theorem~\cite{land_filt} in stable homotopy theory, which allows the construction of a natural class of examples: Indeed, the proof of~\cite[Cor.~6.7]{hovey_htptheory} shows that any ring spectrum $R$ that is Landweber exact over $MU$ or $BPJ$, where $J$ is some finite regular invariant sequence, such that $R_*R$ is commutative, gives a Landweber Hopf algebroid $(R_*,R_*R)$. Examples of such ring spectra include $H\mathbb{F}_p$, topological $K$-theory or more generally Morava $E$-theory $E_n$, and the Brown--Peterson spectrum $BP$. 

Let $\cD_0 \subset \D_{\Psi}$ be the $\infty$-category of complexes $Q$ with homology concentrated in finitely many degrees and such that $H_d(Q) \in \Comod_{\Psi}$ is compact for all $d\in \Z$. These categories admit a convenient description for Landweber Hopf algebroids. 

\begin{lem}\label{lem:boundedhomology}
If $(A,\Psi)$ is a Landweber Hopf algebroid, then $\cD_0$ and $\Thick_{\Psi}(A)$ are equivalent as subcategories of $\cD_\Psi$. 
\end{lem}
\begin{proof}
Since $A$ is compact, $\Thick_{\Psi}(A)\subseteq \cD_0$ is clear, and we will prove the reverse inclusion by induction on the number of non-zero homology groups. Indeed, let $(Q,d) \in \cD_0$ be a complex whose homology is concentrated in degrees $[a,b]$. There is a commutative diagram of chain complexes
\[
\xymatrix{\ldots \ar[r]^{d_{b+2}} & Q_{b+1} \ar[r]^{d_{b+1}} & Q_b \ar[r]^{d_{b}} & Q_{b-1} \ar[r]^{d_{b-1}} & \ldots\\
\ldots \ar[r]^{d_{b+2}} & Q_{b+1} \ar[r]^{d_{b+1}} \ar[u]^{f_{b+1}}_= \ar[d]_{g_{b+1}} & \ker(d_b) \ar[r] \ar@{^{(}->}[u]^{f_{b}} \ar[d]_{g_{b}} & 0 \ar[r] \ar[u]^{f_{b-1}} \ar[d]_{g_{b-1}} & \ldots \\
\ldots \ar[r] & 0 \ar[r] & \ker(d_b)/\im(d_{b+1}) \ar[r] & 0 \ar[r] & \ldots}
\]
with $f$ being a quasi-isomorphism. If $a=b$, then $g$ is also a quasi-isomorphism, so $Q \in \Thick(A)$ since $(A,\Psi)$ is Landweber. If $a<b$, the cofiber of the map $g\circ f^{-1}$ is a complex $Q' \in \cD_0$ with homology concentrated in $[a,b-1]$. Hence $Q$ fits into a cofiber sequence
\[
\xymatrix{Q \ar[r] & \ker(d_b)/\im(d_{b+1}) \ar[r] & Q'}
\]
with $\ker(d_b)/\im(d_{b+1}), Q' \in \Thick(A)$ by induction hypothesis, so the claim follows. 
\end{proof}

Let $\cA$ be a Grothendieck abelian category. Following Krause~\cite{krausestablederivedcat}, Lurie~\cite[Sec.~1.3.6]{ha} constructs a stable category $\Dinj(\cA)$ of unbounded chain complexes of injective objects in $\cA$. Specializing to $\cA = \Comod_{\Psi}$, there exists a factorization
\[
\xymatrix{\Stable_{\Psi} \ar@{-->}[r]^-{\theta} \ar[rd]_{\omega} & \Dinjpsi \ar[d]^{\omega'} \\
& \D_{\Psi},}
\]
where we write $\Dinjpsi  =  \Dinj(\Comod_{\Psi})$.

\begin{lem}\label{lem:thetaff}
If $A$ is Noetherian, then the functor $\theta\colon \Stable_{\Psi} \to \Dinjpsi$ is fully faithful.
\end{lem}
\begin{proof}
As shown in~\cite[Lem.~3.6.7]{sitte2014local} using~\cite[Lem.~11.1]{hasimoto}, $A$ Noetherian implies that the abelian category $\Comod_{\Psi}$ is locally Noetherian. Krause's theorem in the form~\cite[Thm.~1.3.6.7]{ha} then identifies $\Dinjpsi = \Dinj(\Comod_{\Psi})$ with the ind-category $\Ind(\cD_0)$. Since $G \in \cD_0$ for each $G \in \cG$, there is a natural inclusion $\Thick_{\Psi}(\cG) \to \cD_0$, so the claim follows from the universal property of ind-categories. 
\end{proof}

For a Landweber Hopf algebroid $(A,\Psi)$, $\theta$ is an equivalence, and thus $\Stable_{\Psi}$ inherits a number of additional good properties:

\begin{prop}\label{prop:nicestable}
Assume $(A,\Psi)$ is a Landweber Hopf algebroid with $A$ Noetherian. There is a natural $t$-structure on $\Stable_{\Psi}$ such that the inclusion functor $\iota\colon \D_{\Psi} \to \Stable_{\Psi}$ is $t$-exact and induces natural equivalences 
\[
\xymatrix{(\D_{\Psi})_{\le k} \ar[r]^-{\sim} & (\Stable_{\Psi})_{\le k}}
\]
on the full subcategories of $k$-coconnective objects for all $k \in \Z$. 
\end{prop}
\begin{proof}
By \Cref{lem:boundedhomology}, $\cD_0$ is equivalent to $\Thick_{\Psi}(A)$, hence $\Thick_{\Psi}(A) = \Thick_{\Psi}(\cG)$ as dualizable comodules are compact. Combining this with the argument of~\Cref{lem:thetaff} yields an equivalence of stable categories 
\[
\xymatrix{\Stable_{\Psi} \ar[r]^-{\theta}_-{\sim} & \Dinjpsi}
\]
if $(A,\Psi)$ is a Landweber Hopf algebroid. Therefore,~\cite[Prop.~1.3.6.2]{ha} and~\cite[Rem. 1.3.6.3 and Rem.~1.3.6.4]{ha} apply to give the result.
\end{proof}

Note that, by construction, the natural notion of homotopy on $\Stable_{\Psi}$ induced by the $t$-structure of \Cref{prop:nicestable} coincides with the notion of homology given in \Cref{def:homotopygroups}.

\begin{cor}\label{cor:derivedcathomcohom}
If $(A,\Psi)$ is a Landweber Hopf algebroid with $A$ Noetherian, then there is an isomorphism
\[
\xymatrix{\sh^*\Hom_{\Stable_{\Psi}}(M,N) \ar[r]^-{\simeq} & \sh^*\Hom_{\D_{\Psi}}(\omega_*M,\omega_*N) = \Ext_{\Psi}^{*}(\omega_*M,\omega_*N)}
\]
of abelian groups, for all $M,N \in (\Stable_{\Psi})_{\le k}$ and $k\in \Z$.
\end{cor}

\begin{rem}\label{rem:krausegeneral}
Recent work of Krause's~\cite{krause_auslander} suggests that the Noetherianness assumption in \Cref{prop:nicestable} can be removed, but we have not checked the details.
\end{rem}

\subsection{Basic properties of $\Stable_{\Psi}$}
Before we can study local duality on $\Stable_\Psi$ we need to establish some of its basic properties. Indeed~\eqref{eq:torsionkoszul} and~\Cref{cor:completion} tell us that to understand torsion and completion on $\Stable_\Psi$ we need to understand colimits and limits. 

We start with a derived analogue of~\eqref{eq:ihomextended}. 

\begin{lem}\label{lem:ihom}
Let $M \in \mathcal{D}_A$ and $N\in\Stable_\Psi$, then there is an equivalence in $\Stable_\Psi$
	\[
\iHom(N,\epsilon^* M) \simeq \epsilon^*\iHom_{\cD_A}(\epsilon_*N,M). 
	\]
\end{lem}
\begin{proof}
The proof here is identical to the version for abelian categories. Note that by~\Cref{prop:adjoint}, the category $\Stable_\Psi$ is closed symmetric monoidal, i.e., there is a pair $\otimes$ and $\iHom$ satisfying the usual adjunction. Moreover,~\Cref{lem:derivedforget} gives an adjunction of stable $\infty$-categories between $\Stable_\Psi$ and $\mathcal{D}_A$ with $\epsilon_*$ is symmetric monoidal. Thus
	\[
	\begin{split}
\Hom_{\Stable_\Psi}(L,\iHom(N,\epsilon^* M)) & \simeq \Hom_{\Stable_\Psi}(L \otimes N,\epsilon^* M) \\ 
&\simeq \Hom_{\cD_A}(\epsilon_*L \otimes_A \epsilon_*N,M) \\
&\simeq \Hom_{\cD_A}(L,\iHom_{\cD_A}(\epsilon_*N,M)) \\
&\simeq \Hom_{\Stable_\Psi}(L, \epsilon^*\iHom_{\cD_A}(\epsilon_*N,M)).
\end{split}
	\]
	The result follows from the (derived) Yoneda lemma. 
\end{proof}

For the next result, we compare the internal hom $\iHom(-,-)$ in $\Stable_\Psi$ with that in the ordinary derived category, $\iHom_{\D_{\Psi}}(-,-)$.

\begin{lem}\label{lem:derivedhomcohom}
For $M,N \in \Stable_{\Psi}$ with $M$ compact there are natural equivalences
\[
\omega\iHom(M,N) \simeq \iHom_{\D_{\Psi}}(\omega M,\omega N) \quad \text{and} \quad \epsilon_*\iHom(M,N) \simeq \iHom_{\D_{A}}(\epsilon_*M,\epsilon_* N).
\]
In particular, we have $\sh^{*}\iHom(M,N) \cong \uExt_\Psi^*(M,N)$ as $\Psi$-comodules, and $\epsilon_*\uExt^\ast_{{\Psi}}(M,N) \cong \uExt_A^\ast(\epsilon_*M,\epsilon_*N)$ as $A$-modules. 
\end{lem}
\begin{proof}
We prove the claim about $\epsilon_*$, the one for $\omega$ being similar. Let $M$ be compact. Because $\epsilon_*$ is symmetric monoidal it preserves dualizable objects, and the dual of $\epsilon_*M$ is just $\epsilon_*(DM)$. We thus get equivalences
\[
\epsilon_* \iHom_{{\Psi}}(M,N) \simeq \epsilon_* (DM) \otimes \epsilon_* N \simeq D(\epsilon_* M) \otimes \epsilon_* N \simeq \iHom_{\D_{A}}(\epsilon_* M,\epsilon_* N).
\]
The second part of the claim follows directly from the first part along with the definition of $\sh^*$ and the fact that $\epsilon_*$ is exact, and hence commutes with $\sh^*$. 
\end{proof}

\begin{cor}\label{cor:wgendualizable}
For any flat Hopf algebroid $(A,\Psi)$, the stable category $\Stable_{\Psi}$ is compactly generated by dualizable objects. 
\end{cor}
\begin{proof}
By construction, it suffices to show that a dualizable $\Psi$-comodule $G \in \cG_{\Psi}$ remains dualizable in $\cD_{\Psi}$, i.e., that the canonical map
\[
\xymatrix{c_X\colon\iHom_{\cD_\Psi}(G,A) \otimes X \ar[r] & \iHom_{\cD_\Psi}(G,X)}
\]
is an equivalence for all $X \in \cD_{\Psi}$. Since $\epsilon_*G$ is dualizable in $\cD_{A}$ by \Cref{prop:comoduleprop}, \Cref{lem:derivedforget} and \Cref{lem:derivedhomcohom} show that $\epsilon_*c_X$ is a quasi-isomorphism. Since $\epsilon_*$ is exact, \Cref{lem:epsilonconservative} implies that $\epsilon_*\colon \cD_{\Psi} \to \cD_A$ is conservative, thus the claim follows. 
\end{proof}

We take the following as a definition.

\begin{defn}\label{def:tor}
	For $M,N \in \Stable_\Psi$, the higher Tor groups are defined by
	\[
\Tor_s^{\Psi}(M,N) = \sh_s(M \otimes N),
	\]
	which are $\Psi$-comodules by construction. 
\end{defn}

The next lemma is an immediate consequence of the construction of $\Stable_{\Psi}$.
 
\begin{lem}
The functor $\omega\colon\Stable_{\Psi} \to \D_{\Psi}$ preserves all colimits, and
\[
\sh_*\colim_{i \in I} M_{i} \cong \colim_{i \in I} H_*( M_{i}) 
\]
for every filtered diagram $(M_i)_{i \in I}$ of objects in $\Stable_{\Psi}$.
\end{lem}

Limits in $\Stable_{\Psi}$ are somewhat more complicated to understand, so that we require additional hypotheses on the Hopf algebroid. We shall use the notation ${\lim}^p_{i,\Stable_{\Psi}}D_i$ for the $p$-th homology group of the inverse limit ${\lim}_{i,\Stable_{\Psi}}D_i$ of a system $(D_i)_{i\in\N}$ in $\Stable_{\Psi}$.

\begin{lem}\label{lem:comodlim}
Let $(A,\Psi)$ be a Landweber Hopf algebroid with $A$ Noetherian. If $(D_i)_{i \in I}$ is an inverse system in $\cD_{\Psi}$, then there is a natural isomorphism
\[
\xymatrix{{\lim}^p_{i,\cD_{\Psi}}D_i \ar[r]^-{\simeq} & {\lim}^p_{i,\Stable_{\Psi}}\iota D_i}
\]
for all $p \ge 0$. 
\end{lem}

\begin{proof}
By \Cref{prop:nicestable}, the inclusion functor $\iota$ is fully faithful, hence the counit of the adjunction induces an equivalence $\omega \iota \simeq \Id$. Furthermore, $\iota$ is a right adjoint, so
\[
\xymatrix{\ilim_{\Stable_{\Psi}}\iota D_i \ar[r]^-{\sim} & \iota \ilim_{\cD_{\Psi}}D_i.}
\]
Since $\omega \iota \simeq \Id$, we obtain isomorphisms
\begin{align*}
\ilim^p_{\Stable_{\Psi}}\iota D_i & \cong \sh^{p}(\ilim_{\Stable_{\Psi}}\iota D_i) \\
& \cong \sh^{p}(\iota \ilim_{\cD_{\Psi}}D_i) \\
& \cong \sh^{p}(\omega \iota \ilim_{\cD_{\Psi}}D_i) \\
& \cong \sh^{p}(\ilim_{\cD_{\Psi}}D_i) \\
& \cong \ilim^p_{\cD_{\Psi}}D_i. 
\end{align*}
\end{proof}

In the situation of the lemma, we will write ${\lim}^p_{{\Psi},i}D_i$ unambiguously for the $p$-th derived functor of inverse limit of $(D_i)_{i \in I}$. We will adopt this convention from now on.

\subsection{Comparison to Hovey's model}\label{sec:hovcomparasion}

Our next goal is to show that $\Stable_{\Psi}$ is equivalent to the $\infty$-category underlying the homotopy model structure on the category of (unbounded) chain complexes $\Ch_{\Psi}$ constructed by Hovey in~\cite{hovey_htptheory}. In order to do so, we have to briefly recall Hovey's construction. 

Throughout this section, we assume that $(A,\Psi)$ is an \emph{amenable} Hopf algebroid~\cite[Def.~2.3.2]{hovey_htptheory}; note that by Proposition 2.3.3 of \emph{loc.~cit.}, Adams Hopf algebroids are always amenable. Let $\Ch_{\Psi}$ be the category of unbounded chain complexes of $\Psi$-comodules and, as before, let $\cG$ be a set of representatives for isomorphism classes of dualizable $\Psi$-comodules. There exists a model structure on $\Ch_{\Psi}$ known as the projective model structure in which a map $f$ is a weak equivalence or fibration if $\Hom_{\Ch_{\Psi}}(P,f)$ is a homology isomorphism or surjection for all $P \in \cG$, respectively. 

There is a notion of homotopy isomorphism in $\Ch_{\Psi}$, which should not be confused with being a homotopy equivalence. For $M \in \Comod_\Psi$, let $LM$ denote the cobar complex associated to $M$. For $X \in \Ch_{\Psi},P \in \cG$ and $n \in \Z$, Hovey defines
\[
\pi_n^P(X) = H_{-n}(\Hom_\Psi(P,LA \otimes_A X)). 
\]
A map $f$ in $\Ch_{\Psi}$ is said to be a homotopy isomorphism if it induces an isomorphism on $\pi_n^P$ for all $n\in \Z$ and $P \in \cG$. The homotopy model structure on $\Ch_{\Psi}$ is then constructed as the left Bousfield localization of the projective model structure at the homotopy isomorphisms. 

\begin{defn}
The category of unbounded chain complexes $\Ch_{\Psi}$ equipped with the homotopy model structure will be denoted by $\Stableh_{\Psi}$. By a mild abuse of notation, we will refer to the underlying $\infty$-category of this model category by $\Stableh_{\Psi}$ as well. Furthermore, $\Ho(\Stableh_{\Psi})$ denotes the homotopy category of $\Stableh_{\Psi}$. 
\end{defn}

The basic properties of $\Stableh_{\Psi}$ are summarized in the following result. 

\begin{prop}[Hovey]\label{prop:stablepsi}
The category $\Stableh_\Psi$ is a bigraded algebraic stable homotopy theory with the property that 
\[
\Stableh_\Psi(M,\Sigma^kN) \simeq \Ext_\Psi^k(M,N),
\]
for $M$ in the thick subcategory generated by the dualizable $\Psi$-comodules and $N \in \Comod_{\Psi}$, where $\Sigma^kN$ refers to the comodule $N$ thought of as a complex concentrated in degree $k$. Moreover, $\Stable_{\Psi}$ is compactly generated by the set $\cG$ of dualizable $\Psi$-comodules.
\end{prop}
\begin{proof}
This is proven by Hovey in~\cite[Thm.~6.2]{hovey_htptheory} and~\cite[Prop.~6.10]{hovey_htptheory}, although there is one subtlety: Hovey shows that the dualizable comodules form a set of compact weak generators for $\Stableh_{\Psi}$. However, \cite[Thm.~2.3.2]{hps_axiomatic} implies that $\cG$ is in fact a set of compact generators in the strong sense.  
\end{proof}
 
We now come to our main comparison result, which shows that Hovey's model category $\Stableh_{\Psi}$ is equivalent to the one constructed in~\Cref{sec:indconstruction}. 

\begin{thm}\label{thm:hoveystable}
Suppose $(A,\Psi)$ is an amenable Hopf algebroid, then there is a canonical equivalence of stable categories
\[
\xymatrix{\Stable_{\Psi} \ar[r]^-{\sim} & \Stableh_{\Psi},}
\]
which is compatible with the natural functors to the derived stable $\infty$-category $\cD_{\Psi}$. In particular, this induces an equivalence of stable homotopy categories
\[
\xymatrix{\Ho(\Stable_{\Psi}) \ar[r]^-{\sim} & \Ho(\Stableh_{\Psi}).}
\]
\end{thm}
\begin{proof}
Consider the functor of $\infty$-categories $\Stableh_{\Psi} \to \cD_{\Psi}$ given by left Bousfield localization at the homology isomorphisms. The induced morphism of spaces
\[ 
\xymatrix{\Hom_{\Stableh_{\Psi}}(P,Q) \ar[r] & \Hom_{\cD_{\Psi}}(P,Q)}
\]
is a homotopy equivalence for any pair of dualizable comodules $P,Q \in \cG$ by Hovey's result,~\Cref{prop:stablepsi}, so we obtain a fully faithful lift $\xi$ 
\[ 
\xymatrix{& \Stableh_{\Psi} \ar[d]  \\
\Thick_{\Psi}(\cG) \ar[r] \ar@{-->}[ru]^{\xi} & \cD_{\Psi}}
\]
of the natural inclusion. Consequently, there is an induced functor $\Xi$ fitting into a commutative diagram of $\infty$-categories
\[
\xymatrix{\Thick_{\Psi}(\cG) \ar[r]^j \ar[rd]_{\xi} & \Stable_{\Psi} \ar@{-->}[d]^{\Xi} \\
& \Stableh_{\Psi}.}
\]
Since $\xi$ is fully faithful and $\cG$ is a set of compact generators of $\Stableh_{\Psi}$ by~\Cref{prop:stablepsi}, $\Xi$ is an equivalence of stable $\infty$-categories. Moreover,~\Cref{prop:indsymmmon} implies that this equivalence lifts to an equivalence of stable categories, so the claim follows. 
\end{proof}

\begin{rem}\label{rem:monogenic}
The proof of \Cref{lem:boundedhomology} shows that $\Stable_{\Psi}$ is in fact monogenic whenever $(A,\Psi)$ is a Landweber Hopf algebroid. In this case, we could therefore define $\Stable_{\Psi}$ equivalently as the ind-category on $\Thick_{\Psi}(A)$. This is analogous to Hovey's result~\cite[Thm.~6.6]{hovey_htptheory} which says that $\Stableh_{\Psi}$ is monogenic if the dualizable $\Psi$-comodules are contained in the thick subcategory generated by $A$.  
\end{rem}

\section{Local duality for comodules}\label{sec:comodules-duality}

With the stable category $\Stable_\Psi$ constructed, and its basic properties understood, we now move towards a thorough study of local duality for comodules. We will first construct an abelian category of $I$-power torsion comodules, a full subcategory of $\Comod_\Psi$. This does not appear to have been studied before (although several similar results have been obtained by Sitte~\cite{sitte2014local} via a different approach), but we will see that it is very similar to the module case considered in \Cref{sec:modules}. In fact it turns out the torsion functor commutes with the forgetful functor to $A$-modules, even in the derived sense, see~\Cref{lem:comodtorsionforget}. 

For $I$ a finitely generated regular invariant ideal of $A$, we then study the duality context $(\Stable_\Psi,A/I)$. Our abstract local duality theory provides us with torsion, local, and completion functors and categories. 

We then move on to identifying the torsion functor explicitly. Under some mild conditions on the ring $A$ and the ideal $I$ we see that it takes a form analogous to that considered in the module case, namely $\Gamma_I(M) \simeq \colim_k \iHom_\Psi(A/I^k,M)$. This gives an identification of the completion functor as $\Lambda_I(M) \simeq \psilim( A/I^k \otimes_A M)$. However, (derived) comodule limits are usually more difficult to compute than comodule colimits. 

We finish this section by using Naumann's theorem on the relationship between flat Hopf algebroids and certain algebraic stacks, to transport our results on local duality for comodules to those for a suitable stable category of quasi-coherent sheaves on such stacks. 

\subsection{The abelian category of \texorpdfstring{$I$}{I}-power torsion comodules}\label{sec:abeliantorsioncomod}

Before we study local duality for $\Stable_\Psi$, we first study the $I$-power torsion functor for the abelian category $\Comod_\Psi$. As we will see, the torsion functor we construct in this section can be obtained, at least under suitable conditions, from the homology of the torsion functor constructed abstractly using local duality.
\begin{rem}
	Recall that we write $T_{I}^{A}$ for the $I$-power torsion functor for $A$-modules, to distinguish it from the torsion functor $T_I^\Psi$ for $\Psi$-comodules, which will be constructed below. 
\end{rem}
 Let $I$ be a finitely generated ideal in $A$, so that $A/I$ is compact, and assume additionally that the ideal $I$ is invariant, i.e., that for any $M \in \Comod_\Psi$, $IM$ is a sub-comodule of $M$. We note the following.
\begin{lem}\label{lem:invariantideals}
If $I$ is an invariant ideal, then so is $I^k$ for any $k \ge 0$. 	
\end{lem}
\begin{proof}
	The condition that $I$ is an invariant ideal is equivalent to $\eta_L(I) \subseteq \eta_R(I)\Psi$. It follows then that $\eta_L(I^k) \subseteq \eta_R(I^k)\Psi$, so $I^k$ is invariant. 
\end{proof}

Recall from \eqref{eq:abeliantorsfunctorcolimitdescription} that the $I$-power torsion functor for modules can be identified with
\begin{equation}\label{eq:toramod}
T_I^A(M) \cong \colim_k \Hom_A(A/I^k,M),
\end{equation}
and that we can define a full subcategory $\tMod{A}$ of $\Mod_A$ consisting of those $A$-modules such that $T_I^A(M) \cong M$. 
\begin{defn}
  Let $M$ be a $\Psi$-comodule. Say that $M$ is $I$-power torsion if the underlying $A$-module of $M$ is $I$-power torsion.  
\end{defn}
We will write $\tComod$ for the full subcategory of $I$-power torsion comodules and $\iota$ for the inclusion functor $\tComod \to \Comod_\Psi$.  

\begin{lem}\label{lem:constructiontpsi}
	There exists a right adjoint $T_I^\Psi$ to the inclusion functor $\iota\colon\tComod \to \Comod_\Psi$ given by 
	\[
T_I^\Psi(M) \cong \colim_k \iHom_\Psi(A/I^k,M)
	\]
for $M \in \Comod_{\Psi}$.
\end{lem}
\begin{proof}
Following Hovey~\cite[Sec.~1.2]{hovey_htptheory}, write a comodule $M$ as a kernel of a map of extended comodules
	\[
M \xr{\psi} \Psi \otimes M \xr{ \psi p} \Psi \otimes N,
	\]
where $p\colon\Psi \otimes M \to N$ is the cokernel of $\psi$. It is easy to check via adjointness that we must have 
\[
\xymatrix{\Psi \otimes T_{I}^A(M) \ar[r]^-{\simeq} & T_I^\Psi(\Psi \otimes M).}
\]
Given a map of extended comodules $\Psi \otimes M \to \Psi \otimes N$ we define $T_I^\Psi(f)$ as the composite
\[
\begin{split}
\Psi \otimes T_I^A(M) \xr{\Delta \otimes 1} \Psi \otimes \Psi \otimes T_I^A(M) \xr{1 \otimes \alpha} \Psi \otimes T_I^A(\Psi \otimes M) \\
\xr{1 \otimes T_I^A(f)} \Psi \otimes T_I^A(\Psi \otimes N) \xr{1 \otimes T_I^A(\epsilon \otimes 1)} \Psi \otimes_A T_I^A(N),
\end{split}
\] 
where $\alpha$ is the natural isomorphism. 

One can check that this defines an adjoint to $\iota$ on the full subcategory of extended comodules. We then have no choice but to define $T_I^\Psi(M)$ to be the kernel of $\Psi \otimes T_I^A(M) \to \Psi \otimes T_I^A(N)$. 

Then, using~\eqref{eq:ihomextended} and the fact that colimits commute with the tensor product, we have
\[
\Psi \otimes \colim_k \Hom_A(A/I^k,M) \cong \colim_k \iHom_\Psi(A/I^k, \Psi \otimes M), 
\] 
and so the kernel, $T_I^\Psi(M)$, can be identified with
\[
\colim_k \iHom_\Psi(A/I^k,M).
\]
One can easily check that this is right adjoint to the inclusion functor, as required. 
\end{proof}

\begin{defn}
We call the functor $T_I^\Psi\colon \Comod_{\Psi} \to \tComod$ constructed in \Cref{lem:constructiontpsi} the $I$-power torsion functor on $\Comod_{\Psi}$. If the Hopf algebroid $(A,\Psi)$ is clear from context, the superscript $\Psi$ will be omitted from the notation.
\end{defn}

Since the underlying $A$-module of $T_I^\Psi(M)$ is $I$-power torsion, an obvious guess is that it is just given by $T_I^A(M)$. This turns out to be the case. 
\begin{cor}\label{cor:forgettorsion}
	There is an isomorphism of $A$-modules $\epsilon_*T_I^\Psi(M) \cong T_I^A(\epsilon_*M)$.
\end{cor}
\begin{proof}
	Since $A/I^k$ is finitely presented, this follows from Part 2 of~\Cref{prop:comoduleprop} and the fact that $\epsilon_*$ is exact, and so commutes with colimits. 
\end{proof}

There is an alternative description of $T_I^\Psi$, just as in the module case. Let $M \in \Comod_\Psi$ and define a functor $\widetilde T_{I}^{\Psi}$ on $\Comod_{\Psi}$ by
\begin{equation}\label{eq:itorsion}
\widetilde T_{I}^{\Psi}(M) = \{ m \in M | \text{ there exists } n \in \N\colon I^n m = 0\}.
\end{equation}

\begin{lem}\label{lem:torsionfunctorcomod}
	There is a natural isomorphism $T_I^\Psi(M) \xr{\simeq} \widetilde T_{I}^{\Psi}(M)$ of $\Psi$-comodules. 
\end{lem}
\begin{proof}
Consider the comodule $\iHom_\Psi(A/I^k,M)$. By \Cref{prop:comoduleprop} this is equivalent to $\Hom_A(A/I^k,M)$ as an $A$-module, with structure map arising as the composite
\[
\Hom_A(A/I^k,M) \xr{\psi_M^\ast} \Hom_A(A/I^k,\Psi \otimes M) \cong \Psi \otimes \Hom_A(A/I^k,M).
\] 
We define a morphism
\[
\xymatrix@R=0.3em{
	\iHom_\Psi(A/I^k,M) \ar[r] & \{ m \in M | I^km = 0\} \subseteq \widetilde T_I^\Psi M .\\ 
f \ar@{|->}[r] & f(1). 
}
\]
One can check that this is a morphism of comodules and, since it is obviously an isomorphism after applying $\epsilon_*$, it is an isomorphism by~\Cref{lem:epsilonconservative}. These isomorphisms form a directed system which induces the claimed isomorphism $T_I^\Psi(M) \xr{\simeq} \widetilde T_{I}^{\Psi}(M)$. 
\end{proof}

Using the fact that $\Mod_A^{I-\text{tors}}$ is an abelian category, it is standard to check that the same is true for $\tComod$; this is just the same argument as in~\cite[Thm.~A1.1.3]{ravenel_86}. In fact, $\tComod$ is Grothendieck abelian.

\begin{prop}
The category $\tComod$ of $I$-power torsion comodules over a flat Hopf algebroid $(A,\Psi)$ is a Grothendieck abelian category.  
\end{prop}
\begin{proof}
We use the following result~\cite[Lem.~11.2]{hasimoto}: If $\mathcal{A}$ is a cocomplete  abelian category, and $\mathcal{B}$ is a Grothendieck category such that there is an adjunction 
  \[
\xymatrix{F\colon\mathcal{A} \ar@<0.5ex>[r] & \mathcal{B}\colon G \ar@<0.5ex>[l]}
  \]
  with the left adjoint $F$ being faithful exact, and $G$ preserving filtered colimits, then $\mathcal{A}$ is also Grothendieck. We want to apply this result to the case where $\mathcal{A}$ is $\Comod_\Psi^{I-\text{tors}}$ and $\mathcal{B}$ is $\Comod_{\Psi}$, with $F$ and $G$ the inclusion and torsion functors respectively, i.e., to the following:
  \[
\xymatrix{\iota_{\mathrm{tors}}\colon \Comod_\Psi^{I-\text{tors}} \ar@<0.5ex>[r] & \Comod_{\Psi}\colon T_I^{\Psi}. \ar@<0.5ex>[l]}
  \]
Since colimits in $\Comod_\Psi^{I-\text{tors}}$ can be computed in $\Comod_{\Psi}$, we see that $\Comod_\Psi^{I-\text{tors}}$ is cocomplete. It is clear that $\iota_{\mathrm{tors}}$ is faithful and exact, and \Cref{lem:constructiontpsi} implies that $T_I^{\Psi}$ preserves colimits. Thus we can apply the stated result to conclude that $\Comod_\Psi^{I-\text{tors}}$ is also a Grothendieck category. 
\end{proof}

It follows from \Cref{cor:forgettorsion} that we have a commutative diagram of adjoint functors
\[
\xymatrix@C=2cm{
  \Mod_A  \ar@<0.5ex>[r]^-{\epsilon^*} \ar@<0.5ex>[d]^{T_I^A} & \ar@<0.5ex>[d]^{T_I^\Psi} \Comod_\Psi \ar@<0.5ex>[l]^-{\epsilon_*}\\
  \Mod_A^{I-\text{tors}} \ar@<0.5ex>[r]^-{\epsilon^*} \ar@<0.5ex>@{^{(}->}[u]& \tComod. \ar@<0.5ex>[l]^-{\epsilon_*} \ar@<0.5ex>@{^{(}->}[u]
}
\]
The bottom horizontal functors are the cofree/forgetful adjunction: Given $M \in \tMod{A}$ it is not hard to check that $\Psi \otimes_A M \in \tComod$, using the description of $\iHom_\Psi(-,-)$ on extended $\Psi$-comodules, see~\eqref{eq:ihomextended}. 

The functor $T_I^\Psi$ is evidently left exact, and so has right derived functors $\mathbf{R}^sT_I^\Psi = H^s\mathbf{R}T_I^\Psi$ for $s \ge 0$. Since colimits are exact in $\Comod_\Psi$, a simple argument with the Grothendieck spectral sequence shows that 
\begin{equation}\label{eq:derivedtorcomodules}
\mathbf{R}^sT_I^\Psi(-) \cong \colim_k \uExt^s_\Psi(A/I^k,-).
\end{equation}

We have seen previously that the underlying $A$-module of $T^\Psi_I(M)$ is just the torsion functor applied to the underlying $A$-module. We claim that the underlying $A$-modules of the derived functors of $T^\Psi_I$ are also just the derived functors of torsion in $A$-modules. 
\begin{lem}\label{lem:comodtorsionforget}
For $s \ge 0$ there is an equivalence
	\[
\epsilon_*\mathbf{R}^sT_I^\Psi(-) \cong \mathbf{R}^sT_I^A(\epsilon_*(-)) \cong \colim_k \Ext^s_A(\epsilon_*A/I^k,\epsilon_*(-)). 
	\]
\end{lem}
\begin{proof}
	Using~\Cref{lem:derivedhomcohom}, the argument is analogous to the one in~\Cref{cor:forgettorsion}.
\end{proof}

\subsection{Local duality for comodules}

We now wish to construct local cohomology and local homology functors satisfying local duality for comodules; to be more precise, we consider the local duality context $(\Stable_{\Psi},A/I)$ for a flat Hopf algebroid $(A,\Psi)$ and $I \subseteq A$ a finitely generated invariant ideal. This appears to require some conditions on the ideal $I$; in the case of modules we could weaken this condition by using the Dwyer--Greenlees Koszul object as a substitute for $A/I$, but we do not have this in the comodule case. Thus, we need an additional regularity condition.  In what follows, we remind the reader that $\cG$ denotes the set of dualizable $\Psi$-comodules. 
\begin{lem}\label{lem:AIcompact}
	Suppose $I$ is generated by a finite invariant regular sequence $(x_1,\ldots,x_n)$, then $A/I^k$ is compact and dualizable in $\Stable_\Psi$ for all $k \ge 0$.
\end{lem}
\begin{proof}
By construction, the image of $M \in \Comod_\Psi$ in $\Stable_\Psi$ is compact if $M$ is in the thick subcategory generated by $\cG$. Of course, since $A \in \cG$, it is enough to check that $M \in \Thick_\Psi(A)$, which for $A/I$ is shown in~\cite[Lem.~6.5]{hovey_htptheory} using the short exact sequences of comodules
\begin{equation}\label{eq:sesamodi}
0 \to A/(x_1,\ldots,x_{i-1}) \xr{x_i} A/(x_1,\ldots,x_{i-1}) \to A/(x_1,\ldots,x_i) \to 0
\end{equation}
and induction. 

Since $I$ is regular, there is an isomorphism of $A$-modules~\cite[Thm.~1.1.8]{cmrings}
\[
\bigoplus_{k \ge 0} I^k/I^{k+1} \cong A/I[x_1,\dots,x_n];
\]
as $I$ is finitely generated, this shows that each $I^k/I^{k+1}$ is finitely generated and projective over $A/I$. Because $A/I$ is compact in $\Stable_{\Psi}$, and is a finitely presented $A$-module, it follows that $I^k/I^{k-1}$ is also compact in $\Stable_{\Psi}$. The claim that $A/I^k$ is compact then follows by induction on the fiber sequences induced by the short exact sequences of comodules 
\[
0 \to I^{k-1}/I^k \to A/I^k \to A/I^{k-1} \to 0.
\]
To see that $A/I^k$ is dualizable we observe that any non-zero object in $\Thick_\Psi(A)$ 
is dualizable, see \Cref{lem:algcompdual} and \Cref{cor:wgendualizable}. 
\end{proof}

We will consider the local duality context $(\Stable_\Psi,A/I)$; we start with the following definition.

\begin{defn}
Let $I \subseteq A$ be a finitely generated ideal, generated by an invariant regular sequence. Define the category $\Stable_\Psi^{I-\mathrm{tors}}$ of $I$-torsion $\Psi$-comodules as
\[\Stable_\Psi^{I-\mathrm{tors}} = \Locid{\Stable_\Psi}(A/I),\]
the localizing ideal in $\Stable_\Psi$ generated by the compact object $A/I$. The inclusion of the category $\Stable_\Psi^{I-\text{tors}}$ of $I$-torsion $\Psi$-comodules into $\Stable_\Psi$ will be denoted $\iota_{\text{tors}}$. 
\end{defn}
\color{black}
By \Cref{prop:indproperties}, $\Stable_\Psi^{I-\mathrm{tors}}$ is a stable category compactly generated by $A/I \otimes \mathcal{G}$, so $\iota_{\text{tors}}$ admits a right adjoint $\Gamma_I$ by \Cref{prop:indrightadjoint}, 
\[
\xymatrix{\iota_{\mathrm{tors}}\colon \Stable_\Psi^{I-\mathrm{tors}} \ar@<0.5ex>[r] & \Stable_\Psi\colon \Gamma_I.\ar@<0.5ex>[l]}
\]
Therefore, we can apply the results of Section~\ref{sec:ald} to obtain localization and completion adjunctions 
\[\xymatrix{\Stable_\Psi^{I-\mathrm{loc}} \ar@<0.5ex>[r]^-{\iota_{\mathrm{loc}}} & \Stable_\Psi \ar@<0.5ex>[l]^-{L_I} \ar@<0.5ex>[r]^-{\Lambda^I} & \Stable_\Psi^{I-\mathrm{comp}} \ar@<0.5ex>[l]^-{\iota_{\mathrm{comp}}}
}\]
with respect to the ideal $I$, as well as a version of local duality.

We are now in a position to state our version of local duality for comodules, which follows from~\Cref{thm:hps} and \Cref{cor:wgendualizable} by specializing to the local duality context $(\Stable_{\Psi},A/I)$. 
\begin{thm}\label{thm:localdualitystablepsi}
Let $(A,\Psi)$ be a flat Hopf algebroid and let $I \subseteq A$ be a finitely generated invariant regular ideal. 
\begin{enumerate}
 \item The functors $\Gamma_I$ and $L_I$ are smashing colocalization and localization functors and $\Lambda^I$ is the localization functor with category of acyclics given by $\Stable_\Psi^{I-\mathrm{loc}}$. Moreover, $\Gamma_I$ and $\Lambda^I$ satisfy $\Lambda^I \circ \Gamma_I \simeq \Lambda^I$ and $\Gamma_I \circ \Lambda^I \simeq \Gamma_I$ and they induce mutually inverse equivalences
 \[\xymatrix{\Stable_\Psi^{I-\mathrm{comp}} \ar@<1ex>[r]^{\Gamma_I}_{\sim} & \Stable_\Psi^{I-\mathrm{tors}}. \ar@<1ex>[l]^{\Lambda^I}}\]
 \item For $M,N \in \Stable_\Psi$, there is an equivalence
 \[\xymatrix{\iHom_\Psi(\Gamma_IM,N) \ar[r]^{\sim} & \iHom_\Psi(M,\Lambda^IN),}\]
 natural in each variable.
 \end{enumerate}
\end{thm}
As usual, we can represent this by the following diagram of stable categories
  \begin{equation}\label{eq:derivedcomod}
\xymatrix{& \Stable_\Psi^{I-\textrm{loc}} \ar@<0.5ex>[d] \ar@{-->}@/^1.5pc/[ddr] \\
& \Stable_\Psi \ar@<0.5ex>[u]^{L_I} \ar@<0.5ex>[ld]^{\G_I} \ar@<0.5ex>[rd]^{\Lambda^I} \\
\Stable_\Psi^{I-\textrm{tors}} \ar@<0.5ex>[ru] \ar[rr]_{\sim} \ar@{-->}@/^1.5pc/[ruu] & & \Stable_\Psi^{I-\textrm{comp}}. \ar@<0.5ex>[lu]}
\end{equation}

\begin{rem}
 By~\Cref{lem:inddiscrete}, for a discrete Hopf algebroid $(A,A)$, there is an equivalence $\Stable_A \simeq \D_A$ and it follows that $\Stable_A^{I-\mathrm{tors}} \simeq \D_A^{I-\mathrm{tors}}$, and similarly for the local and complete categories. Therefore, local duality for $A$-modules is a special case of local duality for comodules (although we require stronger conditions on the ideal in the comodule case). 
\end{rem}

\subsection{Compatibility with the discrete case}
We have seen in~\Cref{cor:forgettorsion,lem:comodtorsionforget} that in the abelian case, the torsion functor $T_I^\Psi$ (and its derived functors $\mathbf{R}^sT_I^\Psi$) are compatible with the forgetful functor, in the sense that there is an isomorphism $\epsilon_*\mathbf{R}^sT_I^\Psi(M) \cong \mathbf{R}^sT_I^A(\epsilon_*M)$ for any $\Psi$-comodule $M$ and $s \ge 0$. Of course, we can ask a similar question for the derived torsion functor. In fact, we shall see that both the torsion and local functors are compatible with the forgetful functor to $\mathcal{D}_A$; this is true for abstract reasons, and does not rely on an explicit identification of the torsion functor.  The main result of this section,~\Cref{prop:comparison_Stable-Derivedcat}, shows even more, namely that the torsion and local functors are compatible with both the forgetful functor $\epsilon_*$ and its right adjoint $\epsilon^\ast$. 
\begin{rem}
 	In this subsection we will write $\Gamma_I^\Psi$ and $\Gamma_I^A$ to distinguish the torsion functors constructed on $\Stable_\Psi$ and $\mathcal{D}_A$, respectively, and similarly for $L_I^\Psi$ and $L_I^A$.  
 \end{rem} 
By construction, the forgetful functor $\epsilon_*$ fits into a commutative diagram
\[
\xymatrix{\Stable_{\Psi}^{I-\mathrm{tors}} \ar@{^{(}->}[rrr]^-{\iota_{\mathrm{tors}}} \ar@{-->}[dd]_{\epsilon_*} &&& \Stable_{\Psi} \ar[dd]^{\epsilon_*} \\
& \Thick_{\Psi}(A/I \otimes \cG) \ar@{^{(}->}[r] \ar[lu]_{j} \ar[ld] & \Thick_{\Psi}(\cG) \ar[ru]^j \ar[rd]\\
\D_A^{I-\mathrm{tors}} \ar@{^{(}->}[rrr]_{\iota_{\mathrm{tors}}} &&& \D_A.}
\]
Using \Cref{prop:indsymmmon}, it follows that $\epsilon_*$ restricts to a colimit preserving and symmetric monoidal functor $\Stable_{\Psi}^{I-\mathrm{tors}} \to \D_A^{I-\mathrm{tors}}$ which we will also call $\epsilon_*$. The tensor product in $\Stable_{\Psi}^{I-\mathrm{tors}}$ ($\D_A^{I-\mathrm{tors}}$) agrees with the one in $\Stable_\Psi$ ($\D_A$); the unit is given by $\Gamma^\Psi_I(A)$ ($\Gamma^A_I(A)$).
 
\begin{lem}\label{lem:stableforget}
	For any $M \in \Stable_{\Psi}$ there is an equivalence in $\cD_A$
	\[
\epsilon_* \Gamma^\Psi_I(M) \simeq \Gamma^A_I(\epsilon_*M). 
	\]
\end{lem}
\begin{proof}
	Since $\Gamma^\Psi_I$ is smashing and $\epsilon_*$ preserves colimits, it suffices to check the claim for $M=A$. But $\epsilon_*$ is symmetric monoidal, hence preserves the unit object, so $\epsilon_* \Gamma^\Psi_I(A) \simeq \Gamma^A_I(\epsilon_*A)$.
\end{proof}

Similarly, the next lemma shows that $\epsilon^*(-) = \Psi \otimes_A -$ is compatible with the torsion functors on $\D_A$ and $\Stable_{\Psi}$.

\begin{lem}\label{cor:torexchange}
If $I$ is generated by a finite regular invariant sequence and $N$ is in $ \cD_A$, then there is an equivalence
	\[
\Gamma^\Psi_I(\Psi \otimes N) \simeq \Psi \otimes \Gamma^A_I(N),
	\]
	where $\Psi \otimes \Gamma^A_I(N)$ has the extended comodule structure. 
\end{lem}

\begin{proof}
	By adjunction, for any $M \in \Stable_{\Psi}^{I-\mathrm{tors}}$ we have
\[
\begin{split}
\Hom_{\Psi}(M,\Gamma^\Psi_I(\Psi \otimes N)) & \simeq \Hom_{\Psi}(M,\Psi \otimes N) \\
& \simeq \Hom_{A}(\epsilon_*M, N) \\
& \simeq \Hom_{A}(\epsilon_*M,\Gamma^A_I N) \\
& \simeq \Hom_{\Psi}(M,\Psi \otimes \Gamma^A_I N),
\end{split}
\]
so the Yoneda lemma implies the claim. 	
\end{proof}

\begin{prop}\label{prop:comparison_Stable-Derivedcat}
	The adjoint pair $(\epsilon_*,\epsilon^*)$ restricts to the torsion and local subcategories. More precisely, if $I$ is an ideal generated by a regular invariant  sequence, then there is a commutative\footnote{Here commutative means that all the 8 possible ways the diagram can commute, do indeed do so.} diagram of adjoint functors
	\[
\xymatrix{
\Stable_\Psi^{I-\mathrm{tors}} \ar@<0.5ex>[r]^-{\iota_{\mathrm{tors}}} \ar@<0.5ex>[d]^-{\epsilon_*} & \ar@<0.5ex>[d]^-{\epsilon_*}\ar@<0.5ex>[l]^-{\Gamma^\Psi_I} \Stable_\Psi \ar@<0.5ex>[r]^-{L_I^\Psi} & \ar@<0.5ex>[l]^{\iota_{\mathrm{loc}}}\Stable_{\Psi}^{I-\mathrm{loc}}\ar@<0.5ex>[d]^-{\epsilon_*} \\
\D_A^{I-\mathrm{tors}} \ar@<0.5ex>[r]^{\iota_{\mathrm{tors}}} \ar@<0.5ex>[u]^-{\epsilon^*} &\ar@<0.5ex>[u]^-{\epsilon^*} \ar@<0.5ex>[l]^-{\Gamma^A_I} \D_A \ar@<0.5ex>^-{L_I^A}[r]& \ar@<0.5ex>[l]^{\iota_{\mathrm{loc}}}\D_A^{I-\mathrm{loc}}\ar@<0.5ex>[u]^-{\epsilon^*}.
}
	\]

\end{prop}
\begin{proof}
	We have already seen above that $\epsilon_*$ restricts to a functor on the torsion categories. To show that the right adjoint $\epsilon^*$ restricts to $\cD_A^{I-\text{tors}}$, we note the following: if $f$ is an equivalence in $\cD_A$, then $\Psi \otimes f$ is an equivalence in $\Stable_\Psi$ by~\Cref{lem:derivedforget}. Suppose then that $M \in \cD_A^{I-\text{tors}}$, so that there is a quasi-isomorphism
	\[
\xymatrix{\Gamma^A_I(M) \ar[r]^-{\sim} & M.}
	\]
	In $\Stable_\Psi$, there is an equivalence
	\[
\xymatrix{\Psi \otimes \Gamma^A_I(M) \ar[r]^-{\sim} & \Psi \otimes M.}
	\]
By~\Cref{cor:torexchange} this gives rise to an equivalence $\Gamma^\Psi_I(\Psi \otimes M) \simeq \Psi \otimes M$. So, for $M \in \cD_A^{I-\text{tors}}$ we do indeed have $\epsilon^*M = \Psi \otimes M \in \Stable_\Psi^{I-\text{tors}}$. The pair $(\epsilon_*,\epsilon^*)$ is evidently adjoint. 

For the local categories, note that for $N \in \cD_A$ we have a morphism of cofiber sequences arising from~\Cref{thm:hps}
\[
\xymatrix{
	\Gamma^\Psi_I(\Psi \otimes N) \ar[r] \ar[d] & \Psi \otimes N \ar[r] \ar[d]_= & L^\Psi_I(\Psi \otimes N) \ar[d] \\
	\Psi \otimes \Gamma^A_I(N) \ar[r] & \Psi \otimes N \ar[r] & \Psi \otimes L^A_I(N). 
}
\]
By~\Cref{cor:torexchange} the leftmost arrow is an equivalence; since the middle arrow is also clearly an equivalence, so is the rightmost arrow. 

Now assume that $N \in \cD_A^{I-\text{loc}}$. We would like to show that $\epsilon^*(N) \simeq \Psi \otimes N \in \Stable_\Psi^{I-\text{loc}}$. Once again, an equivalence $N \simeq L^A_I(N)$ in $\cD_A$ gives rise to an equivalence $\Psi \otimes N \simeq \Psi \otimes L^A_IN$ in $\Stable_\Psi$. Thus, we see that $\Psi \otimes N \simeq \Psi \otimes L^A_IN \simeq L^\Psi_I(\Psi \otimes N)$, so that $\Psi \otimes N \in \Stable_\Psi^{I-\text{loc}}$, as required.  

Consider the diagram
\[
\xymatrix{
	\epsilon_*\Gamma^\Psi_IM \ar[r] \ar[d] & \epsilon_* M \ar[r] \ar[d]_{=} & \epsilon_*L^\Psi_IM \ar[d] \\
	\Gamma^A_I(\epsilon_*M) \ar[r] & \epsilon_*M \ar[r] & L^A_I(\epsilon_*M). 
}
\]
By~\Cref{lem:stableforget} the left-most vertical arrow is an equivalence, and hence there is an equivalence $\epsilon_*L^\Psi_IM \simeq L^A_I(\epsilon_*M)$; in particular, if $M \in \Stable_\Psi^{I-\text{loc}}$, then $\epsilon_*M \simeq L^A_I(\epsilon_*M)$, so that $\epsilon_*M \in \cD_A^{I-\text{loc}}$. Once again, the pair $(\epsilon_*,\epsilon^*)$ is evidently adjoint. 
\end{proof}

\subsection{The torsion functor}\label{sec:torsionfunctor}

If $(A,\Psi)$ is an appropriate Hopf algebroid, \Cref{lem:stableforget} can be lifted to a characterization of the torsion functor on $\Stable_{\Psi}$, as we will prove below. By \Cref{lem:AIcompact}, there exists an inverse sequence under $A$
\[
\xymatrix{\ldots \ar[r] & A/I^3 \ar[r] & A/I^2 \ar[r] & A/I}
\]
of compact and hence dualizable comodules in $\Stable_{\Psi}$, which upon dualizing yields a sequence
\begin{equation}\label{eq:koszulsequence}
\xymatrix{DA/I \ar[r] & DA/I^2 \ar[r] & DA/I^3 \ar[r] & \ldots}
\end{equation}
over $DA \simeq A$.

\begin{prop}\label{prop:identtorsfun}
Suppose $(A,\Psi)$ is a flat Hopf algebroid with $A$ Noetherian and let $I$ be a finitely generated ideal of $A$ generated by a regular invariant sequence, then the functor $\Gamma_I$ constructed abstractly above is naturally equivalent to the functor $\Gamma_I'\colon \Stable_{\Psi} \to \Stable_{\Psi}^{I-\text{tors}}$ given by 
\[M \mapsto \colim_k \iExt_{\Psi}(A/I^{k},M).\]
\end{prop}
\begin{proof}
First, observe that $\Gamma_I'$ takes values in the category of $I$-torsion comodules: Indeed, by compactness of $A/I^k$ and the fact that $\Stable_{\Psi}^{I-\text{tors}}$ is closed under colimits, it suffices to see that $\iExt_{\Psi}(A/I^k,A)$ is $I$-torsion for all $k$, which is clear. Secondly, we note that $\Gamma_I'$ is smashing:
\begin{align*}
\Gamma_I'(M) & \simeq \colim_k \iExt_{\Psi}(A/I^k,M)\\
& \simeq M \otimes \colim_k \iExt_{\Psi}(A/I^k,A) \\
& \simeq M \otimes \Gamma_I' (A)
\end{align*}
for any $M \in \Stable_{\Psi}$, as $A/I^k$ is compact. Hence the conditions of~\Cref{lem:torsionred} are satisfied, and it is enough to show that there is an equivalence
\begin{equation}\label{eq:koszulequiv2}
\xymatrix{\Gamma'(A) \otimes A/I \ar[r]_-{\sim}^-{f} & A/I.}
\end{equation}

We note that there is such a natural map: as in \eqref{eq:koszulsequence}, each $DA/I^k$ maps naturally (and compatibility) to $A$, and hence there is a natural morphism $\Gamma'(A) \simeq \colim_k DA/I^k \to A$. Tensoring with $A/I$ gives the desired morphism $f$.  

Recall that the functor $\omega\colon \Stable_{\Psi} \to \cD_{\Psi}$ factors as 
\[
\xymatrix{\Stable_{\Psi} \ar[r]^-{\theta} & \Dinjpsi \ar[r]^-{\omega'} & \cD_{\Psi},}
\]
where $\theta$ is fully faithful by~\Cref{lem:thetaff} and $\omega'$ induces an equivalence $(\cD_{\Psi})_{\le 0} \simeq (\Dinjpsi)_{\le 0}$ by~\cite[Rems.~1.3.6.3 and~1.3.6.4]{ha}. By construction, $\Gamma'(A) \simeq \colim_k DA/I^k \in (\Stable_{\Psi})_{\le 0}$, so it suffices to check that~\eqref{eq:koszulequiv2} is a quasi-isomorphism. Since $\epsilon_*\colon\Comod_\Psi \to \Mod_A$ is conservative we can reduce further to showing that $\epsilon_*H_*(f)$ is an equivalence of $A$-modules. Applying~\Cref{lem:torsionred} to the discrete Hopf algebroid $(A,A)$ we see that it is enough to check that the torsion functor on $\cD_A$ takes the form $\colim_k \iHom_A(A/I^k,M)$. The latter is proven in \cref{prop:Gamma=colimHom}.
\end{proof}

\begin{cor}\label{cor:stablecomodcompletion}
For $A$ and $I$ as above, the completion functor $\Lambda_I(-)$ is given by
	\[
	\xymatrix@C=0.1em{
\Lambda_I(M) \simeq  \psilim A/I^k \otimes M,
}
	\]
	for all $M \in \Stable_\Psi$. 
\end{cor}
\begin{proof}
	Following the proof of~\Cref{cor:completion} we see there is an equivalence
	\[
		\xymatrix@C=0.1em{
\Lambda_I M \simeq \psilim D^2(A/I^k) \otimes M,
}
	\]
	but since $A/I^k$ is dualizable by~\Cref{lem:AIcompact}, $D^2(A/I^k) \simeq A/I^k$. 
\end{proof}
If we are only interested in the local cohomology comodules, we can ease the assumptions on the ring $A$; but see also~\Cref{rem:krausegeneral}. 

\begin{prop}\label{lem:partialdertorsion}
Let $(A,\Psi)$ be a flat Hopf algebroid and $I$ a finitely generated ideal generated by an invariant regular sequence, then the natural map
\[
\xymatrix{\colim_k \uExt_{\Psi}^*(A/I^k,M) \ar[r]^-{\simeq} & \sh_{-\ast}\Gamma_IM}
\]
is an isomorphism for all $M \in \Stable_{\Psi}$. 
\end{prop}
\begin{proof}
Suppose $M \in \Stable_{\Psi}$. Using the equivalences
\[
\Hom_{\Psi}(M,A/I^k\otimes M) \simeq \Hom_{\Psi}(\iHom_{\Psi}(A/I^k,M),\Gamma_I^{\Psi}M)
\]
we can construct a natural map 
\[
\xymatrix{f\colon \colim_k \iHom_{\Psi}(A/I^k,M) \ar[r] & \Gamma_I^{\Psi}(M)}
\]
in $\Stable_{\Psi}$. We claim that $\omega(f)$ is a quasi-isomorphism. Since $\epsilon_*\colon \D_{\Psi} \to \D_{A}$ is conservative, this follows from the commutative diagram
\[
\xymatrix{\epsilon_*\colim_k \iHom_{\Psi}(A/I^k,M) \ar[r]^-{\epsilon_*(f)} \ar[d]_{\sim} & \epsilon_*\Gamma_I^{\Psi}(M) \ar[d]^{\sim} \\
\colim_k \iHom_{A}(A/I^k,\epsilon_*M) \ar[r]_-{\sim} & \Gamma_I^A\epsilon_*M.}
\]
Here, the vertical maps are quasi-isomorphisms by \Cref{lem:derivedhomcohom} and \Cref{lem:stableforget}, while the bottom quasi-isomorphism uses \cref{prop:Gamma=colimHom}. The proposition is then a consequence of \Cref{lem:derivedhomcohom}. 
\end{proof}

Combining this result with \Cref{lem:comodtorsionforget} shows that the local cohomology functor $\Gamma_I$ is the right derived functor of $T_I^{\Psi}$, in the following sense:

\begin{cor}
For any $M \in \Stable_\Psi$, there is an equivalence $\mathbf{R}T_I^{\Psi}\omega M \simeq \omega \Gamma_IM$.
\end{cor}

For $1 \le k \le n$, let $I_k = (x_1,\ldots,x_k)$ (so that $I_n=I$), and suppose that each $I_k$ is an invariant ideal. 
\begin{defn}\label{defn:stronglyinvariant}
	 We call the ideal $I$ strongly invariant if, for $1 \le k \le n$ and every comodule $M$ which is $I_{k-1}$-torsion as an $A$-module, there is a comodule structure on $x_k^{-1}M$ such that the natural homomorphism $M \to x_k^{-1}M$ is a comodule morphism.
\end{defn}
\begin{exmp}\label{rem:stronglyinvariant}
	As an example, this is satisfied for $BP_*BP$-comodules and the ideal $I_n = (p,v_1,\ldots,v_{n-1})$ for all $n \ge 0$, by~\cite[p.~439]{miller_ravenel_local}. Applying the functor $ - \otimes_{BP_*}E_*$, for $E$ some variant of height $n$ Morava $E$-theory or Johnson--Wilson theory, we see that this is also satisfied for $E_*E$-comodules, and the ideal $I_k = (p,v_1,\ldots,v_{k-1})$ for $0 \le k \le n$. 
\end{exmp}

For $0 \le i \le n$ and $I$ strongly invariant, we construct comodules $A/(x_1^\infty,\ldots,x_i^\infty)$ recursively by first setting $A/(x_0^\infty) = A$, and then defining $A/(x_1^\infty,\ldots,x_{i}^\infty)$ via the short exact sequences 
\begin{equation}\label{eq:sesstabcomodules}
0  \to A/(x_1^\infty,\ldots,x_{i-1}^\infty) \to x_i^{-1} A/(x_1^\infty,\ldots,x_{i-1}^\infty) \to A/(x_1^\infty,\ldots,x_{i}^\infty)\to 0. 
\end{equation} 
\begin{rem}
 We note that the comodules constructed in this manner agree with the modules defined in the same way after applying the forgetful functor. 
\end{rem}

\begin{lem}\label{cor:stronglyinvarianttorsion}

	If the ideal $I$ is finitely generated, strongly invariant, and regular, then 
	\[
\Gamma_I(A) = \Sigma^{-n} A/(x_1^\infty,\ldots,x_n^\infty). 
	\]
\end{lem}
\begin{proof}

	The short exact sequences of~\eqref{eq:sesstabcomodules} give rise to cofiber sequences
	\[
A/(x_1^\infty,\ldots,x_{i-1}^\infty) \to x_i^{-1} A/(x_1^\infty,\ldots,x_{i-1}^\infty) \to A/(x_1^\infty,\ldots,x_{i}^\infty)
	\]
in $\Stable_\Psi$. Since $A/I \otimes x_{i}^{-1}A/(x_1^{\infty},\ldots,x_{i-1}^{\infty}) = 0$, tensoring the cofiber sequence associated to~\eqref{eq:sesstabcomodules} with $A/I$ yields a natural equivalence
\[\xymatrix{A/I \otimes A/(x_1^{\infty},\ldots,x_{i}^{\infty}) \ar[r]^-{\sim} & \Sigma A/I \otimes A/(x_1^{\infty},\ldots,x_{i-1}^{\infty})}\] 
for any $1\le i\le n$. Composing these we see that 
\[
A/(x_1^\infty,\ldots,x_n^\infty) \otimes A/I \simeq \Sigma^{n}A/I, 
\]
so that $\Gamma_I(A) \simeq \Sigma^{-n} A/(x_1^\infty,\ldots,x_n^\infty)$. 
\end{proof}

\subsection{Local duality for algebraic stacks}\label{sec:adamsstacks}

Stacks were introduced as a generalization of schemes with the objective of studying moduli problems. In many cases, they can be thought of as the quotient of a scheme by a group action. In what follows, we will only consider algebraic stacks as defined in \cite[Sec.~2]{naumann_stack_2007}, since they correspond, for a fixed presentation, to flat Hopf algebroids. Note that this is not the standard definition of an algebraic stack, as in \cite[Tag 026N]{stacks-project} or \cite{laumon_champs_2000}, for example. Moreover, there is a notion of quasi-coherent sheaves over a stack and it is here where the connection with comodules arises; in particular, as proven by Naumann in~\cite[Thm.~8]{naumann_stack_2007}, the category of quasi-coherent sheaves over such stacks is equivalent to the category of comodules over the corresponding Hopf algebroid.

In light of this, it is clear that our results for comodules translate easily into a theory of local duality for algebraic stacks. Furthermore, we will see that under some mild extra hypothesis we can identify the torsion functor in a more explicit fashion.  

We now briefly recall the theory of algebraic stacks, following~\cite{naumann_stack_2007}. Let $S$ be an affine scheme and let $\Aff_{/S}$ (or simply $\Aff$ when $S$ is understood) be the category of $S$-schemes. In the following we will only consider $\Aff_{/S}$ together with the \emph{fpqc} topology; a finite family $\{U_i\to U\}$ is an \emph{fpqc} covering in $\Aff_{/S}$ if $\bigsqcup U_i\to U$ is faithfully flat. 

Recall that a \emph{stack} $\fX$ is a category fibered in groupoids such that:
\begin{enumerate}
	\item (``descent of morphisms") For $U\in \Ob(\Aff)$ and $x,y\in \Ob(\fX_U)$ the presheaf
	\[
  		(V\xrightarrow{\phi} U) \mapsto \Hom_{\fX_V}(x|V,y|V),
\]
is a sheaf.
	\item  (``gluing objects") If $\{U_i\to U\}$ is covering in $\Aff$, $x_i\in \Ob(\fX_{U_i})$ and $f_{ji}: x_i|U_i\times_U U_j \to x_j|U_j\times_U U_j$ are isomorphisms satisfying the cocycle condition, then there are $x\in \Ob(\fX_U)$ and isomorphisms $f_i: x|U_i \to x_i$ such that $f_j|U_i\times_U U_j=f_{ji}\circ f_i|U_i\times_U U_j$.
\end{enumerate}

We will focus our attention on \emph{algebraic stacks}, which we now define. 

\begin{defn}
	A stack $\fX$ is algebraic if the diagonal morphism $\fX\to \fX\times \fX$ is representable and affine, and there is an affine scheme $U$ together with a faithfully flat morphism $P\colon U\to \fX$.
\end{defn}

Let $(A,\Psi)$ be a flat Hopf algebroid. Then $(X_0=\Spec(A), X_1=\Spec(\Psi))$ has the structure of a groupoid object in the category of affine schemes. Moreover, $(X_0=\Spec(A), X_1=\Spec(\Psi))$ determines a presheaf of groupoids on $\Aff$, however the corresponding category fibered in groupoids is not in general a stack, as Condition (2) from the definition above does not hold in general. Nonetheless, the inclusion functor from stacks to presheaves of groupoids has a right adjoint called stackification; we thus get a stack $\fX$ associated with $(A,\Psi)$ which is algebraic with a fixed presentation $\Spec(A)\to \fX$. This correspondence is in fact an equivalence, see~\cite[Thm.~8]{naumann_stack_2007}.

\begin{thm}[Naumann]
	There is an equivalence of $2$-categories
	\[
  		\{\text{Algebraic stacks with a fixed presentation $\Spec(A)\to \fX$}\} \simeq \{\text{flat Hopf algebroids $(A,\Psi)$}\}.
\]
\end{thm}

The analogy between Hopf algebroids and algebraic stacks does not end with the result above. There is a corresponding notion of quasi-coherent sheaves on algebraic stacks, see \cite{laumon_champs_2000, sitte2014local}. If we denote the category of quasi-coherent sheaves over $\fX$ by $\QCoh(\fX)$, there is an equivalence of abelian categories
\begin{equation}\label{comod=QCSh}
\xymatrix{\QCoh(\fX) \ar[r]^-{\sim} & \Comod_\Psi,}
\end{equation}
where $(A,\Psi)$ is the Hopf algebroid corresponding to $\fX$ \cite[Sec.~3.4]{naumann_stack_2007}. Of course, under this equivalence the structure sheaf $\cO_{\fX}$ corresponds to the $\Psi$-comodule $A$, and dualizable quasi-coherent sheaves correspond to dualizable $\Psi$-comodules. 

Finally, we recall that an algebraic stack is an Adams stack if it is the stack associated to an Adams Hopf algebroid. These have been characterized as algebraic stacks satisfying the strong resolution property, i.e., algebraic stacks whose category of quasi-coherent sheaves is generated by the dualizable ones \cite{schappi_characterization_2012,hovey_htptheory}. It turns out that in this case the equivalence of~\eqref{comod=QCSh} is symmetric monoidal.

\begin{prop}
If $\fX$ is an Adams stack, then there is a symmetric monoidal equivalence
\[
\xymatrix{\QCoh(\fX) \ar[r]^-{\sim} & \Comod_\Psi}
\]
of abelian categories, where $(A,\Psi)$ is the Hopf algebroid corresponding to $\fX$. 
\end{prop}
\begin{proof}
Sch{\"a}ppi proves~\cite[Thm.~1.6]{schappi_ind-abelian_2012} that an $A$-linear ind-abelian category $\cA$ is weakly Tannakian (as defined in~\cite[Def.~1.5]{schappi_ind-abelian_2012}) if and only if there exists an Adams stack $\fX$ over $A$ and a symmetric monoidal equivalence $\cA\simeq \QCoh^{\omega}(\fX)$, where $\QCoh^{\omega}(\fX) \subset \QCoh(\fX)$ is the subcategory of finitely presented quasi-coherent sheaves. Moreover, if $(A,\Psi)$ is an Adams Hopf algebroid, $\Comod_\Psi$ is an $A$-linear ind-abelian weakly Tannakian category, and we get a symmetric monoidal equivalence $\Comod_\Psi^{\omega}\simeq \QCoh^{\omega}(\fX)$, where $\fX$ is the corresponding Adams stack, see the proof of \cite[Thm.~1.6]{schappi_ind-abelian_2012}. Passing to the categories of ind-objects we thus get an equivalence $\Comod_\Psi\simeq \QCoh(\fX)$ which is still symmetric monoidal as tensor products of ind-objects are defined as filtered colimits of tensor products in the corresponding categories of finitely presented objects.
\end{proof}
 
\begin{rem}
It is worth pointing out that the Adams condition on stacks is not very restrictive, and is satisfied by many interesting stacks originating both in geometry and topology. For instance, in \cite{thomason_equivariant_1987} Thomason proves that several types of quotient stacks satisfy the resolution property, which is equivalent to being an Adams stack under his hypotheses. This work can be used to show that the stacks associated to a flat Hopf algebroid $(A,\Psi)$ where $A$ is a Dedekind ring also satisfy the Adams condition \cite{schappi_characterization_2012}.

Furthermore, several stacks coming from algebraic topology are Adams; they often arise as the stack corresponding to the Hopf algebroid of cooperations for a given ring spectrum. The stack of formal groups and other related substacks \cite{goerss_quasi-coherent_2008}, and the stacks associated to topologically flat cohomology theories \cite{hovey_htptheory} are all such examples. For instance, $MU$, $BP$, $K$ and $H\mathbb{F}_p$ are topologically flat.

\end{rem}

Let us fix an algebraic stack $\fX$. Building upon the symmetric monoidal equivalence in \eqref{comod=QCSh}, we can now translate our theory of local duality for comodules directly to the language of stacks. First, notice that since $\QCoh(\fX)$ is a Grothendieck category, we can consider its derived category $\cD_{\fX}=\cD(\QCoh(\fX))$. Recall that the structure sheaf $\cO_{\fX}$, is given by $\cO(\Spec(A)\to \fX)=\Psi$. In what follows let $\cG_{\fX}$ denote the set of dualizable objects in $\cD_{\fX}$. 

\begin{defn}
We define $\IndCoh{\fX}$, the stable $\infty$-category of quasi-coherent sheaves over $\fX$ as the category of ind-category of the thick subcategory of $\cD_{\fX}$ generated by a set of representative of dualizable objects $\cG_{\fX}$ of $\QCoh(\fX)$ considered as complexes concentrated in degree zero, that is, 
\[
\IndCoh_{\fX}=\Ind(\Thick_{\fX}(\cG_{\fX})).
\]
\end{defn}

Note that as with the stable category of comodules, $\IndCoh_{\fX}$ is a stable category, which is compactly generated by $\cG_{\fX}$. In fact, the definition of $\IndCoh_{\fX}$ implies:
\begin{prop}
	There is an equivalence of stable categories between $\IndCoh_{\fX}$ and $\Stable_\Psi$.
\end{prop}

\begin{proof}
	Since dualizable quasi-coherent shaves correspond to dualizable $\Psi$-comodules in the equivalence above, and both $\IndCoh_{\fX}$ and $\Stable_\Psi$ are generated by $\cG_{\fX}$ and $\cG_{\Psi}$ (the set of dualizable $\Psi$-comodules), respectively, it follows that this equivalence is in fact symmetric monoidal. Observe that this is true regardless of whether the equivalence is symmetric monoidal at the level of abelian categories.
\end{proof}

Consider a closed substack $\fZ$ defined by an ideal sheaf $\cI\subseteq \cO_{\fX}$ corresponding to an invariant ideal $I\subseteq A$, where $P\colon\Spec(A)\to \fX$ is a given presentation. We assume $\cO_{\fX}/\cI$ is a compact object of $\IndCoh_{\fX}$ and define the category $\IndCoh_{\fX}^{{\cI}-\mathrm{tors}}$ of ${\cI}$-torsion quasi-coherent sheaves as
\[
	\IndCoh_{\fX}^{{\cI}-\mathrm{tors}} = \Locid{\IndCoh_{\fX}}(\cO_{\fX}/\cI),
\]
the localizing ideal in $\IndCoh_{\fX}$ generated by the compact object $\cO_{\fX}/\cI$. Thus, the inclusion of $\IndCoh_{\fX}^{\cI-\mathrm{tors}}$ into $\IndCoh_{\fX}$ admits a right adjoint, 
\[
\xymatrix{\Gamma_{\fZ}\colon \IndCoh_{\fX} \ar@<0.5ex>[r] & \IndCoh_{\fX}^{\cI-\mathrm{tors}}.}
\] 

As usual, we get the familiar diagram of adjunctions between stable categories
  \begin{equation}
\xymatrix{& \IndCoh_{\fX}^{\cI-\textrm{loc}} \ar@<0.5ex>[d] \ar@{-->}@/^1.5pc/[ddr] \\
& \IndCoh_{\fX} \ar@<0.5ex>[u]^{L_{\fZ}} \ar@<0.5ex>[ld]^{\G_{\fZ}} \ar@<0.5ex>[rd]^{\Lambda^{\fZ}} \\
\IndCoh_{\fX}^{\cI-\textrm{tors}} \ar@<0.5ex>[ru] \ar[rr]_{\sim} \ar@{-->}@/^1.5pc/[ruu] & & \IndCoh_{\fX}^{{\cI}-\textrm{comp}} \ar@<0.5ex>[lu]}
\end{equation} 
and from the equivalence of $\IndCoh_{\fX}^{\cI-\textrm{tors}}$ and $\IndCoh_{\fX}^{\cI-\textrm{comp}}$ at the bottom, we obtain the following local duality result:

\begin{thm}
Let $\fX$ be an algebraic stack with a fixed presentation $P\colon\Spec(A)\to \fX$. If  $\cF,\cG \in \IndCoh_{\fX}$, there is an equivalence
 \[
 	\xymatrix{\iHom_{\fX}(\Gamma_{\fZ}\cF,\cG) \ar[r]^-{\sim} & \iHom_{\fX}(\cF,\Lambda^{\fZ}\cG),}
 \]
 natural in each variable.
\end{thm}

From \Cref{lem:AIcompact} we see that a sufficient condition for $\cO_{\fX}/\cI$ to be compact is that the ideal $I$ corresponding to $\cI$ is generated by a finite invariant regular sequence. For the next result, we recall the following definition. 
\begin{defn}
	An algebraic stack $\fX$ is called Noetherian if there exists a presentation $\Spec(A) \to \fX$ with $A$ Noetherian. 
\end{defn}

Then, \Cref{prop:identtorsfun} takes the form:

\begin{cor}
Suppose $\cI$ is a Noetherian algebraic stack, locally generated by a finite regular sequence, then we have a natural equivalence
\[
	\Gamma_{\fZ} \cF \simeq \colim_k \iExt_{\fX}(\cO_{\fX}/\cI^{k},\cF).
\]
for all $\cF \in \IndCoh_{\fX}$.
\end{cor}

Moreover, the pair of adjoints $(P^*,P_*)$ corresponds to the forgetful and cofree functors $(\epsilon_*,\epsilon^*)$ on the corresponding category of $\Psi$-comodules \cite{sitte2014local}. It follows from \cref{prop:comparison_Stable-Derivedcat} that the following diagrams of functors commute
\begin{equation}\label{eq:comparison_adjoints}
\xymatrix{
\IndCoh_{\fX}^{\cI-\text{tors}} \ar@<0.5ex>[r]^-{\iota_{\text{tors}}} \ar@<0.5ex>[d]^{P^*} & \ar@<0.5ex>[d]^{P^*}\ar@<0.5ex>[l]^-{\Gamma_{\fZ}} \IndCoh_{\fX} \ar@<0.5ex>[r]^-{L_{\fZ}} & \ar@<0.5ex>[l]^{\iota_{\text{loc}}}\IndCoh_{\fX}^{\cI-\text{loc}}\ar@<0.5ex>[d]^{P^*} \\
\D_{\Spec(A)}^{\cI-\text{tors}} \ar@<0.5ex>[r]^{\iota_{\text{tors}}} \ar@<0.5ex>^{P_*}[u]&\ar@<0.5ex>[u]^{P_*} \ar@<0.5ex>[l]^-{\Gamma_Z} \D_{\Spec(A)} \ar@<0.5ex>^-{L_Z}[r]& \ar@<0.5ex>[l]^{\iota_{\text{loc}}}\D_{\Spec(A)}^{\cI-\text{loc}}\ar@<0.5ex>[u]^{P_*},
}
\end{equation}
where $Z$ is the closed subscheme of $\Spec(A)$ defined by $I$.
\begin{rem}
	We would like to emphasize that local duality for stacks has been studied before in unpublished work of Goerss~\cite{goerss_quasi-coherent_2008}. He proves a local duality theorem for quasi-coherent sheaves over $\fX$, provided that $\fX$ is a quasi-compact and separated stack together with a scale, hypotheses directly inspired and satisfied by the stack of formal groups. His approach differs in two ways from ours: Firstly, Goerss works directly with quasi-coherent sheaves on the stack, rather than an appropriate derived category of such. Secondly, following \cite{local_cohom_schemes} he reduces local duality for stacks to the affine case via a local-to-global argument. Consequently, his theorem applies to examples not covered by our result, and vice versa. 
\end{rem}

\begin{rem}
	Another interesting work in this vein is Sitte's thorough study of local cohomology for quasi-coherent sheaves over an Adams stack \cite{sitte2014local}. While he does not investigate local duality, he gives a comparison of the right derived functors of torsion on the stack with the corresponding derived functors on the affine cover. More precisely, he shows that if $P\colon\Spec(A)\to\fX$ is a presentation for $\fX$, then there is an equivalence of cohomological $\delta$-functors
	\[
	\mathbf{R}_{\QCoh(\fX)} \Gamma_{\fZ}(-) \to  \mathbf{R}_{\Mod(\cO_{\Spec(A)})} \Gamma_{Z}(P^*-),
	\]
	provided $Z=\Spec(A/I)$, where $I$ is generated by a weakly proregular sequence \cite[Thm.~4.4.2]{sitte2014local}; this should be compared to the equivalence $P^*\Gamma_{\fZ}\simeq \Gamma_ZP^*$ coming from the first square in \eqref{eq:comparison_adjoints}.	
	\end{rem}

\section{Local duality in stable homotopy theory}\label{sec:duality-spectra}

In the following sections we describe how local duality manifests itself in stable homotopy theory. We start by describing local duality in the category of $p$-local and $E$-local spectra, where $E = E(n)$ is height $n$ Johnson--Wilson theory. We show that the functors we construct abstractly via our local duality framework agree with  familiar constructions used in chromatic homotopy theory, and that we can give explicit identifications for them, recovering (or indeed extending) results from~\cite{hovey_morava_1999}. We finish by comparing local duality for the category of spectra with local duality in the category of $BP$-module spectra (or $E$-module spectra), showing that they are compatible in a certain sense. 

We will assume that the reader is familiar with some basic concepts from stable homotopy theory, as can be found in~\cite{hovey_morava_1999} for example. 

\subsection{Local duality in stable homotopy theory --- the global case}\label{sec:globalspectra}

Let $\Sp$ be the category of $p$-local spectra, for $p$ a fixed prime which will mostly be omitted from the notation. This is a stable category compactly generated by its tensor unit $S^0$, whose compact objects are precisely the finite spectra. For $n\ge 0$, recall that a finite spectrum $F$ is said to be of type $n$ if $K(n-1)_*(F) = 0$ and $K(n)_*(F) \ne 0$; here, $K(n)$ denotes Morava $K$-theory of height $n$. By the thick subcategory theorem~\cite{orangebook}, the category of finite spectra admits an exhaustive filtration by thick subcategories $\mathcal{C}(n)$ consisting of the finite spectra of type $\ge n$.

Fix an integer $n\ge 0$ and choose a finite spectrum $F(n)$ of type $n$; by the thick subcategory theorem, $\mathcal{C}(n)$ is the thick subcategory generated by $F(n)$ and none of the following constructions will depend on the choice of $F(n)$. Applying \Cref{thm:hps} to the local duality context $(\Sp,F(n))$ yields the diagram on the left
\begin{equation}\label{eq:globalspectradiagram}
\xymatrix{& \Sp^{I_n-\mathrm{loc}} \ar@<0.5ex>[d] \ar@{-->}@/^1.5pc/[ddr] & & & \Sp_{E(n-1)}^f \ar@<0.5ex>[d] \ar@{-->}@/^1.5pc/[ddr] \\
& \Sp \ar@<0.5ex>[u]^{L_{I_n}} \ar@<0.5ex>[ld]^{\Gamma_{I_n}} \ar@<0.5ex>[rd]^{\Lambda^{I_n}} & & & \Sp \ar@<0.5ex>[u]^{L_{n-1}^f} \ar@<0.5ex>[ld]^{C_{n-1}^f} \ar@<0.5ex>[rd]^{L_{F(n)}} \\
\Sp^{I_n-\mathrm{tors}} \ar@<0.5ex>[ru] \ar[rr]_-{\sim} \ar@{-->}@/^1.5pc/[ruu] & & \Sp^{I_n-\mathrm{comp}} \ar@<0.5ex>[lu] & \mathcal{C}_{n-1}^f \ar@<0.5ex>[ru] \ar[rr]_-{\sim} \ar@{-->}@/^1.5pc/[ruu] & & \Sp_{F(n)}, \ar@<0.5ex>[lu]}
\end{equation}
together with the properties summarized in \Cref{thm:hps}; in particular, this introduces the notation we are going to use. As we will see shortly, the categories and functors so constructed are fundamental objects in stable homotopy theory, as displayed in the right diagram above.\footnote{The shift by 1 in the indexing is unfortunate, but consistent with the literature.}  

We start by recalling the existence of appropriate towers of generalized Moore spectra, which play the role of Koszul systems in stable homotopy theory. 

\begin{defn}
A finite spectrum $M=M(n) \in \Sp$ is called a generalized type $n$ Moore spectrum if $M$ is of type $n$ and its $BP$-homology is of the form $BP_*(M) \cong BP_*/(p^{i_0},\ldots,v_{n-1}^{i_{n-1}})$ with $i_j>0$ for all $j$. An inverse system of such spectra is called cofinal if the corresponding sequence of $n$-tuples of integers $(i_0,\ldots,i_{n-1})$ is cofinal in $\N^n$. 
\end{defn}

Using the periodicity theorem as proven by Devinatz, Hopkins, and Smith~\cite{nilpotence1, nilpotence2}, Hovey and Strickland~\cite[Prop.~4.22]{hovey_morava_1999} construct cofinal systems of generalized Moore spectra at all heights. 

\begin{thm}[Devinatz--Hopkins--Smith; Hovey--Strickland]\label{thm:towermoore}
For every $n\ge 0$, there exists a cofinal tower $(M_i(n))_{i \in \N}$ of generalized type $n$ Moore spectra. Moreover, these towers are compatible in the sense that there exists a cofiber sequence
\[
\xymatrix{(M_i(n))_{i \in \N} \ar[r] & (M_i(n))_{i \in \N} \ar[r] & (M_i(n+1))_{i \in \N}}
\]
of towers of spectra for all $n \ge 0$. 
\end{thm}

\begin{defn}
For  a tower $(M_i(n))_{i \in \N}$ of generalized Moore spectra of type $n$ as in \Cref{thm:towermoore}, we denote the colimit of the dual system $(DM_i(n))_{i \in \N}$ by ${C}^f_{n-1}S^0$. This can be extended to a smashing functor ${C}^f_{n-1}\colon \Sp \to \Sp$ by defining
\begin{equation}\label{eq:globaltorsionid}
{C}^f_{n-1}X = X \otimes {C}^f_{n-1}S^0 \simeq X \otimes \colim_{i} DM_i(n)
\end{equation}
for any $X \in \Sp$. 
\end{defn}
Note that there is a natural transformation from $C_{n-1}^f \to \text{id}$. 
\begin{rem}
Note that, in the literature, $C_{n-1}^f$ is sometimes defined as the acyclification functor with respect to $F(n)$, hence in this case would coincide with $\Gamma_{I_n}$ by definition. The following proposition justifies our choice of notation. 
\end{rem}

The next result identifies the local cohomology functor $\Gamma_{I_n}$ associated to the local duality context $(\Sp,F(n))$ with ${C}^f_{n-1}$; this in particular recovers~\cite[Prop.~7.10]{hovey_morava_1999}.

\begin{prop}\label{prop:torsionfunctorspectra}
With notation as above, there exists a natural equivalence of functors
\[
\xymatrix{{C}_{n-1}^f \ar[r]^-{\sim} & \Gamma_{I_n}.}
\]
In other words, the system $(M_i(n))_{i \in \N}$ of generalized type $n$ Moore spectra is a Koszul system for the local duality context $(\Sp,F(n))$. 
\end{prop}
\begin{proof}
We will prove this by induction on $n$. If $n=0$, we can choose $F(0) = S^0$ and $M_i(0) = S^0$ for all $i$, hence both $\Gamma_{I_0}$ and ${C}_{-1}^f$ are equivalent to the identity functor. 

Now assume the claim has been proven up to height $n-1$. Dualizing the inductive construction of the tower $(M_i(n))$ of generalized Moore spectra of type $n$ from a type $n-1$ tower $(M_i(n-1))$ gives a cofiber sequence of towers
\[\xymatrix{\ldots \ar[d]_{\mathrm{can}} & \ldots \ar[d]^{v_{n-1}^{r_{i-1}}\circ\mathrm{can}} & \ldots \ar[d] \\
DM_i(n-1) \ar[r]^{w_i} \ar[d]_{\mathrm{can}} & DM_i(n-1) \ar[r] \ar[d]^{v_{n-1}^{r_{i}}\circ\mathrm{can}}  & \Sigma DM_{i}(n) \ar[d] \\
DM_{i+1}(n-1) \ar[r]^{w_{i+1}} \ar[d]_{\mathrm{can}} & DM_{i+1}(n-1) \ar[r] \ar[d]^{v_{n-1}^{r_{i+1}}\circ\mathrm{can}}  & \Sigma DM_{i+1}(n) \ar[d] \\
\ldots & \ldots & \ldots}\]
as in~\cite[Sec.~4]{hovey_morava_1999}. Here, the maps labeled ${can}$ are the canonical ones, while $v_{n-1}^{r_i}$ denotes an appropriate power of a $v_{n-1}$-self map. Passing to the colimit over these towers thus yields a cofiber sequence
\begin{equation}\label{eq:chromresol}
\xymatrix{C^f_{n-2}S^0 \ar[r] & v_{n-1}^{-1}C^f_{n-2}S^0 \ar[r] & \Sigma C^f_{n-1}S^0,}
\end{equation}
where $v_{n-1}^{-1}C^f_{n-2}S^0$ denotes the colimit of the tower $(v_n^{-1}DM_i(n))_{i\in \N}$. By the induction hypothesis, there are equivalences
\[
C_{n-2}^fF(n) \simeq \Gamma_{I_{n-1}}F(n) \simeq F(n)
\]
as $F(n) \in \Loc(\mathcal{C}(n)) = \Sp^{I_{n-1}-\mathrm{tors}}$. Furthermore, it is easy to see that $F(n) \otimes v_{n-1}^{-1}C^f_{n-2}S^0 = 0$, so smashing the cofiber sequence \eqref{eq:chromresol} with $F(n)$ shows that
\[
C_{n-1}^fF(n) \simeq F(n).
\]
Therefore, the claim follows from \Cref{lem:torsionred} and \Cref{thm:towermoore}.
\end{proof}

The localization functor $L_{n-1}^f$ has been defined as finite localization away from a finite type $n$ complex in~\cite{miller_finiteloc}, see also~\cite{ravenel_liftafter} and~\cite{mahowaldsadofsky}. This implies the existence of a cofiber sequence
\[
\xymatrix{C_{n-1}^f \ar[r] & \Id \ar[r] & L_{n-1}^f}
\]
for any $X \in \Sp$. Using~\Cref{prop:torsionfunctorspectra}, a comparison of cofiber sequences immediately gives the next result. 

\begin{cor}
There exists a natural equivalence $L_{n-1}^f \xrightarrow{\sim} L_{I_n}$ of endofunctors on $\Sp$.
\end{cor}

\begin{rem}
The notation $v_{n-1}^{-1}C^f_{n-2}S^0$ for the colimit of the middle tower is justified, because \eqref{eq:chromresol} can be identified with the cofiber sequence obtained by evaluating $C_{n-1}^f \to \Id \to L_{n-1}^f$ on $C_{n-2}^fS^0$ and suspending,
\[\xymatrix{C_{n-2}^fS^0 \ar[r] & L_{n-1}^fC_{n-2}^fS^0 \ar[r] & \Sigma C_{n-1}^fS^0;}\]
in particular, $BP_*(v_{n-1}^{-1}C^f_{n-2}S^0) \cong v_{n-1}^{-1}BP_*/I_{n-1}^{\infty}$. Therefore, these sequences are a topological incarnation of the short exact sequences in $BP_*$-modules which splice together to form the chromatic resolution. 
\end{rem}

Our abstract framework then allows us to immediately recover~\cite[Prop.~7.10(d)]{hovey_morava_1999} as well as the last part of~\cite[Cor.~B.8]{hovey_morava_1999}.

\begin{thm}\label{thm:globalspectralocduality}
With notation as above, there is a natural equivalence
\[\Lambda^{I_n}(-) \simeq L_{F(n)}(-) \simeq \lim_i M_i(n) \otimes (-),\]
where $L_{F(n)}$ denotes Bousfield localization with respect to $F(n)$. Furthermore, this functor induces a symmetric monoidal equivalence 
\[\xymatrix{\mathcal{C}_{n-1}^f \ar[r]^-{\sim} & \Sp_{F(n)}.}\]
In particular, $\Sp_{F(n)}$ is compactly generated by $F(n)$. Moreover, local duality takes the form of an adjunction
\[
\iHom(C_{n-1}^fX,Y) \simeq \iHom(X,L_{F(n)}Y),
\]
where $\iHom$ is the internal function object of $\Sp$. 
\end{thm}
\begin{proof}
This follows from \Cref{prop:torsionfunctorspectra}, \Cref{prop:lochombl}, \Cref{thm:hps}, and \Cref{cor:completion}.
\end{proof}

From now on, we will use the more standard notation displayed in the right diagram of \eqref{eq:globalspectradiagram} for the categories and functors obtained from the local duality context $(\Sp, F(n))$.

\begin{rem}
The finite localization functor $L_{n}^f$ is closely related to the Bousfield localization functor $L_{n}$ associated to Johnson--Wilson theory $E(n)$ at height $n$. There is a natural transformation
\[\xymatrix{L_{n}^f \ar[r] & L_{n}}\]
and Ravenel's telescope conjecture~\cite{ravenel_localization} asserts that this is an equivalence or, equivalently, that the localizing subcategory $\ker(L_{n})$ of $\Sp$ is generated by finite spectra. Since the telescope conjecture is only known at heights $\le 1$ and believed to be false~\cite{ravenel_progress}, we here work with the finite localization functors instead.
\end{rem}

\subsection{Local duality in stable homotopy theory --- the local case}\label{sec:localdualitylocalsth}
We now turn to a local version of this picture. To this end, fix an integer $n\ge 0$ and let $E=E(n)$ be Johnson--Wilson theory at height $n$. This spectrum has coefficients $E_*=E(n)_* = \Z_{(p)}[v_1,\ldots,v_n^{\pm 1}]$, where the degree of the generator $v_i$ is $2(p^i-1)$; as usual, we have $v_0=p$. Denote by $\Sp_{E(n)}$ the category of $E$-local spectra or, in other words, the essential image of $L_n=L_E$ in $\Sp$. By the smash product theorem of Hopkins and Ravenel~\cite{orangebook}, the corresponding Bousfield localization functor $L_n$ is smashing, hence $\Sp_{E(n)}$ is a stable subcategory of $\Sp$ with compact generator $L_nS^0$, which is also the tensor unit in this subcategory.

If $F(m)$ is a finite type $m$ spectrum for some $m\le n$, then $L_nF(m) \in \Sp_{E(n)}$ is compact and nonzero, so it generates a localizing subcategory $\mathcal{M}_m = \Sp_{E(n)}^{I_m-\mathrm{tors}} = \Locid{}(L_nF(m))=\Loc(L_nF(m))$, called the $m$-th monochromatic category of $\Sp_{E(n)}$. Note that $\mathcal{M}_m$ depends implicitly on the chosen ambient height $n$. Applying \Cref{thm:hps} to the local duality context $(\Sp_{E(n)},L_nF(m))$ produces a local analogue of \eqref{eq:globalspectradiagram}:
\color{black}
\begin{equation}\label{eq:localspectradiagram}
\xymatrix{& \Sp_{E(n)}^{I_m-\mathrm{loc}} \ar@<0.5ex>[d] \ar@{-->}@/^1.5pc/[ddr] & & & \Sp_{E(m-1)} \ar@<0.5ex>[d] \ar@{-->}@/^1.5pc/[ddr] \\
& \Sp_{E(n)} \ar@<0.5ex>[u]^{L_{I_m}} \ar@<0.5ex>[ld]^{\Gamma_{I_m}} \ar@<0.5ex>[rd]^<<<<<{\Lambda^{I_m}} & & & \Sp_{E(n)} \ar@<0.5ex>[u]^{L_{m-1}} \ar@<0.5ex>[ld]^{M_m} \ar@<0.5ex>[rd]^<<<<<{L_{E(m,n)}} \\
\Sp_{E(n)}^{I_m-\mathrm{tors}} \ar@<0.5ex>[ru] \ar[rr]_-{\sim} \ar@{-->}@/^1.5pc/[ruu] & & \Sp_{E(n)}^{I_m-\mathrm{comp}} \ar@<0.5ex>[lu] &  \mathcal{M}_m \ar@<0.5ex>[ru] \ar[rr]_-{\sim} \ar@{-->}@/^1.5pc/[ruu] & & \Sp_{E(m,n)}, \ar@<0.5ex>[lu]}
\end{equation}
where $E(m,n)$ denotes the spectrum $F(m) \otimes E(n)$. The goal of this subsection is to identify the categories and functors of the diagram on the left with the familiar ones on the right, using a compatibility result between the local and the global case. Note that the functors $\Gamma_{I_m}$, $L_{I_m}$, and $\Lambda^{I_m}$ implicitly depend on the fixed ambient height $n$, but this dependence will be omitted from the notation. 
\begin{lem}\label{lem:monochromaticid}
There is a commutative diagram of adjunctions
\[\xymatrix{\mathcal{C}_{m-1}^f \ar[d]_{L_n} \ar@<-0.5ex>[r] & \Sp \ar@<-0.5ex>[l]_-{C_{m-1}^f} \ar[d]^{L_n}\\
\mathcal{M}_{m} \ar@<-0.5ex>[r]  & \Sp_{E(n)} \ar@<-0.5ex>[l]_-{M_m}.}\]
In particular, there is a natural equivalence
\[M_m(L_nX) \simeq \colim_i DM_i(m) \otimes L_nX\]
for any spectrum $X$, where $(M_i(m))_{i \in \N}$ is a tower of generalized Moore spectra of type $m$. 
\end{lem}
\begin{proof}
First observe that the composite $\mathcal{C}_{m-1}^f \to \Sp \to \Sp_{E(n)}$ takes values in $\mathcal{M}_m$, as it sends the generator $F(m)$ to the generator $L_nF(m)$ and commutes with colimits. This shows that the square involving the inclusion functors commutes. Secondly, we claim that $L_nC_{m-1}^f\iota_n$ is a right adjoint to the inclusion $\mathcal{M}_m \to \Sp_{E(n)}$, where $\iota_n$ denotes the right adjoint to $L_n\colon \Sp_{E(n)} \to \Sp$ with $L_n\iota_n = \text{id}$. To see this, it is enough to test against the generator $L_nF(m)$ of $\mathcal{M}_m$; indeed, for any $X \in \Sp_{E(n)}$, we get
\begin{align*}
\Hom(L_nF(m), L_nC_{m-1}^f\iota_nX) & \simeq \Hom(F(m), C_{m-1}^fL_n\iota_nX) \\
& \simeq \Hom(F(m),L_n\iota_nX) \\
& \simeq \Hom(L_nF(m), L_n\iota_nX) \\
& \simeq \Hom(L_nF(m), X).
\end{align*}
Here, the first equivalence uses the fact that $L_n$ is smashing, so that it commutes with $C_{m-1}^f$ by \eqref{eq:globaltorsionid}. Therefore, we obtain $M_m \simeq L_nC_{m-1}^f\iota_n$. Precomposing with $L_n$ then gives a natural equivalence
\[M_mL_n \simeq L_nC_{m-1}^f\iota_nL_n \simeq C_{m-1}^fL_n\iota_nL_n \simeq C_{m-1}^fL_n \simeq L_nC_{m-1}^f,\]
thereby verifying that the above square commutes. By \Cref{prop:torsionfunctorspectra}, there results a natural equivalence
\begin{align*}
M_mL_nX & \simeq L_nC_{m-1}^fX \\
& \simeq L_n\colim_i DM_i(m) \otimes X \\
& \simeq \colim_i DM_i(m) \otimes L_nX,
\end{align*}
since $L_n$ is a smashing localization. 
\end{proof}

\begin{lem}\label{lem:localizationid}
The left adjoint to the inclusion $\Sp_{E(m-1)} \to \Sp_{E(n)}$ is given by $L_{m-1}$, and there is a commutative diagram
\[\xymatrix{\Sp \ar@<0.5ex>[r]^-{L_{m-1}^f} \ar[d]_{L_n} & \Sp_{E(m-1)}^f \ar[d]^{L_n} \ar@<0.5ex>[l]  \\
\Sp_{E(n)} \ar@<0.5ex>[r]^-{L_{m-1}} & \Sp_{E(m-1)}. \ar@<0.5ex>[l]}\]
\end{lem}
\begin{proof}
Applying $L_n$ to the cofiber sequence of functors
\[\xymatrix{C_{m-1}^f \ar[r] &  \mathrm{id} \ar[r] & L_{m-1}^f}\]
and using \Cref{lem:monochromaticid} gives the cofiber sequence
\[
\xymatrix{M_mL_n \ar[r] & L_n \ar[r] & L_nL_{m-1}^f}
\]
defined on $\Sp_{E(n)}$. By~\cite[Cor.~6.10]{hovey_morava_1999}, there is a natural equivalence $L_nL_{m-1}^f \simeq L_{m-1}$, from which the claim follows easily. 
\end{proof}

\begin{lem}\label{lem:completionid}
The functor $\Lambda^{I_m}\colon \Sp_{E(n)} \to \Sp_{E(m,n)}$ constructed from the local duality context $(\Sp_{E(n)},L_nF(m))$ is Bousfield localization at the cohomology theory $E(m,n) = F(m) \otimes E_n$. Furthermore, the category $\Sp_{E(n)}^{I_m-\mathrm{comp}}$ is equivalent to the category $\Sp_{E(m,n)}$ of $E(m,n)$-local spectra.
\end{lem}
\begin{proof}
For any $X \in \Sp_{E(n)}$, there are equivalences
\begin{align*}
L_{E(m,n)}X & \simeq  \lim_i M_i(m) \otimes L_nX  & & \text{by \Cref{cor:completion} and \Cref{lem:monochromaticid}} \\ 
& \simeq L_{F(m)}L_nX & & \text{by \Cref{thm:globalspectralocduality}} \\
& \simeq L_{F(m) \otimes E(n)}X, & & 
\end{align*}
where the last equivalence follows from the obvious generalization of~\cite[Cor.~2.2]{hovey_csc}. Since $\Sp_{E(n)}^{I_m-\mathrm{comp}}$ and $\Sp_{E(m,n)}$ are the essential images of $\Lambda^{I_m}$ and $L_{E(m,n)}$, respectively, the second part of the claim follows.
\end{proof}

As an immediate application, we  obtain chromatic fracture squares as a special case of \Cref{cor:fracturesquare}. 

\begin{cor}
For any $M \in \Sp_{E(n)}$ and $m \le n$, there is a pullback square
\[
\xymatrix{M \ar[r] \ar[d] & L_{E(m,n)}M \ar[d] \\
L_{m-1}M \ar[r] & L_{m-1}L_{E(m,n)}M.}\
\]
\end{cor}

\begin{rem}
Note that the naive analogues of \Cref{lem:monochromaticid} and \Cref{lem:localizationid} for the completion functor do not hold in general. Indeed, for $m=n$, we get a diagram
\begin{equation}\label{eq:localtoglobalcompletion}
\xymatrix{\Sp \ar@<0.5ex>[r]^-{L_{F(n)}} \ar[d]_{L_n} & \Sp_{F(n)} \ar[d]^{L_n} \ar@<0.5ex>[l]  \\
\Sp_{E(n)} \ar@<0.5ex>[r]^-{L_{K(n)}} & \Sp_{K(n)}. \ar@<0.5ex>[l]}
\end{equation}
By~\Cref{lem:completionid}, there is a natural equivalence $L_{K(n)} \simeq L_{F(n)}L_n$, but this is not equivalent to $L_nL_{F(n)}$ if $n>0$, as can be verified on the sphere spectrum. Therefore, the square \eqref{eq:localtoglobalcompletion} does not commute. 
\end{rem}

\begin{rem}
Since the localization functor $L_{E(m,n)}$ depends only on the Bousfield class of $E(m,n)$, it can also be identified with localization at the cohomology theory $E(n)/I_m$, which is the $E(n)$-algebra with coefficients $E(n)_*/(p,\ldots,v_{m-1})$. In particular, if $m=n$, then $L_{E(n,n)} \simeq L_{K(n)}$. Moreover, in this case the cofiber sequence constructed in the proof of \Cref{lem:localizationid}
\[\xymatrix{M_nL_n \ar[r] & L_n \ar[r] & L_{n-1}}\]
is the usual sequence defining the monochromatic layer $M_nL_n$.
\end{rem}

Local duality for the category $\Sp_{E(n)}$ with respect to a finite type $m$ complex then takes the following form.

\begin{prop}
For $X,Y \in \Sp_{E(n)}$, there are natural equivalences
\[
L_{E(m,n)}X \simeq \lim_i M_i(m) \otimes X 
\]
and
\[
\iHom(M_mX,Y) \simeq \iHom(X,L_{E(m,n)}Y),
\]
where $\iHom$ denotes the internal function object of $\Sp_{E(n)}$, which agrees with the one in $\Sp$. Furthermore, there is an equivalence of stable categories
\[
\xymatrix{\mathcal{M}_m \ar[r]^-{\sim} & \Sp_{E(m,n)}.}
\]
\end{prop}
\begin{proof}
This follows from \Cref{lem:monochromaticid}, \Cref{lem:completionid}, and the results in \Cref{sec:abstract}.
\end{proof}

This is a generalization of~\cite[Prop.~7.10]{hovey_morava_1999}.

\subsection{From $E$-local spectra to $E$-module spectra}\label{sec:Elocal}

The arguments of the previous section can also be applied to the adjunctions
\[\xymatrix{\Sp \ar@<0.5ex>[r]^-{- \otimes BP} & \Mod_{BP} \ar@<0.5ex>[l] \ar@<0.5ex>[r]^-{- \otimes_{BP} E} & \Mod_{E} \ar@<0.5ex>[l] }\]
to obtain similar comparison results between the corresponding torsion, localization, and completion functors. This compatibility has been studied before by Greenlees and May~\cite[\S 6]{gm_localhomology}, see also~\cite[\S 3]{bs_centralizers}.

To this end, let $BP$ be the Brown--Peterson spectrum with coefficients $BP_*=\Z_{(p)}[v_1,v_2,\ldots]$. Landweber showed that the invariant prime ideals of $BP_*$ are given by $I_n = (p,\ldots,v_{n-1})$ for all $n\ge 0$. The free $BP$-module functor $-\otimes BP$ gives a left adjoint to the forgetful functor $\Mod_{BP} \to \Sp$, and this adjunction will be denoted
\[\xymatrix{\Phi_*\colon \Sp \ar@<0.5ex>[r]^-{- \otimes BP} & \Mod_{BP} \ar@<0.5ex>[l]\colon \Phi^*}\]
in this section.

\begin{rem}
By a result of Basterra and Mandell~\cite{basterramandell}, $BP$ admits the structure of an $\mathbb{E}_4$-ring spectrum, but Lawson~\cite{lawson_bp} proved recently that $BP$ does not admit the structure of an $\mathbb{E}_{12}$-ring spectrum at the prime $2$. Consequently, the stable $\infty$-category $\Mod_{BP}$ of $BP$-module spectra is only known to be $\mathbb{E}_3$-monoidal, hence is not a stable category as defined in \Cref{sec:abstract}. However, for the arguments in this section, an $\mathbb{E}_1$-monoidal structure suffices; alternatively, the reader can replace $BP$ by the $p$-localization of the complex cobordism spectrum $MU$, which does have a natural $\mathbb{E}_{\infty}$-ring structure.
\end{rem}

Suppose $F(n)=M(n)$ is a generalized Moore spectrum of type $n$, generating the localizing subcategory $\mathcal{C}_{n-1}^f$ of $\Sp$. Since there exists a sequence of positive natural numbers $(i_0,i_1,\ldots,i_{n-1})$ such that
\[
BP_*(F(n)) \cong BP_*/(p^{i_0},\ldots,v_{n-1}^{i_{n-1}}),
\]
the $BP$-module $BP \otimes F(n)$ generates the category $\Mod_{BP}^{I_n-\text{tors}}$ of $I_n$-torsion $BP$-module spectra, cf.~\cite[Lem.~34]{bakerrichter_invertible} and the proof of \Cref{prop:wt}. Moreover, as a left adjoint, $\Phi_*$ preserves all colimits, so it restricts to a functor $\mathcal{C}_{n-1}^f \to \Mod_{BP}^{I_n-\text{tors}}$ which we also call $\Phi_*$. The next result is then proven using arguments analogous to the ones for \Cref{lem:monochromaticid}, \Cref{lem:localizationid}, and \Cref{lem:completionid}.

\begin{thm}
With notation as in~\Cref{sec:modulespectra} and~\Cref{sec:globalspectra}, we have the following comparison between local duality contexts.
\begin{enumerate} 
 \item There is a commutative diagram 
\[
\xymatrix{\mathcal{C}_{n-1}^f \ar@<0.5ex>[r]^-{} \ar[d]_{\Phi_*} & \Sp_{} \ar[d]_{\Phi_*} \ar@<0.5ex>[l]^-{C_{n-1}^f} \ar@<0.5ex>[r]^-{L_{n-1}^f} & \Sp_{E(n-1)}^f \ar@<0.5ex>[l]  \ar[d]^{\Phi_*} \\
\Mod_{BP}^{I_n-\mathrm{tors}} \ar@<0.5ex>[r]^-{} & \Mod_{BP} \ar@<0.5ex>[l]^-{\Gamma_{I_n}} \ar@<0.5ex>[r]^-{L_{I_n}} & \Mod_{BP}^{I_n-\mathrm{loc}} \ar@<0.5ex>[l]}
\]
with left adjoints displayed on top. In particular, there is an equivalence $L_{n-1}M \simeq L_{n-1}^fM \simeq L_{I_{n}}M$ for any $M \in \Mod_{BP}$.
 \item When restricted to the category of $BP$-module spectra, there is a natural equivalence
 \[L_{BP \otimes F(n)} \simeq \Lambda^{I_n}.\]
\end{enumerate}
\end{thm}
\begin{proof}
As mentioned above, the proof is mostly formal. Because every module spectrum is a retract of a free one, the equivalences of localization and completion functors on $\Mod_{BP}$ follow easily. Finally, the fact that $L_{n-1}^fM \simeq L_{n-1}M$ for $BP$-module spectra is proven in~\cite[Cor.~1.10]{hovey_csc}.
\end{proof}

Note that this theorem recovers Theorem 6.1 and Proposition 6.6 of~\cite{gm_localhomology}. Similar results hold for the adjunction
\[\xymatrix{\Sp_{E(n)} \ar@<0.5ex>[r]^-{- \otimes E(n)} & \Mod_{E(n)}, \ar@<0.5ex>[l]}\]
so we also obtain the results of~\cite[\S 3]{bs_centralizers} as a special case. This local comparison is summarized in the following dictionary:
\renewcommand{\arraystretch}{1.4}
\begin{center}
    \begin{tabular}{| l | l |}
    \hline 
    $\Sp_{E(n)}$ & $\Mod_{E(n)}$ \\ \hline
    $M_m$ & $\Gamma_{I_m}$ \\ 
    $L_{m-1}$ & $L_{I_m}$ \\
    $L_{E(m,n)}$ & $\Lambda^{I_{m}}$ \\ \hline
    \end{tabular}
\end{center}
Here, the torsion, localization, and completion functors are constructed from the local duality contexts $(\Sp_{E(n)},L_nF(m))$ and $(\Mod_{E(n)},E(n) \otimes F(n))$, respectively.

\section{Local duality in chromatic homotopy theory}\label{sec:duality-chromatic}

We now move on to the study of local duality in the category of comodules for $(BP_*,BP_*BP)$ and $(E_*,E_*E)$, where $E_*$ is any Landweber exact $BP_*$-algebra of height $n$, such as Morava $E$-theory or Johnson--Wilson theory. We begin by constructing our local duality functors for the duality context $(\Stable_{BP_*BP},BP_*/I_{n+1})$, and comparing this to work of Hovey and Hovey--Strickland~\cite{hovey_chromatic,hs_localcohom,hs_leht}. There is an adjoint pair 
\[
\xymatrix{
\Phi_*\colon\Stable_{BP_*BP} \ar@<0.5ex>[r] & \ar@<0.5ex>[l] \Stable_{E_*E}\colon\Phi^\ast
}
\]
and we identify the localization functor $L_{I_{n+1}}$ arising from the local duality context above with the composite $\Phi^\ast \Phi_\ast$. Using this we prove that, for $k \le n+1$, the $I_k$-local objects in $\Stable_{BP_*BP}$  are the same as the $I_k$-local objects in $\Stable_{E_*E}$. Using this identification of $L_{I_{n+1}}$ we also prove that the spectral sequence constructed by Hovey and Strickland converging to $BP_*(L_nX)$ is a special case of the $E$-based Adams spectral sequence. Moreover, we obtain generalizations of this spectral sequence for any Landweber exact $BP_*$-algebra.  

\subsection{Local duality for $BP_*BP$-comodules}

Given a suitable ring spectrum $R$, the associated homology theory carries more structure than just that of an $R_*$-module; it is additionally a comodule over the Hopf algebroid $(R_*,R_*R)$, i.e., there is a lift of the form
\[
\xymatrix{& \Comod_{R_*R} \ar[d]^{\text{forget}} \\
\Sp \ar[r]_-{R_*} \ar@{-->}[ru]^{R_*} & \Mod_{R_*}.}
\]

Given a spectrum $X$, we can think of $R_*X$ as a complex concentrated in degree 0, and hence as an object in the derived category of $R_*R$-comodules, or alternatively in $\Stable_{R_*R}$. In particular, if $R = BP$, then the functor $BP_*$ gives a comparison between the duality contexts $(\Sp,F(n+1))$ and $(\Stable_{BP_*BP},BP_*F(n+1))$, where $F(n+1)$ is a type $(n+1)$-spectrum. By the thick subcategory theorem, we can assume that $F(n+1)=M(n+1)$ is a generalized type $(n+1)$ Moore spectrum, and we will do so throughout this section. Algebraically, it is perhaps more natural to consider the local duality context $(\Stable_{BP_*BP},BP_*/I_{n+1})$, where $I_{n+1} = (p,\ldots,v_n)$ is an invariant ideal, but it turns out that they create the same local cohomology and homology functors. This is a consequence of the following result. 

\begin{prop}\label{prop:wt}
Suppose $F(n+1)$ is a generalized Moore spectrum of type $(n+1)$, then 
\[
\Thick_{BP_*BP}(BP_*/I_{n+1}) \simeq \Thick_{BP_*BP}(BP_*F(n+1)).
\]
In particular, the local duality contexts  $(\Stable_{BP_*BP},BP_*/I_{n+1})$ and $(\Stable_{BP_*BP},BP_*F(n+1))$ are equivalent.
\end{prop}
\begin{proof}
Viewing both categories as subcategories of $\Stable_{BP_*BP}$, we will show that they are mutually contained in each other. Since $BP_*F(n+1)$ is $I_{n+1}$-torsion, we get $BP_*F(n+1) \in \Thick(BP_*/I_{n+1})$, so one inclusion is clear. To prove the converse,  let $(i_0,\ldots,i_{n})$ be the sequence of positive integers such that 
\[
BP_*F(n+1) \cong BP_*/(p^{i_0},\ldots,v_{n-1}^{i_{n-1}},v_{n}^{i_{n}}).
\]
Consider the following commutative diagram of cofiber sequences:
\[
\xymatrix{BP_{*}/(p^{i_0},\ldots,v_{n-1}^{i_{n-1}}) \ar[r]^-{v_{n}} \ar[d]^{v_{n}^{i_{n}}} & BP_{*}/(p^{i_0},\ldots,v_{n-1}^{i_{n-1}}) \ar[r] \ar[d]^-{v_{n}^{i_{n}}} & BP_{*}/(p^{i_0},\ldots,v_{n-1}^{i_{n-1}},v_{n}) \ar[d]^-{v_{n}^{i_{n}}\simeq 0}  \\
BP_{*}/(p^{i_0},\ldots,v_{n-1}^{i_{n-1}}) \ar[r]^-{v_{n}} \ar[d] & BP_{*}/(p^{i_0},\ldots,v_{n-1}^{i_{n-1}}) \ar[r] \ar[d] & BP_{*}/(p^{i_0},\ldots,v_{n-1}^{i_{n-1}},v_{n}) \ar[d] \\
BP_{*}/(p^{i_0},\ldots,v_{n-1}^{i_{n-1}},v_{n}^{i_{n}}) \ar[r]_{v_{n}} & BP_{*}/(p^{i_0},\ldots,v_{n-1}^{i_{n-1}},v_{n}^{i_{n}}) \ar[r] & C(v_n).}
\]
The bottom cofiber sequence shows that $C(v_n) \in \Thick(BP_*F(n+1))$, while the right vertical cofiber sequence gives 
\[
C(v_n) \simeq BP_{*}/(p^{i_0},\ldots,v_{n-1}^{i_{n-1}},v_{n}) \oplus \Sigma BP_{*}/(p^{i_0},\ldots,v_{n-1}^{i_{n-1}},v_{n}),\] 
hence $BP_{*}/(p^{i_0},\ldots,v_{n-1}^{i_{n-1}},v_{n}) \in \Thick(BP_*F(n+1))$. Similarly, we can prove that the cofiber $C(v_{n-1})$ of the multiplication by $v_{n-1}$ endomorphism of $BP_{*}/(p^{i_0},\ldots,v_{n-1}^{i_{n-1}})$ is 
\[
C(v_{n-1}) \simeq BP_{*}/(p^{i_0},\ldots,v_{n-2}^{i_{n-2}},v_{n-1}) \oplus \Sigma BP_{*}/(p^{i_0},\ldots,v_{n-2}^{i_{n-2}},v_{n-1}).
\]
Therefore, the cofiber of multiplication by $v_{n-1}$ on $BP_{*}/(p^{i_0},\ldots,v_{n-1}^{i_{n-1}},v_{n})$ is 
\[
BP_{*}/(p^{i_0},\ldots,v_{n-2}^{i_{n-2}},v_{n-1},v_n) \oplus \Sigma BP_{*}/(p^{i_0},\ldots,v_{n-2}^{i_{n-2}},v_{n-1},v_n),
\]
thus $BP_{*}/(p^{i_0},\ldots,v_{n-2}^{i_{n-2}},v_{n-1},v_n) \in \Thick(BP_*F(n+1))$. Continuing this process with $v_{n-2}$ and so on until $v_0=p$, we conclude $BP_*/I_n \in \Thick(BP_*F(n+1))$, hence the claim. 
\end{proof}

In either case, we can apply local duality to obtain the following diagram of stable categories and functors
  \begin{equation}\label{eq:derivedcomod2}
\xymatrix{& \Stable_{BP_*BP}^{I_{n+1}-\textrm{loc}} \ar@<0.5ex>[d] \ar@{-->}@/^1.5pc/[ddr] \\
& \Stable_{BP_*BP} \ar@<0.5ex>[u]^{L_{I_{n+1}}} \ar@<0.5ex>[ld]^{\G_{I_{n+1}}} \ar@<0.5ex>[rd]^{\Lambda^{I_{n+1}}} \\
\Stable_{BP_*BP}^{I_{n+1}-\textrm{tors}} \ar@<0.5ex>[ru] \ar[rr]_{\sim} \ar@{-->}@/^1.5pc/[ruu] & & \Stable_{BP_*BP}^{I_{n+1}-\textrm{comp}}, \ar@<0.5ex>[lu]}
\end{equation}
where the functors have the following properties, via~\Cref{thm:localdualitystablepsi}. 
\begin{thm}\label{thm:localdualitybp}
Let $(\Stable_{BP_*BP},BP_*/I_{n+1})$ be the local duality context above.  
\begin{enumerate}
 \item The functors $\Gamma_{I_{n+1}}$ and $L_{I_{n+1}}$ are smashing localization and colocalization functors and $\Lambda^{I_{n+1}}$ is the localization functor with category of acyclics given by $\Stable_{BP_*BP}^{{I_{n+1}}-\mathrm{loc}}$. Moreover, $\Gamma_{I_{n+1}}$ and $\Lambda^{I_{n+1}}$ satisfy $\Lambda^{I_{n+1}} \circ \Gamma_{I_{n+1}} \simeq \Lambda^{I_{n+1}}$ and $\Gamma_{I_{n+1}} \circ \Lambda^{I_{n+1}} \simeq \Gamma_{I_{n+1}}$ and they induce mutually inverse equivalences
 \[\xymatrix{\Stable_{BP_*BP}^{{I_{n+1}}-\mathrm{comp}} \ar@<1ex>[r]^{\Gamma_{I_{n+1}}}_{\sim} & \Stable_{BP_*BP}^{{I_{n+1}}-\mathrm{tors}}. \ar@<1ex>[l]^{\Lambda^{I_{n+1}}}}\]
 \item For $M,N \in \Stable_{BP_*BP}$, there is an equivalence
 \[\xymatrix{\iHom_{BP_*BP}(\Gamma_{I_{n+1}}M,N) \ar[r]^{\sim} & \iHom_{BP_*BP}(M,\Lambda^{I_{n+1}}N),}\]
 natural in each variable.
 \end{enumerate}
\end{thm}

\begin{rem}
	It is useful to remind the reader at this point that $\Stable_{BP_*BP}$ is monogenic, as $(BP_*,BP_*BP)$ is a Landweber Hopf algebroid, see ~\Cref{rem:monogenic}. 
\end{rem}

As noted previously the cohomology theory $BP_*$ can be thought of as a functor to $\Stable_{BP_*BP}$, via the lift
\[
\xymatrix{& \Comod_{BP_*BP} \ar[d]^{\text{forget}} \ar[r] & \Stable_{BP_*BP} \\
\Sp \ar[r]_-{BP_*} \ar@{-->}[ru]^{BP_*} & \Mod_{BP_*}.}
\]
 This functor is compatible with the local duality context $(\Sp,F(n))$ in the following sense.
\begin{prop}
	The functor $BP_*:\Sp \to \Stable_{BP_*BP}$ restricts to a functor \[BP_*:\mathcal{C}_{n}^f \to \Stable_{BP_*BP}^{I_{n+1}-\textnormal{tors}} .\] 
\end{prop}

\begin{proof}
As we have seen in~\Cref{prop:wt}, we can assume that $BP_*F(n+1)$ is the generator of $\Stable_{BP_*BP}^{I_{n+1}-\text{tors}}$. 

	 Suppose now that $X \xr{f} Y \xr{g} Z$ is a cofiber sequence of spectra such that $BP_*X$ and $BP_*Z$ are $I_{n+1}$-torsion. We claim that $BP_*Y$ is also $I_{n+1}$ torsion. To see this, let $y \in BP_*Y$. There then exists $t_z \in \N$ such that $I_{n+1}^{t_z}g(y) = g(I_{n+1}^{t_z}y)=0$, so that $I_{n+1}^{t_z}y \in \ker g$. 
	Therefore, we can choose $x \in BP_*X$ so that $f(x) = I_{n+1}^{t_z}y$. By assumption, there exists $t_x \in \N$ such that $I_{n+1}^{t_x}x = 0$, hence  $0 = f(I_{n+1}^{t_x}x) = I_{n+1}^{t_x}f(x) = I_{n+1}^{t_x}I_{n+1}^{t_z}y = I_{n+1}^{t_x+t_z}y$, so that $y$ is $I_{n+1}$-torsion.
	  It follows that $BP_*$ restricts to a functor
	  \[
	  \xymatrix{\Thick(F(n+1)) \ar[r] & \Stable_{BP_*BP}^{I_{n+1}-\text{tors}}.}
	  \]
	  Since $BP_*$ preserves filtered colimits, we see that it restricts to a functor from $\Loc(F(n+1)) = \mathcal{C}_{n}^f$ to $\Stable_{BP_*BP}^{I_{n+1}-\text{tors}}$, as desired.
\end{proof}

\subsection{Basic properties of localization on $\Stable_{BP_*BP}$}
In this section we wish to compare the functors $L_{I_{n+1}}$ and $\Gamma_{I_{n+1}}$ with work of Hovey and Strickland~\cite{hs_localcohom,hs_leht} on localization and torsion functors for $\Comod_{BP_*BP}$. Recall that a hereditary torsion theory $\mathcal{T}$ in a cocomplete abelian category $\mathcal{A}$ is a Serre subcategory closed under arbitrary direct sums, i.e., the abelian analogue of a localization subcategory of a stable $\infty$-category. Let
\[
\xymatrix{\Phi\colon (BP_*,BP_*BP) \ar[r] & (E(n)_*,E(n)_*E(n))}
\]
denote the evident map of Hopf algebroids, and let $\Phi_*\colon\Comod_{BP_*BP} \to \Comod_{E(n)_*E(n)}$ be the induced functor, and $\Phi^\ast$ its adjoint, as described in~\Cref{sec:comodintro}. Hovey and Strickland prove that $L^0_n=\Phi^\ast\Phi_*$ is localization with respect to the hereditary torsion theory consisting of $v_n$-torsion comodules, or equivalently $I_{n+1}$-torsion comodules.

Hovey and Strickland also define $T^0_nM$ to be the subcomodule of $v_n$-torsion elements in a comodule $M$. Both $L^0_n$ and $T^0_n$ are left exact, and hence have right derived functors $L_n^i$ and $T_n^i$ respectively. Our first goal is to relate these functors to the functors $\Gamma_{I_{n+1}}$ and $L_{I_{n+1}}$ constructed previously. 

We start with a preliminary lemma.

\begin{lem}\label{lem:localcohomseq}
For any $M \in \Comod_{BP_*BP}$, there is an exact sequence
\[
0 \to \sh^0(\Gamma_{I_{n+1}}M) \to M \to \sh^0(L_{I_{n+1}}M) \to \sh^{1}(\Gamma_{I_{n+1}}M) \to 0,
\]
and natural isomorphisms $\sh^{i}(L_{I_{n+1}}M) \cong \sh^{i+1}(\Gamma_{I_{n+1}}M)$ for $i>0$.
\end{lem}
\begin{proof}
This follows immediately from the long exact sequence associated to the fiber sequence 
\[
\xymatrix{\Gamma_{I_{n+1}}M \ar[r] & M  \ar[r] & L_{I_{n+1}}M.} 
\]
\end{proof}
We can now relate $\sh^*\Gamma_{I_{n+1}}=\Gamma^\ast_{I_{n+1}}$ and $\sh^*L_{I_{n+1}} = L_{I_{n+1}}^\ast$ to the functors $L_n^\ast$ and $T_n^\ast$ constructed by Hovey and Strickland. 

\begin{lem}\label{lem:torident}
Let $M$ be a $BP_*BP$-comodule thought of as an object of $\Stable_{BP_*BP}$, then there are isomorphisms
	\[
\Gamma^i_{I_{n+1}}(M) \cong T_n^i(M) \quad \text{ and } \quad L^i_{I_{n+1}}(M) \cong L_n^i(M) 
	\]
	for all $i \ge 0$. 
\end{lem}
\begin{proof}
By~\Cref{lem:partialdertorsion} there is an isomorphism
\[
\sh^{i}(\Gamma_{I_{n+1}} M) \cong \colim_k\uExt^i_{BP_*BP}(BP_*/I_{n+1}^k,M). 
\]
Moreover~\Cref{lem:torsionfunctorcomod} and~\eqref{eq:derivedtorcomodules} combine to show that $\colim_k\uExt^i_{BP_*BP}(BP_*/I_{n+1}^k,M)$ is the $i$-th right derived functor of the $I_{n+1}$-power torsion functor considered in~\eqref{eq:itorsion}, which is precisely $T_n^i(M)$. By~\cite[Thm. 3.1]{hs_localcohom} and \Cref{lem:localcohomseq}, this identifies $L^i_{I_{n+1}}M$ with $L_n^iM$ as required.
\end{proof}
The existence of the following spectral sequence was suggested by Hovey at the end of~\cite{hovey_chromatic}.

\begin{prop}
Let $M \in (\Stable_{BP_*BP})_{\le d}$ for some $d \in \Z$, then there exists a convergent spectral sequence $L_{I_{n+1}}^p(\sh_qM) \implies L^{p-q}_{I_{n+1}}M$.
\end{prop}
\begin{proof}
By assumption on $M$, we can construct this spectral sequence for $\omega M$ in $\D_{BP_*BP}$. Since $L_{n}^0$ is left exact and $\Comod_{BP_*BP}$ has enough injective objects, there is a convergent hypercohomology spectral sequence of the form
\[
L_{n}^p(\sh_qM) \implies L_{n}^{p-q}M
\]
associated to the total right derived functor of $L_n^0$, see~\cite[5.7.9 and Cor.~10.5.7]{weibel_homological}. In view of \Cref{lem:torident}, the claim follows. 
\end{proof}

\begin{prop}\label{prop:bptorsionexplicit}
Suppose $M \in \Stable_{BP_*BP}$ and let $n\ge 0$, then 
\[
\Gamma_{I_{n}} M \simeq \Sigma^{-n}BP_*/I_{n}^{\infty} \otimes M.
\]
\end{prop}
\begin{proof}
The functor $\Gamma_{I_{n}}$ is smashing, so it suffices to show the claim for $M=BP$. But $I_{n}$ is strongly invariant, hence \Cref{cor:stronglyinvarianttorsion} applies. 
\end{proof}

\begin{cor}
For all $k\ge 0$ and $n\ge 0$, there is an isomorphism of $BP_*$-comodules
\[
\sh^{*}\Gamma_{I_{n}}(BP_*/I_k) \cong 
\begin{cases}
\Sigma^{-(n-k)}BP_*/(p,\ldots,v_{k-1},v_k^{\infty},\ldots,v_{n-1}^{\infty}) & \text{if } k < n \\
BP_*/I_k & \text{if } k\ge n. 
\end{cases}
\]
\end{cor}
\begin{proof}
By \Cref{prop:bptorsionexplicit}, we have a natural isomorphism
\[
\sh^{*}\Gamma_{I_{n}}(BP_*/I_k) \cong \sh^{*-n}(BP_*/I_{n}^{\infty} \otimes BP_*/I_k).
\]
For fixed $k$, the inductive construction of $BP_*/I_{n}^{\infty}$ gives the result easily.
\end{proof}
Combining this corollary with~\Cref{lem:localcohomseq} and \Cref{lem:torident} immediately recovers the calculation given in~\cite[Thm.~B]{hs_localcohom}.

\subsection{Plethories on $\Stable_{BP_*BP}$}\label{sec:plethora}

Let $E$ be a Landweber exact $BP$-algebra of height $n$, i.e., a $BP$-algebra such that the functor
\[
\xymatrix{E_* \otimes_{BP_*}(-)\colon\Comod_{BP_*BP} \ar[r] & \Mod_{E_*}}
\]
is exact, and $n$ is the largest integer such that $E_*/I_n$ is non-zero. We now consider the duality contexts $(\Stable_{BP_*BP},BP_*/I_{k})$ and $(\Stable_{E_*E},E_*/I_k)$ for $k \le n+1$.  The goal of this subsection and the next is to prove that there are equivalences of stable categories
\[
\xymatrix{\Stable_{BP_*BP}^{I_k-\text{loc}} \ar[r]^-{\sim} & \Stable_{E_*E}^{I_k-\text{loc}}\ar[r]^-{\sim} & \Stable_{E(k-1)_*E(k-1)},}
\] 
thereby establishing a close link between the two local duality contexts. This is closely related to a result of Hovey~\cite[Thm.~C]{hovey_chromatic}. 

We start by giving an alternative description of the functor $L_{I_{n+1}}$. There is an algebra map $\phi:BP_* \to E_*$ and, since $E_*$ is Landweber exact, $\phi$ gives an exact functor $\Phi_*\colon \Thick_{BP_*BP}(BP_*) \to \Thick_{E_*E}(E_*)$ with $\Phi_*M = E_*\otimes_{BP_*}M$. By the universal property of ind-categories, we thus obtain a commutative diagram
\begin{equation}\label{eq:constructionofphi}
\xymatrix{\Thick_{BP_*BP}(BP_*) \ar[r] \ar[d]_{\Phi_*} & \Stable_{BP_*BP} \ar@{-->}@<-0.5ex>[d]_{\Phi_*} \\
\Thick_{E_*E}(E_*) \ar[r] & \Stable_{E_*E}, \ar@{-->}@<-0.5ex>[u]_{\Phi^*}}
\end{equation}
where the induced adjunction is denoted $(\Phi_*,\Phi^*)$ by a mild abuse of notation. In fact, we can give an explicit formula for the right adjoint $\Phi^*$.

\begin{lem}\label{lem:phirightadj}
The right adjoint $\Phi^*$ constructed in \eqref{eq:constructionofphi} is naturally equivalent to the derived functor of $E_*E$-primitives of the extended $BP_*BP$-comodule functor, i.e.,
\[
\xymatrix{\Phi^*(N) \ar[r]^-{\sim} & \Hom_{E_*E}(E_*,N\otimes_{BP_*}BP_*BP)}
\] 
is an equivalence for all $N \in \Stable_{E_*E}$. 
\end{lem}
\begin{proof}
For $N \in \Stable_{E_*E}$, first note that $N\otimes_{BP_*}BP_*BP$ can naturally be viewed as an object of both $\Stable_{E_*E}$ and $\Stable_{BP_*BP}$, so $\Hom_{E_*E}(E_*,N\otimes_{BP_*}BP_*BP) \in \Stable_{BP_*BP}$ and the statement of the lemma is meaningful. Viewing $N$ as an object of $\Stable_{BP_*BP}$ via the $BP_*BP$-comodule structure on $E_*$ in $N \otimes_{E_*} E_* \simeq N$, there exists a natural transformation of colimit preserving functors 
\[
\xymatrix{\Phi^*(N) \otimes_{BP_*} (-) \ar[r]^-{\sim} & \Phi^*(N \otimes_{BP_*} -).}
\]
This is an equivalence of functors because it is an equivalence when evaluated on $BP_*$ which generates $\Stable_{BP_*BP}$. We thus obtain a string of natural equivalences:
\begin{align*}
\Phi^*(N) & \simeq \Hom_{BP_*}(BP_*,\Phi^*(N)) \\
& \simeq \Hom_{BP_*BP}(BP_*,\Phi^*(N)\otimes_{BP_*}BP_*BP) \\
& \simeq \Hom_{BP_*BP}(BP_*,\Phi^*(N\otimes_{BP_*}BP_*BP)) \\
& \simeq \Hom_{E_*E}(E_*,N\otimes_{BP_*}BP_*BP),
\end{align*}
and it is straightforward to check that these maps are $BP_*BP$-comodule morphisms.
\end{proof}

\begin{prop}\label{prop:plethora}
The functor $\Phi^*\colon \Stable_{E_*E} \to \Stable_{BP_*BP}$ is plethystic, in the sense that it satisfies the following properties:
\begin{enumerate}
	\item The counit map $\Phi_*\Phi^*\to \Id$ is an equivalence, so $\Phi^*$ is conservative. 
	\item $\Phi^*$ has a left adjoint $\Phi_*$.
	\item $\Phi^*$ has a right adjoint $\Phi_!$.
\end{enumerate}
In particular, the pairs $(\Phi_* \dashv \Phi^*)$ and $(\Phi^*\dashv \Phi_!)$ are monadic and comonadic, respectively.
\end{prop}
\begin{proof}
As observed above, $\Phi_*$ has a right adjoint $\Phi^*$. In order to show that the counit of the adjunction $\Phi_* \dashv \Phi^*$ is an equivalence, we use the formula for $\Phi^*$ given in \Cref{lem:phirightadj}. Indeed, for any $M \in \Stable_{E_*E}$, the counit can be decomposed into equivalences:
\begin{align*}
\Phi_*\Phi^*M & \simeq E_* \otimes_{BP_*}\Hom_{E_*E}(E_*,M \otimes_{BP_*} BP_*BP) \\
& \simeq \Hom_{E_*E}(E_*,M \otimes_{BP_*} BP_*BP \otimes_{BP_*} E_*) \\
& \simeq \Hom_{E_*E}(E_*,M \otimes_{E_*} E_* \otimes_{BP_*} BP_*BP \otimes_{BP_*} E_*) \\
& \simeq \Hom_{E_*E}(E_*,M \otimes_{E_*} E_*E) \\
& \simeq M.
\end{align*}
This implies that $\Phi_*\Phi^* \simeq \Id$; in particular, $\Phi^*$ is fully faithful and hence conservative. Moreover, there are natural equivalences
\begin{align*}
\Hom_{BP_*BP}(BP_*, \Phi^* \colim_i M_i) & \simeq \Hom_{E_*E}(E_*,\colim_i M_i) \\
& \simeq \colim_i\Hom_{E_*E}(E_*, M_i) \\
& \simeq \colim_i\Hom_{BP_*BP}(BP_*, \Phi^*M_i) \\
& \simeq \Hom_{BP_*BP}(BP_*, \colim_i\Phi^*M_i)
\end{align*} 
for any filtered system $(M_i)_i$ in $\Stable_{BP_*BP}$. Therefore, $\Phi^*$ preserves all colimits, and thus admits a right adjoint $\Phi_!$ by the adjoint functor theorem. 
\end{proof}

\begin{rem}
This result is an algebraic incarnation of the fact that $\Stable_{E_*E}$ corresponds  to the category of ind-coherent sheaves on the open substack $\mathcal{U}(n)$ of the stack of commutative 1-dimensional formal groups $\mathcal{M}_{\mathrm{fg}}$.
\end{rem}

We claim that the composite $\Phi^*\Phi_*$ is equivalent to the localization functor $L_{I_{n+1}}$ corresponding to the local duality context $(\Stable_{BP_*BP},BP_*/I_{n+1})$. Before we can give the proof, we need an auxiliary result.

\begin{lem}\label{lem:bpunit}
For any $0 \le k \le n$ and any $M \in \Stable_{BP_*BP}$, the unit $\eta$ of the adjunction $(\Phi_*,\Phi^*)$ induces a natural equivalence
\[
\xymatrix{v_k^{-1}BP_*/I_k \otimes M \ar[r]^-{\sim} & v_k^{-1}BP_*/I_k \otimes \Phi^*\Phi_*(M).}
\]
\end{lem}
\begin{proof}
Consider the Hopf algebroids $(v_k^{-1}BP_*/I_k,v_k^{-1}BP_*BP/I_k)$ and $(v_k^{-1}E_*/I_k,v_k^{-1}E_*E/I_k)$. The algebra map $\phi\colon BP_* \to E_*$ then induces a commutative diagram
\[
\xymatrix{\Thick(v_k^{-1}BP_*/I_k) \ar[r] \ar[d]_{\Phi_*} & \Stable_{v_k^{-1}BP_*BP/I_k} \ar@{-->}[d]_{\Phi_*}^{\sim} & \Stable_{BP_*BP} \ar[l]_-{\Phi_{BP}} \ar[d]^{\Phi_*} \\ 
\Thick(v_k^{-1}E_*/I_k) \ar[r] & \Stable_{v_k^{-1}E_*E/I_k} & \Stable_{E_*E} \ar[l]^-{\Phi_E}}
\]
in which $\Phi_{BP}$ and $\Phi_{E}$ denote the obvious base change functors. By construction, the middle vertical functor $\Phi_*$ is an equivalence  of stable categories if and only if 
\[
\End(v_k^{-1}BP_*/I_k) \simeq \End(v_k^{-1}E_*/I_k),
\]
which is the content of Hovey's change of rings theorem~\cite[Thm.~6.4]{hovey_morita}.

It thus suffices to show that 
\[
\xymatrix{\Phi_{E}\Phi_*(\eta_M) \simeq \Phi_*\Phi_{BP}(\eta_M)\colon v_k^{-1}E_*/I_k \otimes \Phi_*M \ar[r] & v_k^{-1}E_*/I_k \otimes \Phi_*\Phi^*\Phi_*(M)}
\]
is an equivalence. The counit of the adjunction $(\Phi_*,\Phi^*)$ is a natural equivalence by \Cref{prop:plethora}, so the claim follows. 
\end{proof}

\begin{prop}\label{prop:philn}
There is a natural equivalence $L_{I_{n+1}} \xrightarrow{\sim} \Phi^*\Phi_*$.
\end{prop}
\begin{proof}
Indeed, the unit $\eta$ of the adjunction $(\Phi_*,\Phi^*)$ yields a fiber sequence
\[
\xymatrix{F(M) \ar[r] & M \ar[r]^-{\eta_M} & \Phi^*\Phi_*M}
\]
for any $M \in \Stable_{BP_*BP}$, and we have to show that 
\begin{enumerate}
	\item $\Phi^*\Phi_*M$ is an object of $\Stable_{BP_*BP}^{I_{n+1}-\mathrm{loc}}$, and
	\item $F(M) \in \Stable_{BP_*BP}^{I_{n+1}-\mathrm{tors}}$. 
\end{enumerate}
In order to prove (1), note that 
\[
\Hom(BP_*/I_{n+1},\Phi^*\Phi_*M) \simeq \Hom(E_*/I_{n+1},\Phi_*M) \simeq 0
\] 
as $E_*/I_{n+1}=0$, hence $\Phi^*\Phi_*M \in \Stable_{BP_*BP}^{I_{n+1}-\mathrm{loc}}$. In particular, the counit $\eta$ factors through a natural transformation $L_{I_{n+1}} \to \Phi^*\Phi_*$. 

For $0\le k \le n$, consider the fiber sequences
\[
\xymatrix{F(v_k^{-1}BP_*/I_k^{\infty} \otimes M) \ar[r] & v_k^{-1}BP_*/I_k^{\infty}\otimes M \ar[r]^-{\eta_k} & \Phi^*\Phi_*(v_k^{-1}BP_*/I_k^{\infty}\otimes M).}
\]

On the one hand, since $\Phi^\ast\Phi_\ast$ commutes with colimits, by passing to the localizing subcategory generated by $v_k^{-1}BP_*/I_k$ and using~\Cref{lem:bpunit}, we see that the maps $\eta_k$ are equivalences, hence 
\[
F(v_k^{-1}BP_*/I_k^{\infty}\otimes M) \simeq 0\]
for $0\le k \le n$. On the other hand, $F$ is exact, so there are fiber sequences
\[
\xymatrix{F(BP_*/I_k^{\infty}\otimes M) \ar[r] & F(v_k^{-1}BP_*/I_k^{\infty}\otimes M) \ar[r] &F(BP_*/I_{k+1}^{\infty}\otimes M).}
\]

Inductively, we thus obtain $F(M) \simeq \Sigma^{-(k+1)} F(BP_*/I_{k+1}^{\infty}\otimes M)$ if $k \le n$. But $\Phi^*\Phi_*(BP_*/I_{n+1}^{\infty}\otimes M) \simeq \Phi^*(E_* \otimes BP_*/I_{n+1}^{\infty}\otimes M) \simeq 0$, hence
\[
F(M) \simeq \Sigma^{-(n+1)}F(BP_*/I_{n+1}^{\infty}\otimes M) \simeq \Sigma^{-(n+1)} BP_*/I_{n+1}^{\infty}\otimes M,
\]
which is clearly $I_{n+1}$-torsion, thereby verifying (2). 
\end{proof}
\begin{rem}
	This result should be compared to the situation for $\Comod_{BP_*BP}$, where Hovey and Strickland define $L^0_n$ as the composite $\Phi^\ast \Phi_*$. 
\end{rem}
\begin{cor}
There is a natural equivalence $\Delta_{I_{n+1}} \simeq \Phi^*\Phi_!$.
\end{cor}
\begin{proof}
On the one hand, \Cref{prop:plethora} implies that $\Phi^*\Phi_!$ is right adjoint to $\Phi^*\Phi_*$, as
\[
\Hom(\Phi^*\Phi_*M,N) \simeq \Hom(\Phi_*M,\Phi_!N) \simeq \Hom(M,\Phi^*\Phi_!N)
\]
for all $M,N \in \Stable_{BP_*BP}$. On the other hand, by \Cref{prop:philn} and \Cref{thm:hps}(4), $\Delta_{I_{n+1}}$ is also right adjoint to $L \simeq \Phi^*\Phi_*$, so $\Phi^*\Phi_! \simeq \Delta_{I_{n+1}}$.
\end{proof}

\subsection{Base-change for $E_*E$-comodules}\label{sec:basechange}

Using the results of the previous subsection, we are now ready to prove a general base-change result for $E_*E$-comodules. 

\begin{thm}\label{thm:stablebasechange}
Let $E$ be a Landweber exact $BP$-algebra of height $n$, then the ring map $\phi\colon BP_* \to E_*$ induces a natural equivalence
\[
\xymatrix{\Stable_{BP_*BP}^{I_k-\mathrm{loc}} \ar[r]^-{\sim} & \Stable_{E_*E}^{I_k-\mathrm{loc}}}
\]
for any $k\le n+1$. 
\end{thm}

\begin{proof}
We will first prove the special case $k=n+1$, i.e., that the map $\phi\colon BP_* \to E_*$ induces an equivalence
\[
\xymatrix{\Stable_{BP_*BP}^{I_{n+1}-\text{loc}} \ar[r]^-{\sim} & \Stable_{E_*E}.}
\]
We have already seen that the adjunction $(\Phi_*,\Phi^*)$ restricts to an adjunction
\[
\xymatrix{\Stable_{BP_*BP}^{I_{n+1}-\mathrm{loc}} \ar@<0.5ex>[r]^-{\Phi_*} & \Stable_{E_*E} \ar@<0.5ex>[l]^-{\Phi^*}}
\]
which, by a mild abuse of notation, we also call $(\Phi_*,\Phi^*)$. By \Cref{prop:philn}, $\Phi_*$ sends the generator $L_{I_{n+1}}BP_*$ of $\Stable_{BP_*BP}^{I_{n+1}-\mathrm{loc}}$ to $\Phi_*L_{I_{n+1}}BP_* \simeq \Phi_*BP_* \simeq E_*$, which is the generator of $\Stable_{E_*E}$. Hence we have reduced to showing that the functor $\Phi_*$ is fully faithful, and this again follows from \Cref{prop:philn}:
\[
\Hom_{E_*E}(E_*,E_*) \simeq \Hom_{BP_*BP}(BP_*,L_{I_{n+1}}BP_*) \simeq \Hom_{BP_*BP}(L_{I_{n+1}}BP_*,L_{I_{n+1}}BP_*).
\]
In order to prove the general result, it thus suffices to note that, for $M \in \Stable_{E_*E}$ and $k\le n+1$, we have $\Hom_{E_*E}(E_*/I_k,M)=0$ if and only if $\Hom_{BP_*BP}(BP_*/I_k,\Phi^*M) \simeq 0$.
\end{proof}
\begin{rem}
The case $k=n+1$ of this theorem could also be deduced directly from~\cite[Thm.~C]{hovey_chromatic} using the methods of~\Cref{sec:hovcomparasion}. 
\end{rem}

As an immediate consequence of this theorem, we obtain:

\begin{cor}\label{cor:ethybasechange}
For any $n$ and $k \le n+1$, there is a natural equivalence of stable categories
\[
\xymatrix{\Stable_{E_*E}^{I_{k}-\mathrm{loc}} \ar[r]^-{\sim} & \Stable_{E(k-1)_*E(k-1)}.}
\] 
More generally, this is true for any Landweber exact $BP_*$-algebra of height $k-1$ in place of $E(k-1)_*$.
\end{cor}
\begin{proof}
This follows from \Cref{thm:stablebasechange} and the fact that 
\[
\Stable_{E(k-1)_*E(k-1)}^{I_{k}-\mathrm{loc}} \simeq \Stable_{E(k-1)_*E(k-1)}
\]
via the natural inclusion.
\end{proof}

\begin{rem}
The analogous result for the torsion categories does not hold. To be precise, the adjunction $(\Theta_*,\Theta^*)$ induced by the map of ring spectra $\theta\colon E \to E/I_k$ factors through the torsion category as follows
\[
\xymatrix{\Stable_{E_*E} \ar[r]^{\Theta_*} & \Stable_{E_*E/I_k} \\
\Stable_{E_*E}^{I_k-\text{tors}}, \ar@{-->}[ru]_{\tilde{\Theta}_*} \ar[u]}
\]
where only the left adjoints of the three adjunctions are displayed. Moreover, it is easy to check that $\tilde{\Theta}_*$ sends the generator $E_*/I_k$ of $\Stable_{E_*E}^{I_k-\text{tors}}$ to the generator of $\Stable_{E_*E/I_k}$. However, by comparing the endomorphisms of $E_*/I_k$ in these two categories, one can see that $\tilde{\Theta}_*$ is not fully faithful. 
\end{rem}

We can translate \Cref{thm:stablebasechange} into a statement about the compatibility of the various localization functors. To this end, let $k \le n+1$ and write 
\[
\xymatrix{\Stable_{E_*E} \ar@<0.5ex>[r]^-{L_{I_{k}}^E} & \Stable_{E_*E}^{I_k-\mathrm{loc}} \ar@<0.5ex>[l]^-{\iota_{I_k}^E}}
\]
for the adjunction resulting from the local duality context $(\Stable_{E_*E},E_*/I_k)$, and similarly $(L_{I_{k}}^{BP},\iota_{I_k}^{BP})$ for the $BP$-version. 
 
\begin{cor}\label{cor:compatiblelocfun}
For $k\le n+1$, there is a natural equivalence $\Phi_*L_{I_k}^{BP} \simeq L_{I_{k}}^E\Phi_*$.
\end{cor}
\begin{proof}
By construction, there is a natural equivalence $\iota_{I_k}^{BP}\Phi^* \simeq \Phi^*\iota_{I_k}^{E}$: Indeed, the proof of \Cref{thm:stablebasechange} makes this clear for $k=n+1$, and the general case follows from this by restricting to the appropriate subcategories. Therefore, there exists a string of natural equivalences
\begin{align*}
\Hom_{E_*E}(\Phi_*L_{I_k}^{BP}M,N) & \simeq \Hom_{BP_*BP}(L_{I_k}^{BP}M,\Phi^*N) \\
& \simeq \Hom_{BP_*BP}(M,\iota_{I_k}^{BP}\Phi^*N) \\
& \simeq \Hom_{BP_*BP}(M,\Phi^*\iota_{I_k}^{E}N) \\
& \simeq \Hom_{E_*E}(\Phi_*M,\iota_{I_k}^{E}N) \\
& \simeq \Hom_{E_*E}(L_{I_{k}}^{E}\Phi_*M,N),
\end{align*}
where $M \in \Stable_{E_*E}^{I_k-\mathrm{loc}}$ and $N \in \Stable_{BP_*BP}$. The claim thus follows from the uniqueness of adjoints. 
\end{proof}

This gives us the following change of rings theorem.

\begin{cor}
  Suppose $k \le n$, $M,N \in (\Stable_{E_*E})_{\le d}$ for some $d \in \Z$, and that $N$ satisfies $L_{I_{k+1}}N\simeq N$, then 
  \[
\Ext^{\ast}_{E_*E}(M,N) \cong \Ext^{\ast}_{v_k^{-1}E_*E}(v_k^{-1}M,v_k^{-1}N). 
  \]
In particular, the condition on $N$ holds if $L^0_{I_{k+1}}N \simeq N$ and $L_{I_{n+1}}^iN \simeq 0$ for all $i\ne 0$. 
\end{cor}
\begin{proof}
There are canonical maps $BP_* \xrightarrow{\phi} E \xrightarrow{\phi'} v_k^{-1}E$ of Landweber exact algebras, inducing a diagram
\[
\xymatrix{& \Stable_{BP_*BP} \ar[d]^{L_{I_{k+1}}^{BP}} \ar@/_1pc/[ldd]_{\Phi'_*\Phi_*} \ar@/^1pc/[rdd]^{L_{I_{k+1}}^{E}\Phi_*} \\ 
& \Stable_{BP_*BP}^{I_{k+1}-\mathrm{loc}} \ar[ld]^{\sim} \ar[rd]_{\sim} \\
\Stable_{v_{k}^{-1}E_*E}  & & \Stable_{E_*E}^{I_{k+1}-\mathrm{loc}} \ar@{-->}[ll]_{\rho}^{\sim}}
\]
which commutes by \Cref{cor:ethybasechange} and \Cref{cor:compatiblelocfun}. This gives a natural equivalence $\Phi'_*\Phi_* \simeq \rho L_{I_{k+1}}^{E}\Phi_*$ of functors which, when evaluated on $\Phi^*(M)$ with $M \in \Stable_{E_*E}$, yields
\[
v_{k}^{-1}E_* \otimes_{E_*} M \simeq \Phi'_*(M) \simeq \rho L_{I_{k+1}}^{E}M.\]
The assumption on $N$ then implies that
\[
\begin{split}
\Hom_{E_*E}(M,N) &\simeq \Hom_{E_*E}(L_{I_{k+1}}^EM,L_{I_{k+1}}^EN) \\
&\simeq \Hom_{v_k^{-1}E_*E}(v_k^{-1}M,v_k^{-1}N),
\end{split}
\]
where the last equivalence is induced by $\rho$. The result then follows from~\Cref{cor:derivedcathomcohom}, as $(E_*,E_*E)$ is a Landweber Hopf algebroid. Finally, if $L^0_{I_{k+1}}N \simeq N$ and $L^i_{I_{k+1}}N = 0$ for all $i\ne 0$, then \Cref{lem:localcohomseq} implies that $T^i_{I_{k+1}}(N)=0$ for all $i$, hence $L_{I_{k+1}}N\simeq N$.
\end{proof}

\subsection{The local cohomology spectral sequence}

In~\cite{hs_localcohom} Hovey and Strickland construct spectral sequences that compute the $BP$-homology of $L_nX$ and $C_nX$ for a spectrum $X$, starting from $L_n^iBP_*X$ and $T_n^iBP_*X$ respectively. Our aim in this section is to give an independent construction of the former spectral sequence, as well as a similar spectral sequence computing the $E(n)$-homology of $L_kX$ for $k \le n$. We will show that the spectral sequence that computes $BP_*(L_nX)$ is in fact the $E$-based Adams spectral sequence for $BP \otimes X$. Before we can construct these spectral sequences, we need a preliminary result. 

\begin{lem}\label{lem:bpe2term}
For $n\ge0$ and any $M \in \Stable_{BP_*BP}$, there is a natural equivalence
\[
\xymatrix{L_{I_{n+1}}M \simeq \Hom_{E_*E}(E_*,M \otimes_{BP_*}E_*\otimes_{BP_*}BP_*BP).}
\]
In particular, we have natural isomorphisms
\[
\sh^{s}L_{I_{n+1}}BP_*(X) \cong \Ext_{E_*E}^s(E_*,E_*(BP \otimes X))
\]
for all $s\ge 0$ and any $X \in \Sp$. 
\end{lem}
\begin{proof}
The first part follows immediately from \Cref{prop:philn} and \Cref{lem:phirightadj}. To prove the second claim, observe that $E_*(BP \otimes X) \cong \pi_*(E \otimes BP \otimes X)$ is isomorphic to 
\[
\pi_*(E \otimes X \otimes_{BP} BP \otimes BP) \cong E_*(X) \otimes_{BP_*}  BP_*BP \cong BP_*(X) \otimes_{BP_*} E_* \otimes_{BP_*} BP_*BP,
\]
where we used Landweber exactness of $BP_* \to E_*$ and the flatness of $BP_*BP$ over $BP_*$. Since both $E_*$ and $E_*(BP\otimes X)$ are objects of $(\Stable_{E_*E})_{\le 0}$, $\sh^{s}\Hom_{E_*E}$ does indeed compute the usual comodule Ext by \Cref{cor:derivedcathomcohom}, and the formula for $\sh^{s}L_{I_{n+1}}BP_*(X)$ follows from the first part.
\end{proof}

For an arbitrary spectrum $X$, the $E$-based Adams spectral sequence for $L_n(BP \otimes X)$ thus relates the local cohomology of $BP_*(X)$ to $BP_*(L_nX)$.

\begin{thm}\label{thm:hsspectralsequence}
	 Let $X$ be a spectrum, then there is a natural spectral sequence of $BP_*BP$-comodules with $E_2^{s,t} \cong (L_n^sBP_*X)_t$, converging conditionally and strongly to $BP_{t-s}(L_nX)$. 
		Furthermore, every element in $E_2^{0,\ast}$ that comes from $BP_*X$ is a permanent cycle. 
\end{thm}

\begin{proof}
For any spectrum $Y$, it is known that the $E$-based Adams spectral sequence with $E_2$-term $\Ext^{s,t}_{E_*E}(E_*,E_*Y)$ converges conditionally to $\pi_{t-s}L_nY$~\cite[Thm.~5.3]{hovey_sadofsky}. Combining~\Cref{lem:bpe2term} and \Cref{lem:torident} we see that the $E$-based Adams spectral sequence with $Y = BP \otimes X$ has the form
	\[
E_2^{s,t} \cong (L_n^sBP_*X)_t \implies BP_{t-s}(L_nX). 
	\]
Since there is a horizontal vanishing line on the $E_2$-page, the spectral sequence is also strongly convergent. 

	 A priori, this is just a spectral sequence of $\Z_p$-modules, however we claim that it is in fact a spectral sequence of $BP_*BP$-comodules. To see this, let $\overline E$ be the fiber of the unit map $S^0 \to E$. The canonical $E$-Adams resolution for $BP \otimes X$ has the form~\cite[Lem.~2.2.9]{ravenel_86}
	 \[
\xymatrix{
	BP \otimes X \ar[d] & \overline E \otimes BP \otimes X \ar[l] \ar[d] & \overline E^2 \otimes BP \otimes X \ar[l] \ar[d] & \cdots \ar[l] \\
	E \otimes BP \otimes X & E \otimes \overline E \otimes BP \otimes X & E \otimes \overline E^2 \otimes BP \otimes X.
}
	 \]
	 The associated exact couple that gives rise to the $E$-based Adams spectral sequence is
	 \[
	 \xymatrix{
D_1 \ar[rr] && D_1 \ar[dl] \\
& E_1 \ar[ul],
}
	 \]
	 with $D_1^{s,t} = \pi_{t-s}(\overline E^s \otimes BP \otimes X)$ and $E_1^{s,t} = \pi_{t-s}(E \otimes \overline E^s \otimes BP \otimes X)$. In particular, this is an exact couple in the category of $BP_*BP$-comodules, and it follows that the spectral sequence is a spectral sequence of $BP_*BP$-comodules. 

For the second part of the theorem note that there are maps of spectra $BP \to E$ and $\overline {BP} \to \overline E$, that give maps between the canonical $BP$-Adams resolution for $BP \otimes X$ and the canonical $E$-Adams resolution of $BP \otimes X$. The spectral sequence associated to the former has $E_2^{s,t} = 0$ if $s> 0$ and $E_2^{0,t} \cong BP_tX$ for all $t$. The induced map from this spectral sequence to our spectral sequence gives the claimed result. 
\end{proof}

Therefore, the local cohomology spectral sequence constructed by Hovey and Strickland \cite{hs_localcohom} is in fact a special case of the $E$-based Adams spectral sequence. Similarly, there is a spectral sequence for the torsion functor,
		\[
		E_2^{s,t} \cong (T_n^sBP_*X)_t \implies BP_{t-s}(C_nX),
		\] 
see~\cite[Thm.~5.8]{hs_localcohom}. Furthermore, we obtain generalizations of these spectral sequences  for any Landweber exact $BP_*$-algebra.

\begin{cor}\label{thm:ethyhsspectralsequence}
Let $X$ be a spectrum and $k \le n$.
	\begin{enumerate}
		\item There is a natural, conditionally and strongly convergent, spectral sequence of $E_*E$-comodules with 
		\[
		E_2^{s,t} \cong (L_k^sE_*X)_t \implies E_{t-s}(L_kX). 
		\] 
		Furthermore, every element in $E_2^{0,\ast}$ that comes from $E_*X$ is a permanent cycle. 
		\item There is a natural, conditionally and strongly convergent, spectral sequence of $E_*E$-comodules with 
		\[
		E_2^{s,t} \cong (T_k^sE_*X)_t \implies E_{t-s}(C_kX). 
		\] 
		Furthermore, every element in $E_2^{0,\ast}$ that comes from $E_*X$ is a permanent cycle. 
	\end{enumerate}
\end{cor} 
\begin{proof}
Applying the exact functor $\Phi_*(-)= E_*\otimes_{BP_*} -$ to the spectral sequence of \Cref{thm:hsspectralsequence}(1) gives a spectral sequence with $E_2$-term
\[
E_2^{s,t} \cong \Phi_*(L_k^sBP_*X)_t \cong (L_k^s\Phi_*BP_*X)_t \cong (L_k^sE_*X)_t,
\]
using~\Cref{cor:compatiblelocfun}. This spectral sequence has abutment $\Phi_*BP_{t-s}(L_kX) \cong E_{t-s}(L_kX)$. The second spectral sequence is obtained similarly. 
\end{proof}

\begin{rem}
	It is evident that the same methods produce similar spectral sequences with $E_n$ and $E_k$ in place of $E(n)$ and $E(k)$. 
\end{rem}

\section{Other duality contexts}\label{sec:other}
Due the abstract nature of~\Cref{thm:hps}, it is natural to consider other categories to which our framework applies. Three examples are of particular interest: quasi-coherent sheaves over a scheme, the stable category of equivariant spectra, and the stable motivic homotopy category, all of which we briefly outline in the following section.

\subsection{Local duality for schemes}

Let $X$ be a Noetherian and separated scheme, and let $Z \subseteq X$ be a closed subset with quasi-compact complement $X \backslash Z$. Then, $Z$ is the support $\text{Supp}(X/\mathcal{I})$ for some finite-type quasi-coherent $\mathcal{O}_{X}$-ideal $\mathcal{I}$. Let $\Mod_{\mathcal{O}_X}$ and $\QCoh(X)$ denote the categories of $\mathcal{O}_{X}$-modules and quasi-coherent $\mathcal{O}_{X}$-modules respectively, and let $\D_{\mathcal{O}_X}(X)$ and $\mathcal{D}_{qc}(X)$ be the associated derived categories. 

Recall that a perfect complex is a complex locally isomorphic to a bounded complex of vector bundles. In the case of $\mathcal{D}_{qc}(X)$, a perfect complex is strongly dualizable~\cite[Prop.~4.4]{derivedqcs}. Moreover $\mathcal{D}_{qc}(X)$ is compactly generated by a single perfect complex $\cG$~\cite[Thm. 3.1.1(2)]{bvdb} and is a stable category~\cite{derivedqcs}.  In particular, it has a closed symmetric monoidal structure, where we write $\iHom_{\mathcal{D}_{qc}(X)}(\mathcal{E},\mathcal{F})$ for the internal mapping object. This can differ from the ordinary hom sheaf $\iHom_{\D_{\mathcal{O}_X}(X)}(\mathcal{E},\mathcal{F})$, since the latter need not be quasi-coherent, even if $\mathcal{E}$ and $\mathcal{F}$ are. The relationship between the two is the following: There is a (nonderived) functor, called the coherator, $Q\colon\Mod_{\mathcal{O}_X} \to \QCoh(X)$, right adjoint to the exact inclusion functor $\iota$, which gives rise to a functor $\textbf{R}Q\colon\D_{\mathcal{O}_X}(X) \to \mathcal{D}_{qc}(X)$. We then have an equivalence
\[
\iHom_{\mathcal{D}_{qc}(X)}(\mathcal{E},\mathcal{F}) \simeq \textbf{R}Q\iHom_{\D_{\mathcal{O}_X}(X)}(\mathcal{E},\mathcal{F}),
\]
see~\cite[Sec.~3.7]{derivedqcs}. However, for perfect complexes, the two are equivalent, as the following lemma shows. 
\begin{lem}\label{lem:coherator}
The adjunction 
\[
\xymatrix{\mathcal{D}_{qc}(X) \ar@<0.5ex>[r]^{\iota} & \D_{\mathcal{O}_X}(X) \ar@<0.5ex>[l]^{\textbf{R}Q}}
\]
restricts to an equivalence on perfect $\mathcal{O}_X$-modules. 
\end{lem}
\begin{proof}
Since $\mathcal{O}_X \in \D_{\mathcal{O}_X}(X)$ is quasi-coherent, we have an equivalence $\mathcal{O}_X \simeq \textbf{R}Q(\mathcal{O}_X)$. But $\textbf{R}Q$ is exact, so it must be equivalent to the identity functor on all objects in the thick subcategory generated by $\mathcal{O}_X$. 
\end{proof}
Assuming that $\mathcal{O}_X/\mathcal{I}$ is a compact object of $\mathcal{D}_{\mathcal{O}_X}$, we define the category of quasi-coherent $\mathcal{I}$-torsion sheaves as 
\[
\mathcal{D}_{qc}(X)^{\mathcal{I}-\text{tors}} = \Locid{\mathcal{D}_{qc}(X)}(\mathcal{O}_{X}/\mathcal{I}).
\]
\color{black}
Considering the local duality context $(\mathcal{D}_{qc}(X),\mathcal{O}_X/\mathcal{I})$, our version of abstract local duality implies that, as usual, we have torsion, local, and completion functors as well as their adjoints:

\[
\xymatrix{& \mathcal{D}_{qc}(X)^{\mathcal{I}-\text{loc}} \ar@<0.5ex>[d] \ar@{-->}@/^1.5pc/[ddr] \\
& \mathcal{D}_{qc}(X) \ar@<0.5ex>[u]^{L_\mathcal{I}} \ar@<0.5ex>[ld]^{\Gamma_I} \ar@<0.5ex>[rd]^{\Lambda^\mathcal{I}} \\
\mathcal{D}_{qc}(X)^{\mathcal{I}-\text{tors}} \ar@<0.5ex>[ru] \ar[rr]_{\sim} \ar@{-->}@/^1.5pc/[ruu] & & \mathcal{D}_{qc}(X)^{\mathcal{I}-\text{comp}}. \ar@<0.5ex>[lu]}
\]
\Cref{thm:hps} also implies that there are equivalences
\[
\iHom_{\mathcal{D}_{qc}(X)}(\Gamma_\mathcal{I}\mathcal{E},\mathcal{F}) \simeq \iHom_{\mathcal{D}_{qc}(X)}(\mathcal{E},\Lambda^\mathcal{I}\mathcal{F})
\]
for all $\mathcal{E}, \mathcal{F} \in \mathcal{D}_{qc}(X)$. 
\begin{prop}
	For $\mathcal{E}, \mathcal{F} \in \mathcal{D}_{qc}(X)$, there are equivalences
	\[
\Gamma_{\mathcal{I}} \mathcal{E} \simeq \colim_k \iHom_{\mathcal{D}_{qc}(X)}(\mathcal{O}_X/\mathcal{I}^k,\mathcal{E})
	\]
and
\[
\Lambda^{\mathcal{I}}\mathcal{F} \simeq \lim_k \mathcal{O}_X/\mathcal{I}^k \otimes \mathcal{F}.
\]
\end{prop}
\begin{proof} 
We will use the criterion of \Cref{lem:torsionred} to show that the functor 
\[
\xymatrix{\Gamma_{\mathcal{I}}' = \colim_k \iHom_{\mathcal{D}_{qc}(X)}(\mathcal{O}_X/\mathcal{I}^k,-)\colon \mathcal{D}_{qc}(X) \ar[r] & \mathcal{D}_{qc}(X)^{\mathcal{I}-\text{tors}}}
\]
is equivalent to $\Gamma_{\mathcal{I}}$. The claim about $\Lambda^{\mathcal{I}}$ will then follow formally. Since $\Gamma_{\mathcal{I}}'$ is clearly smashing, it remains to construct a natural equivalence
\begin{equation}\label{eq:qcohtorsionmap}
\xymatrix{\colim_k \iHom_{\mathcal{D}_{qc}(X)}(\mathcal{O}_X/\mathcal{I}^k,\mathcal{O}_X) \otimes  \mathcal{O}_X/\mathcal{I} \ar[r]^-{\phi}_-{\sim} & \mathcal{O}_X/\mathcal{I}.}
\end{equation}
To this end, we first observe that the natural quotient maps $\mathcal{O}_X \to \mathcal{O}_X/\mathcal{I}^k$ form a compatible system, hence induce the desired map $\phi$ of \eqref{eq:qcohtorsionmap}. By \Cref{lem:coherator} and because colimits of quasi-coherent sheaves are quasi-coherent, we can consider $\phi$ as a map in $\D_{\mathcal{O}_X}(X)$; in particular, $\iHom$ can be interpreted as either the internal function object in $\mathcal{D}_{qc}(X)$ or $\D_{\mathcal{O}_X}(X)$. Consequently, $\phi$ is an equivalence if and only if it is an equivalence locally on $X$, so the result reduces to the affine case which was proven in \Cref{prop:Gamma=colimHom}.
\end{proof}
This identifies our version of local duality with that proved in~\cite[Rem.~0.4]{local_cohom_schemes}. We note that a slightly stronger result is also proven in~\cite[Thm.~0.3]{local_cohom_schemes}. Here, they define a torsion functor $\widetilde \Gamma_\mathcal{I}$ for all $\mathcal{E} \in \D_{\mathcal{O}_X}(X)$, and show that local duality also holds more generally between $\widetilde \Gamma_I$ and $\Lambda^\mathcal{I}$.

\subsection{Local duality in equivariant homotopy theory} 
Let $G$ be a finite group, and let $\Sp_G$ be the category of $G$-spectra indexed on a complete universe~\cite{lms_86,mandell_may}.  This is a symmetric monoidal category, see~\cite[Ch.~II.3]{lms_86} or \cite[Ch.~2]{mandell_may}; to be consistent with our usual notation we will write $\otimes$ for the monoidal product, and $\iHom(-,-)$ for the closed structure. The category $\Sp_G$ has a model as a stable category, see~\cite[Def.~5.10]{mnn_eqnilpotence} for example. 

Recall that the dualizable objects in $\Sp_G$ are precisely the retracts of finite $G$-CW spectra~\cite[Thm.~XVI.7.4]{may_96}. Since the unit $S^0$ of $\Sp_G$ is compact, dualizable implies compact; in fact, in $\Sp_G$ dualizable and compact are equivalent, see \cite[Prop.~3.1.3]{joachimi}. Let $H \subseteq G$ be a normal subgroup. By~\cite[Cor.~II.6.3]{lms_86}, $\Sigma^\infty (G/H)_+$ is dualizable, and hence compact in $\Sp_G$, and the collection $\{ \Sigma^\infty (G/H)_+,H \subseteq G\}$ is a generating set of compact  dualizable objects for $\Sp_G$. 

We recall the following definition.
\begin{defn}
	A family $\mathcal{F}$ of subgroups of a discrete group $G$ is a collection of subgroups that is closed under conjugation and taking subgroups. 
\end{defn}

Fix a family $\mathcal{F}$ of subgroups of $G$. There exists a $G$-space $E\mathcal{F}$ characterized by the property 
\[
(E\mathcal{F})^H \simeq 
\begin{cases}
\ast & \text{if } H \in \mathcal{F} \\
\varnothing & \text{otherwise}
\end{cases}
\]
inducing a cofiber sequence of $G$-spectra
\[
\xymatrix{E\mathcal{F}_+ \ar[r] & S^0 \ar[r] & \widetilde E\mathcal{F}_+.}
\]
Let $\mathcal{K} = \{ \Sigma^\infty(G/H)_+ | H \in \mathcal{F} \}$; by the above this is a set of compact objects, and as per usual this is enough for our theory of local duality to apply. We note analogous constructions are considered in~\cite[Sec.~4]{greenlees_axiomatic} and~\cite{mnn_eqnilpotence}. It follows from the proof of~\cite[Cor.~9.4.4]{hps_axiomatic} that $\mathcal{K}$ is closed under tensoring with the generators of $\Sp_G$, and so $\Locid{\Sp_G}(\mathcal{K})=\Loc_{\Sp_G}(\mathcal{K})$. We then consider the local duality context $(\Sp_G,\mathcal{K})$; in particular, define
\[
\Sp_G^{\mathcal{F}-\text{tors}} = \Loc_{\Sp_G}(\mathcal{K}). 
\]
\color{black}
Our framework implies that there are localization and completion functors, and corresponding full subcategories of $\Sp_G$, as indicated in Diagram~\ref{eq:abstractdiagramequivariant}:
\begin{equation}\label{eq:abstractdiagramequivariant}
\xymatrix{& \Sp_G^{\mathcal{F}-\text{loc}} \ar@<0.5ex>[d] \ar@{-->}@/^1.5pc/[ddr] \\
& \Sp_G \ar@<0.5ex>[u]^{L_\mathcal{F}} \ar@<0.5ex>[ld]^{\Gamma_\mathcal{F}} \ar@<0.5ex>[rd]^{\Lambda^{\mathcal{F}}} \\
\Sp_G^{\mathcal{F}-\text{tors}} \ar@<0.5ex>[ru] \ar[rr]_{\sim} \ar@{-->}@/^1.5pc/[ruu] & & \Sp_G^{\mathcal{F}-\text{comp}}. \ar@<0.5ex>[lu]}
\end{equation}

\begin{thm}
\begin{enumerate}
	\item For all $X,Y \in \Sp_G$, there is an adjunction
	\[
\iHom(\Gamma_{\mathcal{F}}(X),Y) \simeq \iHom(X,\Lambda^{\mathcal{F}}(Y)). 
	\] 
	\item The $\mathcal{F}$-torsion, $\mathcal{F}$-localization, and $\mathcal{F}$-completion functors are given by 
	\[
\begin{split}
\Gamma_\mathcal{F}(X) & \simeq E\mathcal{F}_+ \otimes X, \\
L_{\mathcal{F}}(X) & \simeq \widetilde E\mathcal{F}_+ \otimes X, \\
\Lambda^{\mathcal{F}}(X) & \simeq \iHom(E\mathcal{F}_+,X).
\end{split}
\]
\end{enumerate}
\end{thm}
\begin{proof}
	The first part follows immediately from~\Cref{thm:hps}. 

	The second part is essentially proved in~\cite[Sec.~4]{greenlees_axiomatic}. Greenlees shows that the following classes of $G$-spectra are equal:
	\begin{enumerate}
		\item $G$-spectra formed from spheres $G/H_+ \otimes S^n$ with $H \in \mathcal{F}$, and
		\item $G$-spectra $X$ such that $E \mathcal{F}_+ \otimes X \xr{\sim} X$ is an equivalence.
	\end{enumerate}  
	
	 Along with the fact $\widetilde{E}\mathcal{F}$ is $\mathcal{F}$-contractible, we see that $X \to \widetilde{E}\mathcal{F}\otimes X$ is localization away from $\mathcal{K}$, which implies the identification of $L_\mathcal{F}(X)$. The fiber sequence from~\Cref{thm:hps}
	\[
\Gamma_{\mathcal{F}}(X) \to X \to L_{\mathcal{F}}X
	\]
implies that $\Gamma_\mathcal{F}(X) \simeq EF_+ \otimes X$, and the identification of $\Lambda^{\mathcal{F}}(X)$ follows from the adjunction in Part (1). 
\end{proof}

\subsection{Local duality in motivic homotopy theory}
Let $S$ be a Noetherian scheme of finite Krull dimension, and $\Sp(S)$ be the stable motivic homotopy category constructed by Morel and Voevodsky~\cite{morel_voevodsky}. This construction was originally obtained using model categorical methods, however it is known by work of Robalo~\cite[Cor.~1.2]{robalo} that the $\infty$-category underlying $\Sp(S)$ (which we abusively denote by the same symbol) is a stable category.  

Now assume that $S = \Spec(k)$, for $k$ any field, and let $\mathbf{Sm}/k$ denote the category of smooth schemes of finite type over $k$. In contrast to the equivariant case, the dualizable and compact objects are no longer equivalent in $\Sp(k)$, at least not without restrictions on the field $k$. In particular, they are equivalent if $k$ has characteristic 0, see \cite[Prop.~5.2.9]{joachimi}. In general, the collection $\mathcal{G} = \{\Sigma^{2n,n}\Sigma^\infty U_+ | U \in \mathbf{Sm}/k,n\in \Z \}$ is a set of compact generators for $\Sp(k)$~\cite[Thm.~9.2]{dugger_isaksen}. 
\begin{rem}\label{rem:realization}
	When $k \subseteq \mathbb{C}$ there is a functor
	\[
R_k:\Sp(k) \to \Sp,
	\]
	see~\cite[Sec.~3.3]{morel_voevodsky} for the case $k = \mathbb{C}$ or~\cite[Ch.~4]{joachimi}. The functor $R_k$ commutes with colimits (see~\cite{isaksen_shkembi}), and can be promoted to a symmetric monoidal functor of stable categories~\cite[Ex.~1.3]{robalo}. Let $X = \Sigma^{2n,n}\Sigma^\infty U_+$ be a compact generator of $\Sp(k)$, then by~\cite[Prop.~5.2.10]{joachimi} $R_{k}(X)$ is a finite spectrum in $\Sp$. 
	\end{rem}

Assuming $k \subseteq \mathbb{C}$, there exist algebraic Morava $K$-theory spectra $AK(n)$~\cite{borghesi_k_theory}. In fact, there are two models described, based on two variants of constructing connective Morava $K$-theory $k(n)$ classically, namely killing generators in $\pi_*MU$, or beginning with $H\mathbb{F}_p$ and killing secondary cohomology operations to construct $k(n)$ as the inverse limit of its Postnikov tower. We will be interested in the first approach. Briefly, for $k \subseteq \mathbb{C}$ there is an algebraic cobordism spectrum $MGL$~\cite[Sec.~6.3]{voevodsky}, which is the analogue of $MU$ in classical stable homotopy theory. There are elements $a_i \in MGL_{2i,i},i \ge 1$, whose images under $R_k$ are the usual classes $a_i^{\text{top}} \in MU_{2i}$. The motivic Brown--Peterson spectrum is defined as 
\[
ABP = MGL_{(p)}/(a_i| i \ne p^j-1).
\]
With $v_0=p$ and $v_i = a_{p^i-1}$ for $i \ge 1$, define
\[
Ak(n) = ABP/(v_0,\ldots,v_{n-1},v_{n+1},\ldots) \text{ and } AK(n) = v_n^{-1}Ak(n).
\]

In the ordinary stable homotopy category we studied the local duality context $(\Sp,F(n))$, for $F(n)$ a finite type $n$-spectrum. Analogously, say that a (compact) motivic spectrum has type $n$ if $AK(n-1)_{\ast,\ast}(X) = 0$ and $AK(n)_{\ast,\ast}(X) \ne 0$. When $k=\mathbb{C}$ itself there exists a (compact) motivic spectrum of type $n$~\cite[Thm.~8.5.12]{joachimi}, which we denote by $AM(n)$. By~\cite[Sec.~8.6]{joachimi}, $R_{\mathbb{C}}(AM(n))$ is the type $n$ spectrum defined in~\cite[Thm.~C.3.2]{orangebook}.

Let $\mathcal{K} = AM(n)$. Consider now the local duality context $(\Sp(\mathbb{C}),\mathcal{K})$, and define\footnote{Here the use of $n-1$ instead of $n$ is for consistency with the conventions in the stable homotopy category.}
\[
{\Sp(\mathbb{C})}_{n-1}^{f-\text{tors}} = \Locid{\Sp(\mathbb{C})}(\cK). 
\]
\color{black}
Applying \Cref{thm:hps} to this context yields a diagram
\begin{equation}\label{eq:globalspectradiagrammotivic}
\xymatrix{& {\Sp(\mathbb{C})}^{f-\text{loc}}_{n-1} \ar@<0.5ex>[d] \ar@{-->}@/^1.5pc/[ddr] \\
& \Sp(\mathbb{C}) \ar@<0.5ex>[u]^{L_{n-1}^f} \ar@<0.5ex>[ld]^{\Gamma_{n-1}^f} \ar@<0.5ex>[rd]^{\Lambda_{n-1}^{f}} \\
{\Sp(\mathbb{C})}_{n-1}^{f-\text{tors}} \ar@<0.5ex>[ru] \ar[rr]_-{\sim} \ar@{-->}@/^1.5pc/[ruu] & & {\Sp(\mathbb{C})}^{f-\text{comp}}_{n-1}, \ar@<0.5ex>[lu]}
\end{equation}
as well as an equivalence
\[
\iHom(\Gamma_{n-1}^fX,Y) \simeq \iHom(X,\Lambda_{n-1}^fY)
\]
for all $X,Y \in \Sp(\mathbb{C})$. The functor $L_{n-1}^f$ is finite localization away from a type $n$ complex, in analogy with the functor denoted $L_{n-1}^f$ in the ordinary stable homotopy category. We note the following.
\begin{lem}
		The realization functor $R_\mathbb{C}\colon\Sp(\mathbb{C}) \to \Sp$ restricts to a functor
		\[
R_{\mathbb{C}}\colon{\Sp(\mathbb{C})}_{n-1}^{f-\text{tors}} \to \mathcal{C}_{n-1}^f.
		\]
\end{lem}
\begin{proof}
	The category ${\Sp(\mathbb{C})}_{n-1}^{f-\text{tors}}$ is compactly generated by $\{ AM(n) \otimes X | X \in \mathcal{G} \}$. 	As noted previously the realization functor takes $AM(n)$ to a type $n$ spectrum in the category of spectra, which we denote by $M(n)$. Since the construction of $\mathcal{C}_{n-1}^f$ is independent of the choice of type $n$ spectrum, it follows that $R_\mathbb{C}$ takes $AM(n)$ to a generator of $\mathcal{C}_{n-1}^f$. Since $R_{\mathbb{C}}$ is symmetric monoidal we use~\Cref{rem:realization} to see that $R(AM(n) \otimes X) = M(n) \otimes Y$ where $X \in \mathcal{G}$ and $Y$ is a finite spectrum. Since $M(n) \otimes Y \in \mathcal{C}_{n-1}^f$, $R_{\mathbb{C}}$ sends all the generators of ${\Sp(\mathbb{C})}_{n-1}^{f-\text{tors}}$ to objects of $\mathcal{C}_{n-1}^f$. Since $R_{\mathbb{C}}$ commutes with colimits it restricts to the torsion categories as claimed. 
\end{proof}

We do not know if it is true that $R_\mathbb{C}$ restricts to the local and complete categories, and nor are we able to construct explicit formulae for these functors. It would be interesting to know if there are analogues of~\Cref{thm:globalspectralocduality} for a suitable tower of motivic generalized Moore spectra, so that explicit descriptions could be given.

\section{Comparison of various local duality contexts} 
In \Cref{fig:ultimate} we show the various local duality contexts considered in this paper and the relationships between them. 
\begin{figure}[h!]
\[\resizebox{\textwidth}{!}{\xymatrix{
& \Sp_{E(m-1)}^f \ar@<0.5ex>[d] \ar@{-->}@/^1.5pc/[ddr] & & & \Stable_{BP_*BP}^{I_m-\mathrm{loc}} \ar@<0.5ex>[d] \ar@{-->}@/^1.5pc/[ddr] \\
& \Sp_{(p)} \ar@<0.5ex>[u]^{L_{m-1}^f} \ar@<0.5ex>[ld]^{C_{m-1}^f} \ar@<0.5ex>[rd]^{L_{F(m)}} & \ar@{~>}[r]^{BP_*} & & \Stable_{BP_*BP} \ar@<0.5ex>[u]^{L_{m}} \ar@<0.5ex>[ld]^{\Gamma_{I_m}} \ar@<0.5ex>[rd]^{\Lambda^{I_m}} \\
\mathrm{Loc}(F(m)) \ar@<0.5ex>[ru] \ar[rr]^{\sim} \ar@{-->}@/^1.5pc/[ruu] &  \ar@<-1ex>@{~>}[dd]_{L_n}  & \Sp_{F(m)} \ar@<0.5ex>[lu] &  \Stable_{BP_*BP}^{I_m-\mathrm{tors}} \ar@<0.5ex>[ru] \ar[rr]^{\sim} \ar@{-->}@/^1.5pc/[ruu] & \ar@<-1ex>@{~>}[dd]_{-\otimes_{BP_*} E(n)_*} & \Stable_{BP_*BP}^{I_m-\mathrm{comp}} \ar@<0.5ex>[lu] \\ \\
& \Sp_{E(m-1)} \ar@<0.5ex>[d] \ar@{-->}@/^1.5pc/[ddr] \ar@<-1ex>@{~>}[uu]_{\text{incl}} & & & \Stable_{E(n)_*E(n)}^{I_m-\mathrm{loc}} \ar@<0.5ex>[d] \ar@{-->}@/^1.5pc/[ddr] \ar@<-1ex>@{~>}[uu]_{\Phi_n^*} \\
& \Sp_{E(n)} \ar@<0.5ex>[u]^{L_{m-1}} \ar@<0.5ex>[ld]^{M_m} \ar@<0.5ex>[rd]^{L_{E(m,n)}} & \ar@{~>}[r]^{E(n)_*} & & \Stable_{E(n)_*E(n)} \ar@<0.5ex>[u]^{L_{m}} \ar@<0.5ex>[ld]^{\Gamma_{I_m}} \ar@<0.5ex>[rd]^{\Lambda^{I_m}} \\
\mathcal{M}_m \ar@<0.5ex>[ru] \ar[rr]^{\sim} \ar@{-->}@/^1.5pc/[ruu] &  \ar@<-1ex>@{~>}[dd]_{-\otimes E(n)}  & \Sp_{E(m,n)} \ar@<0.5ex>[lu] &  \Stable_{E(n)_*E(n)}^{I_m-\mathrm{tors}} \ar@<0.5ex>[ru] \ar[rr]^{\sim} \ar@{-->}@/^1.5pc/[ruu] & \ar@<-1ex>@{~>}[dd]_{\epsilon_*} & \Stable_{E(n)_*E(n)}^{I_m-\mathrm{comp}} \ar@<0.5ex>[lu] \\ \\
& \Mod_{E(n)}^{I_m-\mathrm{loc}} \ar@<0.5ex>[d] \ar@{-->}@/^1.5pc/[ddr] \ar@<-1ex>@{~>}[uu]_{\text{incl}} & & & \cD_{E(n)_*}^{I_m-\mathrm{loc}} \ar@<0.5ex>[d] \ar@{-->}@/^1.5pc/[ddr] \ar@<-1ex>@{~>}[uu]_{- \otimes_{E(n)_*} E(n)_*E(n)} \\
& \Mod_{E(n)} \ar@<0.5ex>[u]^{L_{I_m}} \ar@<0.5ex>[ld]^{\Gamma_{I_m}} \ar@<0.5ex>[rd]^{\Lambda^{I_m}} & \ar@{~>}[r]^{\pi_*} & & \cD_{E(n)_*} \ar@<0.5ex>[u]^{L_{I_m}} \ar@<0.5ex>[ld]^{\Gamma_{I_m}} \ar@<0.5ex>[rd]^{\Lambda^{I_m}} \\
\Mod_{E(n)}^{I_m-\mathrm{tors}} \ar@<0.5ex>[ru] \ar[rr]^{\sim} \ar@{-->}@/^1.5pc/[ruu] & &\Mod_{E(n)}^{I_m-\mathrm{comp}} \ar@<0.5ex>[lu] &  \cD_{E(n)_*}^{I_m-\mathrm{tors}} \ar@<0.5ex>[ru] \ar[rr]^{\sim} \ar@{-->}@/^1.5pc/[ruu] & & \cD_{E(n)_*}^{I_m-\mathrm{comp}} \ar@<0.5ex>[lu] 
}}\]
\caption{Topological and algebraic local duality contexts.}

\label{fig:ultimate}
\end{figure}
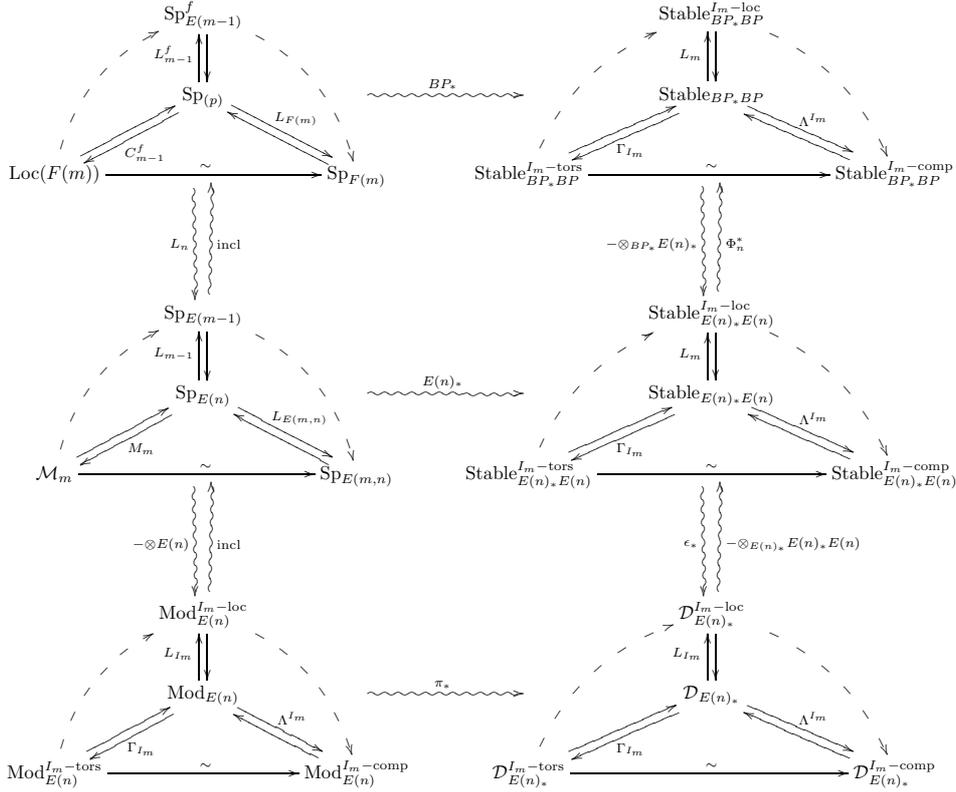
\newpage

\bibliography{duality}\bibliographystyle{alpha}

\end{document}